\documentclass{article}

\RequirePackage{amsmath,amsthm,amsfonts,amssymb,mathtools}
\RequirePackage{mathrsfs}
\RequirePackage{bbm}
\RequirePackage[numbers]{natbib}

\RequirePackage[colorlinks,citecolor=blue,urlcolor=blue,linkcolor =blue,hypertexnames=false]{hyperref}
\RequirePackage{cleveref} 
\RequirePackage{graphicx}
\usepackage{comment}

\RequirePackage{caption}
\RequirePackage{subcaption}
\usepackage{float}

\RequirePackage[numbers]{natbib}
\RequirePackage{etoolbox}

\usepackage[shortlabels]{enumitem}
\setenumerate{label=(\roman*),itemsep=3pt,topsep=3pt}

\numberwithin{equation}{section}

\theoremstyle{plain}

\newtheorem{theorem}{Theorem}[section]
\newtheorem{lemma}[theorem]{Lemma}
\newtheorem{corollary}[theorem]{Corollary}
\newtheorem{proposition}[theorem]{Proposition}

\theoremstyle{definition}

\newtheorem{remark}[theorem]{Remark}
\newtheorem{example}[theorem]{Example}

\newtheorem*{assumption*}{\assumptionnumber}\providecommand{\assumptionnumber}{}
\makeatletter

\newenvironment{assumptionRegTh}[2]{%
  \stepcounter{theorem}
  \renewcommand{\assumptionnumber}{Assumption (\arabic{section}.\arabic{theorem}, Regularity, $#1$, $#2$)}%
\begin{assumption*}
\protected@edef\@currentlabel{\arabic{section}.\arabic{theorem}}
 }{%
  \end{assumption*}
}

\newenvironment{assumptionInitial}[2]{%
  \stepcounter{theorem}
  \renewcommand{\assumptionnumber}{Assumption (\arabic{section}.\arabic{theorem}, Initial, $#1$, $#2$)}%
\begin{assumption*}
\protected@edef\@currentlabel{\arabic{section}.\arabic{theorem}}
 }{%
  \end{assumption*}
}

\newenvironment{assumptionKernel}[1]{%
  \stepcounter{theorem}
  \renewcommand{\assumptionnumber}{Assumption (\arabic{section}.\arabic{theorem}, Kernel, $#1$)}%
\begin{assumption*}
\protected@edef\@currentlabel{\arabic{section}.\arabic{theorem}}
 }{%
  \end{assumption*}
}

\newenvironment{assumptionApprox}[2]{%
  \stepcounter{theorem}
  \renewcommand{\assumptionnumber}{Assumption (A.\arabic{theorem}, Approx, $#1$, $#2$)}%
\begin{assumption*}
\protected@edef\@currentlabel{A.\arabic{theorem}}
 }{%
  \end{assumption*}
}

\makeatother

\theoremstyle{remark}

\newcounter{casecount}
\newtheorem{case}[casecount]{Case}
\AtBeginEnvironment{proof}{\setcounter{casecount}{0}}

\newcounter{stepcount}
\newtheorem{step}[stepcount]{Step}
\AtBeginEnvironment{proof}{\setcounter{stepcount}{0}}

\newcommand\restr[2]{{
  \left.\kern-\nulldelimiterspace 
  #1 
  \vphantom{\big|} 
  \right|_{#2} 
  }}

\newcommand{\fv}{\omega} 
\newcommand{\locu}{\mathcal{L}_{\delta}^{C}} 
\newcommand{\locv}{\mathcal{L}_{\delta}^{S}} 
\newcommand{\asskernel}[1][$K$]
    {\hyperref[asskernel]{\normalfont \textbf{(\ref{asskernel}, Kernel, {\color{black} #1})}}}
\newcommand{\assth}
    {\hyperref[assumption:regularity_wave_speed]{\normalfont \textbf{(\ref{assumption:regularity_wave_speed}, Regularity, {\color{black} $\vartheta$}, {\color{black} $\alpha$})}}}

\newcommand{\assini}
    {\hyperref[assumption:initial_conditions]{\normalfont \textbf{(\ref{assumption:initial_conditions}, Initial, {\color{black} $u_0$}, {\color{black} $v_0$})}}}

\newcommand{\asskernelapprox}[2][$w^{\delta}$, $z$]
    {\hyperref[asskernel_approx]{\normalfont \textbf{(\ref{asskernel_approx}, Approx,  {\color{black} #1}, {\color{black} #2})}}}

\begin{document}
\author{Eric Ziebell\\ Humboldt-Universit\"at zu Berlin, Germany\\Email: ziebelle@hu-berlin.de}
\date{}
\title{Non-parametric estimation for the stochastic wave equation}
\maketitle

\begin{abstract}
  The spatially dependent wave speed of a stochastic wave equation driven by space-time white noise is estimated using the local observation scheme. Given a fixed time horizon, we prove asymptotic normality for an augmented maximum likelihood estimator as the resolution level of the observations tends to zero. We show that the expectation and variance of the observed Fisher information are intrinsically related to the kinetic energy within an associated deterministic wave equation and prove an asymptotic equipartition of energy principle using the notion of asymptotic Riemann-Lebesgue operators.
  \medskip



\end{abstract}

\section{Introduction}
\subsection{Motivation}
Let $A_{\vartheta} z \coloneqq \mathrm{div}(\vartheta \nabla z )$, $\vartheta:\Lambda \rightarrow (0, \infty)$, be a weighted Laplacian in divergence form satisfying Dirichlet boundary conditions on an open and bounded domain $\Lambda \subset \mathbb{R}^{d}$ with $C^{2}$-boundary. We consider the stochastic wave equation
\begin{equation}
\label{eq:SPDE_introduction}
\partial_{tt}^{2}u(t)=A_{\vartheta}u(t) +  \dot{W}(t), \quad t \in [0,T],
\end{equation} 
driven by Gaussian space-time white noise $(\dot{W}(t),t \in [0,T])$. This work is devoted to the non-parametric estimation of the spatially varying wave speed $\vartheta:\Lambda \rightarrow (0,\infty)$ based on  local observations. The stochastic wave equation is of interest both from a theoretical and applied point of view, see e.g. \cite{dalangStochasticWaveEquation2009, conusNonLinearStochasticWave2008, charousDynamicallyOrthogonalDifferential2021, dorflerWavePhenomenaMathematical2023, galkaDatadrivenModelGeneration2008} and the references therein.
\par This work is related to \citet*{liuEstimatingSpeedDamping2008} as well as \citet*{delgado-vencesStatisticalInferenceStochastic2023}, who inferred the parametric wave speed, i.e. assuming $\vartheta$ to be constant, of a one-dimensional stochastic wave equation based on spectral observations up to a finite time horizon $T>0$. The spectral approach is, however, restricted to operators $A_{\vartheta}$ whose leading order eigenfunctions are independent of $\vartheta$ and is not suitable for estimating a spatially varying wave speed.
The constant wave speed of a stochastic wave equation was also inferred using the following two different large-time observation schemes.
Using the ergodicity of the solution, \citet*{janakParameterEstimationStochastic2020, janakParameterEstimationStochastic2021} estimated the wave speed and damping of a strongly damped stochastic wave equation based on an observation window but required the asymptotic regime $T \rightarrow \infty$. Furthermore, \citet*{markussenLikelihoodInferenceDiscretely2003} proposed an approximate likelihood approach and analysed asymptotic properties of the associated maximum likelihood estimator based on discrete observations at $t=\Delta, 2\Delta, \dots, n\Delta$, $\Delta>0$, as $n \rightarrow \infty$. Other works on statistics for the stochastic wave equation, e.g. \citet*{shevchenkoQuadraticVariationsFractionalwhite2023} and \citet{shevchenkoGeneralizedKvariationsHurst2020}, concern estimators of the unknown Hurst coefficient of a fractional stochastic wave equation using Malliavin calculus. \par The local observation scheme was introduced in the parabolic case of the stochastic heat equation by \citet*{altmeyerNonparametricEstimationLinear2021} and was used to construct rate optimal estimators of a spatially varying diffusivity. A first extension to semilinear parabolic SPDEs was achieved by \citet*{altmeyerParameterEstimationSemilinear2023}. Based on multiple local measurements, \citet*{altmeyerOptimalParameterEstimation2022} obtained optimal rates for estimating lower-order coefficients. Multiple local measurements were then used in \citet*{reissChangePointEstimation2023} to detect a change point within the diffusivity and by \citet{strauchNonparametricVelocityEstimation2024} to estimate the non-parametric velocity of a stochastic heat equation. Based on real data, the local observation scheme identified parameters in the Meinhardt model for cell repolarisation, see \citet*{altmeyerParameterEstimationSPDE2022}. \citet*{janakParameterEstimationStochastic2023} show that local measurements can also be used to estimate the diffusivity in the case of multiplicative noise. \citet{aiharaIdentificationDiscontinuousParameter1994,aiharaMaximumLikelihoodEstimate1994} took a first step towards the estimation of a spatially varying parameter of an elastic operator in a damped stochastic hyperbolic system under noisy partial observations by considering a Kalman filtering problem and using the methods of sieves. Consistency, however, is only achieved for global observations in the large time regime $T \rightarrow \infty$. 

To our knowledge, the statistical inference for a spatially varying wave speed $\vartheta$ has not yet been explored for a fixed time horizon $T<\infty$. In the parabolic case, to control the asymptotic behaviour of the observed Fisher information, \citet{altmeyerNonparametricEstimationLinear2021} use the fact that the heat semigroup naturally decays in time. It is a priori unclear if the local measurement approach is also applicable to the undamped stochastic wave equation, which is inherently energy-preserving.
Intriguingly, we will show that the local observation scheme is a suitable tool for studying stochastic hyperbolic equations by relating the observed Fisher information to the energetic behaviour of an associated deterministic system.

More specifically, given a fixed finite time horizon $T > 0$, we will estimate the spatially dependent wave speed $\vartheta$ at $x_0 \in \Lambda$ using the augmented maximum likelihood approach (MLE) based on local observations $\langle u(t), K_{\delta, x_0} \rangle$, where the solution process $u$ is tested against a kernel localised around $x_0$. The augmented MLE is consistent with rate $\delta$ whilst satisfying a central limit theorem. The asymptotic variance of the limiting normal distribution depends on the fixed time horizon $T$ through $T^{-2}$, which is specific to the hyperbolic case and different from the parabolic dependence $T^{-1}$. This discrepancy highlights the difference between the energetically dissipative heat equation and the energetically stable wave equation. A similar influence of the dissipative behaviour of the underlying equation is also observed in the MLE's rate of convergence for the ordinary Ornstein-Uhlenbeck process or the harmonic oscillator. Indeed, the rate is $\sqrt{T}$ in the ergodic and $T$ in the energetically stable case; see \citet*[Proposition 3.46]{kutoyantsStatisticalInferenceErgodic2004} and \citet*{linUndampedHarmonicOscillator2011}.

In the hyperbolic case of the stochastic wave equation, the covariance structure of the local measurements involves rescaled operator cosine and sine families, which generalise the concept of semigroups to the hyperbolic case. Thus, we can analyse the asymptotic behaviour of the augmented MLE, which is linked through the observed Fisher information to the energetic behaviour within a scaled deterministic wave equation, by showing a scaled version of the equipartition of energy principle. The equipartition of energy was studied for abstract hyperbolic Cauchy problems by several works, including \cite{goldsteinAbstractEquipartitionEnergy1979, goldsteinAsymptoticEquipartitionEnergy1976, goldsteinAsymptoticPropertySolutions1969, goldsteinAsymptoticPropertySolutions1970, goldsteinAutocorrelationEquipartitionEnergy1982, picardAsymptoticEquipartitionEnergy1987, shinbrotAsymptoticBehaviorSolutions1968}, and relies on the concept of Riemann-Lebesgue operators. In particular, we show that the scaled weighted Laplace operator $A_{\vartheta(\delta \cdot)}$ behaves asymptotically as $\delta \rightarrow 0$ like a Riemann-Lebesgue operator. In fact, our analysis of the observed Fisher information is based on the weak operator topology and not as in the parabolic case on norm bounds of the heat semigroup. This is accomplished using functional calculus and convergence results for the resolution of the identity associated with the scaled weighted Laplace operator on the growing spatial domain $\delta^{-1}(\Lambda-x_0)$. 

The probabilistic backbone of the asymptotic results is a scaling limit for the observation processes as $\delta \rightarrow 0$. In the parametric case, this scaling limit can even be achieved in finite time due to the finite propagation speed of the wave equation.
\subsection{An overview of the main results}
We briefly discuss the main results of this work. The stochastic partial differential equation and all objects related to its solution are rigorously introduced in \Cref{section:model}. \par In the local observation scheme, for a given resolution $\delta>0$, we assume to observe continuous time processes of the form
\begin{equation}
\label{eq:local_measurements_overview}
u_{\delta}(t)= \langle u(t), K_{\delta, x_0}\rangle_{L^{2}(\Lambda)}, \quad u_{\delta}^{\Delta}(t) = \langle u(t), \Delta K_{\delta, x_0}\rangle_{L^{2}(\Lambda)}, \quad t \in [0,T],
\end{equation}
where $K_{\delta, x_0}(x)=\delta^{-d/2}K(\delta^{-1}(x-x_0))$ is a kernel localised around $x_0 \in \Lambda$. Abbreviate by $v_{\delta}(t)\coloneqq \dot{u}_{\delta}(t)$ for $t \in [0,T]$ local measurements of the velocity field, which are retrieved as the time derivative of the continuously observed differentiable process $u_{\delta}(t)$. The formal relationship between the velocity field $v$ and the amplitude $u$ is analysed in \Cref{section:well-posedness-main}. \par In analogy to \citet{altmeyerNonparametricEstimationLinear2021}, we consider the augmented maximum likelihood estimator
\begin{equation}
\label{eq:augmented_MLE_overview}
\hat{\vartheta}_\delta(x_0) \coloneqq \frac{\int_{0}^{T}u_\delta^{\Delta}(t)\mathrm{d}v_{\delta}(t)}{\int_{0}^{T}(u_{\delta}^{\Delta}(t))^{2}\mathrm{d}t}, \quad \delta > 0,
\end{equation}
This estimator can be motivated using the Girsanov theorem or a least squares criterion, c.f. \citet[Section 4.1]{altmeyerNonparametricEstimationLinear2021}. In \Cref{result:asymptotic_normality} of \Cref{section:asymptotics_augmented_MLE_main}, we prove, under smoothness assumptions on the kernel $K$, the initial conditions and the wave speed $\vartheta$, that the augmented MLE satisfies for $\delta \rightarrow 0$:
\begin{equation}
\label{eq:asymptotic_normality_aMLE_overview}
\delta^{-1}(\hat{\vartheta}_\delta(x_0) - \vartheta(x_0)) \xrightarrow{d} \mathcal{N}\left(  \frac{\langle \nabla K, \nabla \beta^{(0)}\rangle_{L^{2}(\mathbb{R}^{d})}}{\Vert \nabla K \Vert^{2}_{L^{2}(\mathbb{R}^{d})}}, \frac{ 4 \vartheta(x_0)\Vert K \Vert^{2}_{L^{2}(\mathbb{R}^{d})}}{T^{2} \Vert \nabla K \Vert^{2}_{L^{2}(\mathbb{R}^{d})}} \right), 
\end{equation}
with $\beta^{(0)}$ given by \eqref{eq:beta_delta}. 
In \Cref{result:confidence_interval}, we provide reasonable assumptions for the kernel $K$ such that the asymptotic bias of the augmented MLE vanishes, allowing us to construct asymptotic confidence intervals. 
Crucial in proving the asymptotic normality is the asymptotic behaviour of the observed Fisher information $I_{\delta}=\int_{0}^{T}u_{\delta}^{\Delta}(t)^{2}\mathrm{d}t$ involved in the error decomposition \eqref{eq:error_decomposition}:
\begin{equation*}
\hat{\vartheta}_{\delta}(x_0) - \vartheta(x_0) = \Vert K \Vert_{L^{2}(\mathbb{R}^{d})} {I}_{\delta}^{-1}{M}_{\delta}+ I_{\delta}^{-1}R_{\delta},
\end{equation*}
where $M_{\delta}$ is a martingale term with quadratic variation $I_{\delta}$, and $R_{\delta}$ characterises the remaining bias. The quadratic variation of the martingale $M_{\delta}/({\mathbb{E}[I_{\delta}]})^{1/2}$ is given by $I_{\delta}/\mathbb{E}[I_{\delta}]$. For the asymptotic normality \eqref{eq:asymptotic_normality_aMLE_overview}, we require the convergence $I_{\delta}/\mathbb{E}[I_{\delta}] \xrightarrow{\mathbb{P}} 1$ as $\delta \rightarrow 0$, as the martingale central limit theorem then implies $M_{\delta}/({\mathbb{E}[I_{\delta}]})^{1/2} \xrightarrow{d} \mathcal{N}(0,1)$. Therefore, we will analyse the expectation and variance of the observed Fisher information $I_\delta$ and the bias $R_\delta$, which is due to the variability of the wave speed $\vartheta(\cdot)$ in space.

In the case of zero initial conditions, the observed Fisher information satisfies
\begin{equation*}
\mathbb{E}[\delta^{2}I_{\delta}]=\int_{0}^{T} \int_{0}^{t}\Vert S_{\vartheta, \delta}(\delta^{-1}r) \Delta K \Vert^{2}_{L^{2}(\Lambda_{\delta})} \mathrm{d}r \mathrm{d}t,
\end{equation*}
where $(S_{\vartheta,\delta}(t), t \in \mathbb{R})$ is the rescaled operator sine function associated with the operator $A_{\vartheta,\delta}=A_{\vartheta(\delta \cdot)}$, which converges asymptotically to $\vartheta(0)\Delta$ on $H^{2}(\mathbb{R}^{d})$. Using functional calculus, we can represent the operator sine function through
\begin{equation*}
S_{\vartheta, \delta}(\delta^{-1}r) = (-A_{\vartheta,\delta})^{-1/2} \sin(\delta^{-1}r(-A_{\vartheta,\delta})^{1/2}), \quad r \geq 0.
\end{equation*}
Precise definitions of the operator cosine and sine are provided in \Cref{section:operator_cosine_and_sine}, and the associated scaling properties are introduced in \Cref{result:rescaling_trigonometric_operator_families} of \Cref{section:scaling_limit}. We further show fixed-time scaling limits for both the deterministic and stochastic wave equation, c.f. \Cref{result:Limits_operator_cosine_sine} and \Cref{result:scaling_limit}.

A key result in harmonic analysis is the Riemann-Lebesgue lemma that states that the Fourier transform of an $L^{1}$-function vanishes at infinity, c.f. \citet{kahaneGeneralizationsRiemannLebesgueCantorLebesgue1980}. This result is important for studying the wave equation as it provides a tool for studying the long-term behaviour of the energy within the hyperbolic system. Using the Riemann-Lebesgue lemma, it can be shown that the limiting operator $\vartheta(x_0)\Delta$ on $H^{2}(\mathbb{R}^{d}) \subset L^{2}(\mathbb{R}^{d})$ is a so-called Riemann-Lebesgue operator, satisfying the convergence 
\begin{equation*}
\sin(\delta^{-1}t (-\Delta)^{1/2}) \xrightarrow{w} 0, \quad \delta \rightarrow 0,
\end{equation*}
in the weak operator topology. In particular, if the generator is a Riemann-Lebesgue operator, long term, the kinetic and potential energy contribute equally to the total energy within the system. This effect is called the equipartition of energy.

In \Cref{section:asymptotic_energetic_behaviour}, we formally introduce Riemann-Lebesgue operators and show that asymptotically $A_{\vartheta,\delta}$ satisfies
\begin{equation*}
\label{eq:overview_main_results}
\sin(\delta^{-1}t (-A_{\vartheta,\delta})^{1/2}) \xrightarrow{w} 0, \quad \delta \rightarrow 0,
\end{equation*}
see \Cref{result:approximate_riemann_lebesgue} for the precise notion of convergence associated with varying domains. Thus, we prove that $A_{\vartheta,\delta}$ inherits the Riemann-Lebesgue property of the limiting operator $\vartheta(0)\Delta$ and we obtain the convergence
\begin{equation*}
\Vert S_{\vartheta, \delta}(\delta^{-1}t) \Delta K \Vert_{L^{2}(\Lambda_{\delta})}^{2} \rightarrow \frac{1}{2 \vartheta(x_0)} \Vert \nabla K \Vert^{2}_{L^{2}(\mathbb{R}^{d})}, \quad \delta \rightarrow 0,  \quad t \in [0,T]. 
\end{equation*}
The scaling of the local observations as $\delta \rightarrow 0$, therefore, corresponds to the long-term energetic behaviour within a deterministic wave equation with first-order initial condition $\Delta K$. \Cref{example:unbounded_domain_main_example} of \Cref{section:asymptotic_energetic_behaviour}, illustrates how the Riemann-Lebesgue lemma emerges when analysing local measurements of a stochastic wave equation on the unbounded spatial domain $\mathbb{R}^{d}$ given a constant wave speed $\vartheta$. It provides the blueprint on how the asymptotic properties of the statistical quantities in \Cref{section:asymptotics_augmented_MLE_main} are derived based on the analytical results of \Cref{section:asymptotic_energetic_behaviour}.

The finite-sample properties of the augmented MLE, which are demonstrated in the simulations of \Cref{section:numerical_example}, are in line with our theoretical findings. Surprisingly, we can even detect the hyperbolic dependence $T^{-2}$ within the empirical asymptotic variance of the augmented MLE. The appendix is devoted to more technical proofs of the well-posedness of the SPDE in \Cref{section:well-posedness}, spectral asymptotics in \Cref{section:spectral_asymptotics}, the asymptotic energetic behaviour in \Cref{section:energy_results} and the asymptotic properties of the observed Fisher information $I_{\delta}$ and the bias $R_{\delta}$ in the error decomposition for the augmented MLE in \Cref{section:augmented_MLE}.

\section{The model}
\label{section:model}
\subsection{Notation}
Let $\Lambda \subset \mathbb{R}^{d}$ be an open bounded and convex set with $C^{2}$-boundary and consider the standard $L^2$-space $(L^2(\Lambda), \langle \cdot, \cdot\rangle_{L^{2}(\Lambda)})$ equipped with the usual inner product. Let $H^k(\Lambda)$ be the $L^2$-Sobolev spaces of order $k \in \mathbb{N}$ and define $H^1_0(\Lambda)$ as the closure of $C_c^{\infty} (\overline{\Lambda})$ in $H^1(\Lambda)$. The notation $H^{-1}_0(\Lambda)=H_0^{1}(\Lambda)^{*}$ denotes the dual space of $H_0^{1}(\Lambda)$. Similarly, we will abbreviate by $H_0^{\alpha}(\Lambda)$ and $\alpha \in \mathbb{R}$ the associated fractional Sobolev spaces introduced in \citet[Section 3]{kovacsFiniteElementApproximation2010}. The space of $L^{2}(\Lambda)$ will be identified with its dual space $L^{2}(\Lambda)^{*}$. We write $\langle \cdot,\cdot \rangle_{\mathbb{R}^{d}}$ for the Euclidean inner product and $|\cdot|_{\mathbb{R}^{d}}$ for the norm. If $U$ and $H$ are two Hilbert spaces, we abbreviate the space of bounded linear operators by $\mathcal{L}(U, H)$ and the Hilbert-Schmidt norm by $\Vert \cdot \Vert_{\mathrm{HS}(U, H)}$. Let $\Delta$ be the Laplace operator on $L^{2}(\mathbb{R}^{d})$ defined on the domain $D(\Delta)=H^{2}(\mathbb{R}^{d})$. For $f \in L^{1}(\mathbb{R}^{d})$, we define the Fourier transform by $\mathcal{F}[f](\fv)=\int_{\mathbb{R}^{d}} f(x)e^{i \langle \fv, x\rangle_{\mathbb{R}^{d}}} \mathrm{d}x$ for $\fv \in \mathbb{R}^{d}$. We further use the usual extension of the Fourier transform to the isomorphism $ \mathcal{F}:L^{2}(\mathbb{R}^{d}) \rightarrow L^{2}(\mathbb{R}^{d})$. 
\subsection{The SPDE model}
Throughout this work, we consider a fixed time horizon $T < \infty$. Consider the general second-order operator $A_{\vartheta}:H_0^{1}(\Lambda) \rightarrow H_0^{-1}(\Lambda)$ defined through
\begin{equation}
\label{eq:weak_laplace}
\langle A_{\vartheta}z_1, z_2\rangle_{H_0^{-1}(\Lambda),H_0^{1}(\Lambda)} \coloneqq -\langle \vartheta \nabla z_1, \nabla z_2\rangle_{L^{2}(\Lambda)}, \quad z_1,z_2 \in H_0^{1}(\Lambda).
\end{equation}
For $z \in H_0^{1}(\Lambda) \cap H^{2}(\Lambda) \subset L^{2}(\Lambda)$ we recover the usual second order elliptic operator $A_\vartheta z \coloneqq \mathrm{div}(\vartheta \nabla z) =\sum_{i=1}^d \partial_i(\vartheta \partial_i z)$ with Dirichlet boundary conditions. The aim of this work is the non-parametric estimation of the wave speed of the stochastic wave equation
\begin{equation}
\label{eq:stochastic_wave_equation}
\begin{cases}
\partial_{tt}^{2}u(t)=A_{\vartheta}u(t)+ \dot{W}(t), \quad &t \in (0,T],\\
\restr{u(t)}{\partial \Lambda} = 0,  \quad &t \in (0,T],\\
u(0) = u_0, \dot{u}(0) =v_0,\\
\end{cases}
\end{equation}
with the wave speed $\vartheta$ satisfying \assth{} throughout, a cylindrical Wiener process $(W(t), t \in [0,T])$ on $L^{2}(\Lambda)$ and initial conditions $(u_0,v_0) \in L^{2}(\Lambda) \times H_0^{-1}(\Lambda)$.
\begin{assumptionRegTh}{\vartheta}{\alpha}\label{assumption:regularity_wave_speed}
  Suppose that $\vartheta \in C^{1+\alpha}(\overline{\Lambda})$ for $\alpha>0$ describes the space-dependent wave speed, satisfying $\mathrm{min}_{x \in \overline{\Lambda}} \vartheta(x)>0$. If $d=1$, we further assume $\alpha>1/2$.
\end{assumptionRegTh}
\begin{assumptionInitial}{u_0}{v_0}\label{assumption:initial_conditions} Suppose $u_0 \in H_0^{1}(\Lambda) \cap H^{2}(\Lambda)$ and $v_0 \in H_0^{1}(\Lambda)$.
 \end{assumptionInitial}
\begin{remark}
\label{remark:regularity}
\begin{enumerate}
  \item[]
  \item The regularity Assumption \assth{} is analogous to the assumption on the diffusivity in \citet*{altmeyerNonparametricEstimationLinear2021} because we require similar results based on the hypercontractivity of the heat semigroup. In particular, since we have 
  \begin{equation*}
  \Vert e^{t \Delta}u \Vert_{L^{2}(\mathbb{R}^{d})} \lesssim \mathrm{min}(\Vert u \Vert_{L^{2}(\mathbb{R}^{d})}, t^{-d/4} \Vert u \Vert_{L^{1}(\mathbb{R}^{d})}), \quad u \in L^{1}(\mathbb{R}^{d}) \cap L^{2}(\mathbb{R}^{d}), 
  \end{equation*}
  the decay in time of the heat semigroup $(e^{t\Delta}, t \geq 0)$ becomes stronger as the spatial dimension increases. The resulting decay is minimal in dimension one. Hence, we require a slightly higher order of regularity for the wave speed $\vartheta$ for $d=1$.
  \item By the Assumption \assth{} the wave speed $\vartheta$ is lower and upper bounded on the domain $\Lambda$. Thus, the standard inner product on $H_{0}^{1}(\Lambda)$ can be replaced by the inner product $\langle \vartheta \nabla \cdot, \nabla \cdot\rangle_{L^{2}(\Lambda)}$ weighted with the wave speed $\vartheta$. By \citet*[Remark 24]{brezisFunctionalAnalysisSobolev2011}, the weighted Dirichlet problem is as regular as the usual Dirichlet problem provided that $\vartheta \in C^{1}(\overline{\Lambda})$. The operator \eqref{eq:weak_laplace}, incorporating the spatially varying wave speed $\vartheta$, is bijective, and the associated existence results for the stochastic wave equation \eqref{eq:stochastic_wave_equation} will essentially reduce to the parametric case in \Cref{section:well-posedness}.
\end{enumerate}
\end{remark}
\subsection{Operator cosine and sine functions}
\label{section:operator_cosine_and_sine}
Consider the deterministic abstract wave equation
\begin{equation*}
\label{eq:deterministic_wave_equation}
\partial_{tt}^{2}u(t)=A_{\vartheta}u(t), \quad u(0)=u_0, \quad \dot{u}(0)=v_0, \quad \quad t \in [0,T].
\end{equation*}
Setting $v=\partial_t u$, we may rewrite \eqref{eq:deterministic_wave_equation} system of two equations
\begin{equation*}
\label{eq:deterministic_system}
\partial_t u(t)=v(t), \quad \partial_t v(t)=A_{\vartheta}u(t), \quad u(0)=u_0, \quad v(0)=v_0, \quad t \in [0,T].
\end{equation*}
In particular, the system \eqref{eq:deterministic_system} gives rise to the first-order Cauchy problem
\begin{equation*}
\label{eq:deterministic_matrix_system}
\partial_t \begin{pmatrix}
u(t)\\
v(t)
\end{pmatrix}= \begin{pmatrix}
0 & I \\
A_{\vartheta} & 0
\end{pmatrix}\begin{pmatrix}
u(t)\\v(t)
\end{pmatrix}, \quad (u(0), v(0))^{\top}=(u_0, v_0)^{\top},
\end{equation*}
on a suitable product space between Hilbert spaces called phase space. In the rest of this section, we will find a representation of the strongly continuous group associated with \eqref{eq:deterministic_matrix_system} on a suitable phase-space through operator cosine and sine functions, which are the hyperbolic counterpart to the strongly continuous semigroups associated with parabolic Cauchy problems. For an introduction to the theory of operator sine and cosine function, we refer to \citet{arendtVectorvaluedLaplaceTransforms2001}. \par 
Using the operator \eqref{eq:weak_laplace}, we may interpret $H_0^{-1}(\Lambda)$ itself as a Hilbert space by defining the inner product
\begin{equation}
\label{eq:dual_inner}
\langle l,l' \rangle_{H_0^{-1}(\Lambda)} \coloneqq \langle A_{\vartheta}^{-1}l, A_{\vartheta}^{-1}l'\rangle_{H_0^{1}(\Lambda)}, \quad l,l' \in H_0^{-1}(\Lambda).
\end{equation}
By \Cref{remark:regularity}, we obtain the Gelfand triple 
\begin{equation*}
(H_{0}^{1}(\Lambda), L^{2}(\Lambda), H_{0}^{-1}(\Lambda)).
\end{equation*}
\citet*[Proposition 7.1.5]{arendtVectorvaluedLaplaceTransforms2001} show that with \assth{} the operator $A_{\vartheta}$ defined in \eqref{eq:weak_laplace} is a bounded, self-adjoint operator and that $L^{2}(\Lambda)\times H_{0}^{-1}(\Lambda)$ is the phase space associated with the cosine function generated by $A_{\vartheta}$ on $H_{0}^{-1}(\Lambda)$ with the inner product \eqref{eq:dual_inner}. \par
 Given \citet*[Theorem 3.14.11]{arendtVectorvaluedLaplaceTransforms2001}, let $(C_{\vartheta}(t), t \in\mathbb{R})$ and $(S_{\vartheta}(t), t \in \mathbb{R})$, with $S(t)\coloneqq\int_{0}^{t}C(s)\mathrm{d}s$ for $t \in [0,T]$, be the operator cosine and sine-functions associated with the operator $A_{\vartheta}$. We abbreviate by $\mathcal{A}_{\vartheta}$ the associated generator on the phase space $L^{2}(\Lambda)\times H_0^{-1}(\Lambda)$, given by
 \begin{equation}
 \label{eq:matrix_generator}
 \begin{aligned}
 D(\mathcal{A}_{\vartheta})&=D(A_{\vartheta}) \times L^{2}(\Lambda) = H_0^{1}(\Lambda)\times L^{2}(\Lambda),\\
 \mathcal{A}_{\vartheta}\begin{pmatrix}z_1\\z_2
\end{pmatrix} &= \begin{pmatrix}
0 & I\\
A_{\vartheta} & 0
\end{pmatrix}\begin{pmatrix}z_1\\z_2
\end{pmatrix} = \begin{pmatrix}
z_2 \\ A_{\vartheta}z_1
\end{pmatrix}.
\end{aligned}
 \end{equation}
The operator sine and cosine functions take the following values
\begin{equation*}
\begin{aligned}
\label{eq:regularity_cosine_sine}
 S_{\vartheta}(\cdot)l &\in C(\mathbb{R},L^{2}(\Lambda)), \quad l \in H_0^{-1}(\Lambda),\\
 S_{\vartheta}(\cdot)z &\in C(\mathbb{R},H_0^{1}(\Lambda)), \quad z \in L^{2}(\Lambda),\\
 C_{\vartheta}(\cdot)z &\in C^{1}(\mathbb{R},H^{-1}_0(\Lambda))\cap C(\mathbb{R},L^{2}(\Lambda)), \quad z \in L^{2}(\Lambda).
\end{aligned}
\end{equation*}
In particular, by \citet[Theorem 3.14.11]{arendtVectorvaluedLaplaceTransforms2001}, the operator $\mathcal{A}_{\vartheta}$ defined through \eqref{eq:matrix_generator} generates a $C_0$-group $\mathcal{J}_{\vartheta}$ on the phase-space $L^{2}(\Lambda)\times H_0^{-1}(\Lambda)$ given by 
\begin{equation}
\label{eq:operator_semigroup_wave_equation}
\mathcal{J}_{\vartheta}(t)=\begin{pmatrix}
C_{\vartheta}(t)& S_{\vartheta}(t)\\
C'_{\vartheta}(t)& C_{\vartheta}(t)
\end{pmatrix} = \begin{pmatrix}
C_{\vartheta}(t)& S_{\vartheta}(t)\\
A_{\vartheta} S_{\vartheta}(t) & C_{\vartheta}(t)
\end{pmatrix}, \quad t \in \mathbb{R}.
\end{equation}
\begin{remark}
The operator cosine function is precisely the cosine applied to the square root of the negative of the generator by the functional calculus. In contrast, the operator sine is defined through $S(t)=\int_{0}^{t}C(s)\mathrm{d}s$ for $t \in[0,T]$ and does not correspond to the operator sine as induced by the functional calculus. The operator sine $S(t)$ is only used because it emerges naturally in d'Alembert's formula and is therefore involved directly in the solution of a second-order abstract Cauchy problem. For an overview of the different types of notations for operator cosine and sine functions, we refer to \citet*{pandolfiSystemsPersistentMemory2021}.
\end{remark}
\subsection{Well-posedness for the SPDE}
\label{section:well-posedness-main}
Throughout this section, we assume zero-initial conditions for simplicity:
\begin{equation}
\label{eq:zero_initial_condition}
u_0(x) = v_0(x) = 0, \quad x \in \Lambda.
\end{equation}
All the results in this section immediately extend to initial conditions satisfying \assini{}. We begin by considering the particular case $d=1$.
The following result provides the existence and uniqueness of a mild function-valued solution to the stochastic wave equation on a bounded spatial domain. 
\begin{theorem}[Existence of the one-dimensional stochastic wave equation]
\label{theorem:existence_stochastic_wave_equation}
Assume zero initial conditions \eqref{eq:zero_initial_condition} in \eqref{eq:stochastic_wave_equation}. Let $\Lambda \subset \mathbb{R}$ be an open and bounded one-dimensional spatial domain. Suppose that $(W(t), t \in [0,T])$ is a cylindrical Wiener process on $L^{2}(\Lambda)$ on a filtered probability space $(\Omega, \mathcal{F}, \mathcal{F}_{t \geq 0}, \mathbb{P})$. Then, \eqref{eq:stochastic_wave_equation} has a unique mild solution given by the variations of constants formula
\begin{equation}
\label{eq:mild_solution_stochastic_wave_equation}
X(t)=\int_{0}^{t}\mathcal{J}_{\vartheta}(t-s)B\mathrm{d}W(s), \quad t \geq 0,
\end{equation}
where $(\mathcal{J}_{\vartheta}(t), t \in [0,T])$ is given by $\eqref{eq:operator_semigroup_wave_equation}$ and $B:L^{2}(\Lambda) \rightarrow L^{2}(\Lambda)\times H^{-1}_0(\Lambda)$ with $Bu=(0, u) \in L^{2}(\Lambda) \times  (L^{2}(\Lambda))^{'} \subset  L^{2}(\Lambda)\times H^{-1}_0(\Lambda)$.
\end{theorem}
\begin{proof}[Proof of \Cref{theorem:existence_stochastic_wave_equation}] This is an immediate corollary of \citet*[Theorem 3.1 and Remark 3.2]{kovacsFiniteElementApproximation2010}.
\end{proof}
\noindent
By \citet*[Theorem 5.4]{dapratoStochasticEquationsInfinite2014}, the mild solution \eqref{eq:mild_solution_stochastic_wave_equation} is also a weak solution and satisfies
\begin{equation}
\label{eq:weak_solution_stochastic_wave_equation}
\begin{aligned}
&\langle X(t), U \rangle_{L^{2}(\Lambda)\times H^{-1}_0(\Lambda)} \\ &= \int_{0}^{t}\langle X(s),\mathcal{A}^{*}_{\vartheta}U \rangle_{L^{2}(\Lambda)\times H^{-1}_0(\Lambda)}  \mathrm{d}s+ \langle BW(t), U\rangle_{L^{2}(\Lambda)\times H^{-1}_0(\Lambda)}, \quad U \in D(\mathcal{A}^{*}_{\vartheta}).
\end{aligned}
\end{equation}
Let us denote $(u(t),v(t))^{\top}\coloneqq X(t)$. 
We obtain the following important dynamical representation by determining the adjoint of $\mathcal{A}_{\vartheta}$ and carefully 
differentiating between the inner product and dual-pairing associated with $H_0^{-1}(\Lambda)$.
\begin{proposition}[Dynamic representation of the weak solution]
\label{result:system_of_equations_from_weak_solution}
Assume zero initial conditions \eqref{eq:zero_initial_condition} in \eqref{eq:stochastic_wave_equation}. Let $\Lambda \subset \mathbb{R}$ be an open and bounded one-dimensional spatial domain. For every function $z \in H_0^{1}(\Lambda)\cap H^{2}(\Lambda)$, we have
\begin{equation}
\label{eq:dynamic_representation_d_1}
\begin{aligned}
\langle u(t), z \rangle_{L^{2}(\Lambda)} &= \int_{0}^{t} \langle v(s),z \rangle_{H^{-1}_0(\Lambda), H^{1}_0(\Lambda)}\mathrm{d}s,\\
\langle v(t), z\rangle_{H^{-1}_0(\Lambda), H^{1}_0(\Lambda)} &= \int_{0}^{t} \langle u(s),A_{\vartheta}z \rangle_{L^{2}(\Lambda)} \mathrm{d}s + \langle W(t),z \rangle_{L^{2}(\Lambda)},
\end{aligned}
\end{equation}
for all $t \in [0,T]$ on the same set of probability one. 
\end{proposition}
\begin{proof}[Proof of \Cref{result:system_of_equations_from_weak_solution}]
The result is proved on page \pageref{proof:system_of_equations_from_weak_solution}.
\end{proof}
\noindent We now turn to $d>1$. The arguments employed for \Cref{theorem:existence_stochastic_wave_equation} and \Cref{result:system_of_equations_from_weak_solution} cannot be applied because 
\begin{equation*}
\int_{0}^{T}\Vert \mathcal{J}_{\vartheta}(t)B \Vert^{2}_{\mathrm{HS}(L^{2}(\Lambda), L^{2}(\Lambda) \times H^{-1}_0(\Lambda))} \mathrm{d}t= \infty,
\end{equation*}
violates condition (5.3) of \citet*[Theorem 5.2]{dapratoStochasticEquationsInfinite2014}, required for an informative version of Itô's isometry. In this case, the stochastic integral \eqref{eq:mild_solution_stochastic_wave_equation} is only well-defined as a distribution, and there does not exist any standard function-valued solution.\par 
This issue is also prevalent in the case of the stochastic heat equation and was resolved in \citet[Proposition 2.1]{altmeyerNonparametricEstimationLinear2021} by passing to a Gaussian process that preserves the covariance structure induced by the associated strongly continuous semigroup. The following result extends this argument to the setting of the stochastic wave equation.
\begin{proposition}[Properties of the Gaussian process solution]
\label{result:Gaussian_process_solution}
There is a centred Gaussian process $(\mathcal{V}(t, U), t \in [0,T], U \in L^{2}(\Lambda)\times H_0^{-1}(\Lambda))$ given by \eqref{eq:gaussian_process_joint_solution} with the covariance function
\begin{equation}
\label{eq:covariance_gaussian_process_solution}
\mathrm{Cov}(\mathcal{V}(t, U), \mathcal{V}(t', U')) = \int_{0}^{t \land t'} \langle B^{*}\mathcal{J}^{*}_{\vartheta}(t-s)U,B^{*}\mathcal{J}^{*}_{\vartheta}(t'-s)U' \rangle_{L^{2}(\Lambda)} \mathrm{d}s,
\end{equation}
for $U, \tilde{U} \in L^{2}(\Lambda)\times H_0^{-1}(\Lambda)$ and $t, t' \in [0,T]$. 
The process $(\mathcal{V}(t, U), U \in L^{2}(\Lambda)\times H_0^{-1}(\Lambda))$ is given by the sum of two Gaussian processes, $(u_{\mathcal{V}}(t, z), t \in [0,T], z \in L^{2}(\Lambda))$ and $(v_{\mathcal{V}}(t, z), t \in [0,T], z \in L^{2}(\Lambda))$, satisfying the dynamic
\begin{align}
\label{eq:system_of_equations_gaussian_process_solution_eq_1}
u_{\mathcal{V}}(t, z) &= \int_{0}^{t}v_{\mathcal{V}}(s, z) \mathrm{d}s\\
\label{eq:system_of_equations_gaussian_process_solution_eq_2}
v_{\mathcal{V}}(t, z) &= \int_{0}^{t}u_{\mathcal{V}}(s, A_{\vartheta} z) \mathrm{d}s + \langle W(t), z\rangle_{L^{2}(\Lambda)}, \quad z \in H_0^{1}(\Lambda) \cap H^{2}(\Lambda).
\end{align}
\end{proposition}
\begin{proof}[Proof of \Cref{result:Gaussian_process_solution}]
The result is proved on page \pageref{proof:Gaussian_process_solution}.
\end{proof}
Justified by this result, we will write $\langle u(t), z\rangle_{L^{2}(\Lambda)}$ and $\langle v(t), z\rangle$ throughout instead of $(u_{\mathcal{V}}(t, z))$ and $(v_{\mathcal{V}}(t, z))$ for $z \in L^{2}(\Lambda)$ and any $t \in [0,T]$ and consider the Gaussian process from \eqref{result:Gaussian_process_solution} as the solution to the stochastic wave equation \eqref{eq:stochastic_wave_equation}.
\begin{remark}[Solution concepts]
\begin{enumerate}
\item[]
\item Notice that the inclusion mappings $\iota_{r,s}:H^{r}_0(\Lambda) \rightarrow H_{0}^{s}(\Lambda)$ are Hilbert-Schmidt provided that $r-s > d/2$ for any $r>s$ and $r,s \in \mathbb{R}$, where $H^{-s}_0(\Lambda)$ are fractional Sobolev spaces of the negative order. In particular, there is a Hilbert-Schmidt embedding $\tilde{\iota}:L^{2}(\Lambda)\times H_0^{-1}(\Lambda) \rightarrow H_0^{-s}(\Lambda)\times H_0^{-(s+1)}(\Lambda)$, and we have 
    \begin{equation*}
    \int_{0}^{T}\Vert \tilde{\iota} \mathcal{J}(t)B \Vert^{2}_{\mathrm{HS}(L^{2}(\Lambda), H_0^{-s}(\Lambda)\times H_0^{-(s+1)}(\Lambda))} \mathrm{d}t< \infty.
    \end{equation*}
    Consequently, (c.f. Remark 5.7 in \citet*{hairerIntroductionStochasticPDEs2023}) the process $(X(t), t \in [0,T])$ defined by \eqref{eq:mild_solution_stochastic_wave_equation} takes values in $H_0^{-s}(\Lambda)\times H_0^{-(s+1)}(\Lambda)$. Note that the process $\mathcal{V}$ from \Cref{result:Gaussian_process_solution} is well-defined independently of the embedding space for $(X(t),t \in [0,T])$ and extends the linear form $U \mapsto \langle X(t), U\rangle$ from $C_c^{\infty}(\overline{\Lambda}) \times C_c^{\infty}(\overline{\Lambda})$ to $L^{2}(\Lambda) \times H_0^{-1}(\Lambda)$.
\item For the unbounded spatial domain $\Lambda=\mathbb{R}$, function-valued solutions exist in a weighted $L^{2}_{\rho}$-space where $\rho$ is an integrable weight function, see \citet*{karczewskaStochasticPDEsFunctionvalued1999}. In this case, \Cref{result:system_of_equations_from_weak_solution} is not immediate as the change in the underlying norm also changes partial integration properties crucial to the behaviour of the Laplace operator. Instead, a random-field approach for the stochastic wave equation similar to the approach described in \citet*{walshIntroductionStochasticPartial1986} can also lead to dynamic representations analogous to \eqref{eq:dynamic_representation_d_1}, see \citet*{delerueNormalApproximationSolution2021}.
\item For the case of space-time white noise, there is neither a function-valued solution \eqref{eq:stochastic_wave_equation} nor a random field solution to the stochastic wave equation for $d>1$, see for instance \citet*{foondunLocaltimeCorrespondenceStochastic2010}. As we are only interested in the covariance structure of the process and a representation of the form \eqref{eq:dynamic_representation_d_1}, it is sufficient for our purpose to understand the behaviour of the distribution valued solution evaluated through more regular test functions.
\end{enumerate}
\end{remark}
\begin{corollary}[Covariance structure of the contributing processes]
\label{result:covariance_structure}
For $z, z' \in L^{2}(\Lambda)$ and $t,s \in [0,T]$, we have
\begin{align*}
\mathrm{Cov}(\langle u(t), z\rangle_{L^{2}(\Lambda)}, \langle u(s), z'\rangle_{L^{2}(\Lambda)}) &= \int_{0}^{t \land s} \langle S_{\vartheta}(t-r) z, S_{\vartheta}(s-r)z'\rangle_{L^{2}(\Lambda)} \mathrm{d}r,\\
\mathrm{Cov}(\langle v(t), z\rangle, \langle v(s), z'\rangle)  &= \int_{0}^{t \land s}  \langle C_{\vartheta}(t-r) z, C_{\vartheta}(s-r)z'\rangle_{L^{2}(\Lambda)} \mathrm{d}r.
\end{align*}
\end{corollary}
\begin{proof}[Proof of \Cref{result:covariance_structure}] The result follows from immediately from \citet[Proposition 4.28]{dapratoStochasticEquationsInfinite2014}, \Cref{result:unitary_group} and the definition of the processes $u_{\mathcal{V}}(t, z)$ and $v_{\mathcal{V}}(t, z)$.
\end{proof}
\section{Scaling limits}
\label{section:scaling_limit}
This section is devoted to deriving a scaling limit for the stochastic wave equation. As introduced in \citet*{altmeyerNonparametricEstimationLinear2021}, we fix $\delta>0$ and define for any $z \in L^{2}(\mathbb{R}^{d})$ the rescaling
\begin{equation}
\label{eq:rescaling}
\begin{aligned}
\Lambda_{\delta} &\coloneqq \delta^{-1}\Lambda = \{\delta^{-1}x \colon x \in \Lambda\}, \quad \Lambda_{0}\coloneqq \mathbb{R}^{d},\\
\quad z_{\delta}(x)&\coloneqq \delta^{-d/2}z(\delta^{-1}x), \quad x \in \mathbb{R}^{d}. 
\end{aligned}
\end{equation}
For convenience only, we consider a localisation around zero and estimate the unknown wave speed $\vartheta$ at $\vartheta(0)$. If we wish to estimate $\vartheta$ at some different $x_0 \in \Lambda$ the rescaling \eqref{eq:rescaling} has to be shifted by $x_0$ as introduced in
\citet{altmeyerNonparametricEstimationLinear2021}. The normalisation of the rescaling is arbitrary and satisfies $\Vert z_{\delta} \Vert_{L^{2}(\Lambda)}=\Vert z \Vert_{L^{2}(\mathbb{R}^{d})}$ for convenience.\par
The rescaled generator $A_{\vartheta, \delta} \coloneqq A_{\vartheta(\delta \cdot)}$ with $D(A_{\vartheta,\delta})=H_0^{1}(\Lambda_{\delta})\cap H^{2}(\Lambda_{\delta})$ induces operator sine and cosine functions $(C_{\vartheta, \delta}(t), t \in [0,T])$ and $(S_{\vartheta,\delta}(t), t \in [0,T])$. Analogous to \citet[Lemma 3.1]{altmeyerNonparametricEstimationLinear2021}, the following result characterises the rescaling behaviour of operator cosine and sine functions acting on localised functions.
\begin{lemma}[Rescaling of operator cosine and sine functions]
\label{result:rescaling_trigonometric_operator_families}
 For $\delta>0$:
\begin{enumerate}
\item If $z \in H_0^{1}(\Lambda_{\delta})\cap H^{2}(\Lambda_\delta)$, then $A_{\vartheta}z_{\delta}=\delta^{-2}(A_{\vartheta,\delta}z)_{\delta}$.
\item \label{enum:rescaling} If $z \in L^{2}(\Lambda_{\delta})$, then $S_{\vartheta}(t)z_{\delta}= \delta(S_{\vartheta,\delta}(\delta^{-1}t)z)_{\delta}$ and $C_{\vartheta}(t)z_{\delta}=(C_{\vartheta,\delta}(\delta^{-1}t)z)_{\delta}$.
\end{enumerate}
\end{lemma}
\begin{proof}[Proof of \Cref{result:rescaling_trigonometric_operator_families}]
It suffices to prove the result for any $z_1, z_2 \in C_c^{\infty}(\overline{\Lambda}_{\delta})$. In particular, (i) is a special case of (i) in \citet[Lemma 3.1]{altmeyerNonparametricEstimationLinear2021}. For (ii), we define
\begin{equation*}
w(t)\coloneqq \delta(S_{\vartheta,\delta}(\delta^{-1}t)z_1)_{\delta} + (C_{\vartheta,\delta}(\delta^{-1}t)z_2)_{\delta}, \quad t \in [0,T].
\end{equation*}
The first and second derivatives of $w$ are given by
\begin{align*}
\dot{w}(t) &= (C_{\vartheta,\delta}(\delta^{-1}t)z_1)_{\delta} + \delta^{-1}(A_{\vartheta,\delta}S_{\vartheta,\delta}(\delta^{-1}t)z_2)_{\delta}\\
\ddot{w}(t)&= \delta^{-1}(A_{\vartheta,\delta}S_{\vartheta,\delta}(\delta^{-1}t)z_1)_{\delta} + \delta^{-2}(A_{\vartheta, \delta} C_{\vartheta,\delta}(\delta^{-1}t)z_2)_{\delta}\\
&= A_{\vartheta} \delta (S_{\vartheta,\delta}(\delta^{-1}t)z_1)_{\delta} + A_{\vartheta}(C_{\vartheta,\delta}(\delta^{-1}t)z_2)_{\delta} = A_{\vartheta}w(t), \quad t \in [0,T].
\end{align*}
Given \citet[Corollary 3.14.8]{arendtVectorvaluedLaplaceTransforms2001}, we conclude from $w(0)=(z_1)_{\delta}$, $\dot{w}(0)=(z_2)_{\delta}$ and $w''(t)=A_{\vartheta}w(t)$ the identity $w(t)=S_{\vartheta}(t)(z_1)_\delta + C_{\vartheta}(t)(z_2)_{\delta}$ as the function $w$ is the unique solution to the second-order abstract Cauchy problem
\begin{equation*}
\partial_{tt}^{2}g(t)=A_{\vartheta}g(t), \quad g(0)=(z_1)_{\delta},\quad \dot{g}(0)= (z_2)_{\delta}, \quad t \in [0,T]. 
\end{equation*}
The result follows by setting $z_1$ or $z_2$ to zero, respectively. 
\end{proof}
The following example illustrates the extension of the rescaling \Cref{result:rescaling_trigonometric_operator_families} to an unbounded spatial domain in the one-dimensional case. 
\begin{example}[Parametric rescaling on the unbounded domain]
\label{example:parametric_rescaling}
\Cref{result:rescaling_trigonometric_operator_families} extends naturally to the setting of an unbounded domain. In that case, we extend our notation for $\delta=0$ and suppose that $(C_{\vartheta,0}(t), t \in [0,T])$ and $(S_{\vartheta,0}(t), t \in [0,T])$ are the operator cosine and sine functions associated with the operator $\vartheta(0)\Delta$ defined on $H^{2}(\mathbb{R}^{d})$. In order to simplify our notation, we will write $(S_{\vartheta(0)}(t), t \in [0,T])$ and $(C_{\vartheta (0)}(t), t \in [0,T])$ instead of $(C_{\vartheta,0}(t), t \in [0,T])$ and $(S_{\vartheta,0}(t), t \in [0,T])$, respectively.\par
Assume for simplicity that $d=1$ and consider the left translation group $(T_{\vartheta(0)}(t)z)(x)=z(x+ \vartheta(0)t)$ for $t \in [0,T]$, $z \in L^{2}(\mathbb{R})$ and $x \in \mathbb{R}$. Using \citet[Example 3.14.15]{arendtVectorvaluedLaplaceTransforms2001}, we can represent the operator cosine function through $C_{\vartheta(0)}(t)z \coloneqq (T_{\vartheta(0)}(t)z + T_{\vartheta(0)}(-t)z)/2$ for $z \in L^{2}(\mathbb{R})$ and $t \in [0,T]$. The operator sine function associated with $(C_{\vartheta(0)}(t), t \in [0,T])$ is then given by
\begin{equation}
\label{eq:representation_operator_sine_unbounded}
\begin{aligned}
[S_{\vartheta(0)}(t)z](x) &\coloneqq \int_{0}^{t}[C_{\vartheta(0)}(s)z](x) \mathrm{d}s \\&= \frac{1}{2}\int_{0}^{t}z(x-\vartheta(0)s)+z(x+\vartheta(0)s)\mathrm{d}s, \quad t \in [0,T], \quad x \in \mathbb{R}.
\end{aligned}
\end{equation}
We recover the rescaling \Cref{result:rescaling_trigonometric_operator_families} (ii) in the case of the unbounded one-dimensional spatial domain by applying the operator cosine to a localised function
\begin{align*}
[C_{\vartheta(0)}(t)z_{\delta}](x) &= \frac{1}{2}(z_{\delta}(x+\vartheta(0)t) + z_{\delta}(x-\vartheta(0)t)) \\
&= \frac{\delta^{-1/2}}{2}(z(\delta^{-1}x+\delta^{-1}\vartheta(0)t ) + z(\delta^{-1}x - \delta^{-1}\vartheta(0)t)) \\
&= [C_{\vartheta(0)}(\delta^{-1}t)z]_{\delta}(x), \quad t \in [0,T],\quad  x \in \mathbb{R}.
\end{align*}
Similarly, using \eqref{eq:representation_operator_sine_unbounded}, the rescaling of the operator sine function is obtained using the transformation theorem:
\begin{align*}
[S_{\vartheta(0)}(t)z_{\delta}](x)&=\delta^{-1/2}\int_{0}^{t} z(\delta^{-1}x+\delta^{-1}\vartheta(0)s)+ z(\delta^{-1}x - \delta^{-1}\vartheta(0)s)\mathrm{d}s
\\&= \delta^{1/2} \int_{0}^{\delta^{-1}t}z(\delta^{-1}x+\vartheta(0)r)+ z (\delta^{-1}x-\vartheta(0)r)\mathrm{d}r\\
&= \delta \delta^{-1/2} [S_{\vartheta(0)}(\delta^{-1}t)z](\delta^{-1}x)= \delta [S_{\vartheta(0)}(\delta^{-1}t)z]_{\delta}(x), 
\end{align*}
for $t \in [0,T]$ and $x \in \mathbb{R}$.
As expected, \Cref{result:rescaling_trigonometric_operator_families} still holds in the parametric case on the unbounded spatial domain $\mathbb{R}$, only that the rescaled families of operators do not depend themselves on the scaling parameter $\delta>0$ through the unknown function. 
\end{example}
\noindent 
Let us denote by 
\begin{equation}
\label{eq:orthogonal_projection}
P_{\delta}z = \mathbbm{1}_{\Lambda_{\delta}}z, \quad P_{\delta}:L^{2}(\mathbb{R}^{d}) \rightarrow L^{2}(\Lambda_{\delta}),
\end{equation}
the orthogonal projection from $L^{2}(\mathbb{R}^{d})$ onto $L^{2}(\Lambda_{\delta})$. For more details on the conventions associated with this projection, see also \Cref{remark:subspace_projections}.\par
Using the asymptotic behaviour of the partition of unity associated with $A_{\vartheta,\delta}$ analysed in \Cref{section:spectral_asymptotics}, we obtain the following scaling behaviour for the operator cosine and sine functions $(C_{\vartheta,\delta}(t),t  \in [0,T])$ and $(S_{\vartheta,\delta}(t),t  \in [0,T])$.
\begin{proposition}[Deterministic scaling limits for the operator sine and cosine]
\label{result:Limits_operator_cosine_sine}
Grant \assth{}. Then, for any $z \in L^{2}(\mathbb{R}^{d})$, we have 
\begin{equation}
\label{eq:operator_convergence_varying_domain}
\begin{aligned}
S_{\vartheta,\delta}(t)P_{\delta}z \rightarrow S_{\vartheta(0)}(t)z, \quad C_{\vartheta,\delta}(t)P_{\delta}z \rightarrow C_{\vartheta(0)}(t)z, \quad \delta \rightarrow 0, \quad t \in [0,T],
\end{aligned}
\end{equation}
where $(C_{\vartheta(0)}(t), t \in [0,T])$ and $(S_{\vartheta(0)}(t), t \in [0,T])$ are the operator cosine and sine functions associated with the operator $\vartheta(0)\Delta$, see also page \pageref{example:parametric_rescaling}.
\end{proposition}
\begin{proof}[Proof of \Cref{result:Limits_operator_cosine_sine}] The result is proved on page \pageref{proof:Limits_operator_cosine_sine}.
\end{proof}
Similar to the scaling limit for the stochastic heat equation, see \citet[Theorem 3.6]{altmeyerNonparametricEstimationLinear2021}, we also obtain a scaling limit for the stochastic wave equation \eqref{eq:stochastic_wave_equation}. Suppose we consider the scaled localised process up until a finite time horizon $T>0$. Then, in the parametric case, as the wave equation has a finite speed of propagation. There exists some $\delta_T>0$ such that for $\delta \in (0,\delta_T)$, the localised process associated with the bounded domain cannot be differentiated from the process associated with the unbounded spatial domain upon testing against a function compactly supported in $\Lambda$. These insights are summarised in the following result if we assume zero initial conditions. 
\begin{proposition}[Scaling limit for the stochastic wave equation]
\label{result:scaling_limit}
Assume that $u_0 = 0$ and $v_0 = 0$.  Consider the process $Z_{\delta}(t, z) \coloneqq \delta^{- 3/2}\langle u (\delta t), (P_{\delta}z)_{\delta} \rangle_{L^{2}(\Lambda)}$ for  $z \in L^{2}(\mathbb{R}^{d})$ and $t \geq 0$. For $t \in [0,T]$ and $z \in L^{2}(\mathbb{R}^{d})$, let $\underline{Z}(t,z)=\langle \underline{u}(t), z\rangle_{L^{2}(\mathbb{R}^{d})}$ be the scaled localisation of the limiting process $(\underline{u}(t), t \in [0,T])$ solving the stochastic wave equation 
\begin{equation*}
\label{eq:limiting_SPDE}
\partial_{tt}^{2}\underline{u}(t)=\vartheta(0)\Delta\underline{u}(t)+ \dot{\underline{W}}(t), \quad \underline{u}(0) = \dot{\underline{u}}(0) = 0, \quad t \geq 0,
\end{equation*}
as a Gaussian process in the sense of \Cref{result:Gaussian_process_solution} with the space-time white-noise $(\dot{\underline{W}}(t), t \in [0,T])$ on $L^{2}(\mathbb{R}^{d})$.
\begin{enumerate}
\item Then, the finite-dimensional distributions of the process $(Z_{\delta}(t, z), t \geq 0, z \in L^{2}(\mathbb{R}^{d}))$ converge to those of $(\underline{Z}(t, z), t \geq 0, z \in L^{2}(\mathbb{R}^{d}))$.
\item Let $T>0$ be some fixed time horizon and assume that $\vartheta$ is constant. Let $z \in L^{2}(\mathbb{R}^{d})$ be compactly supported in $\Lambda$. Then, there exists some $\delta_T=\delta_T(\vartheta_*, z)>0$ such that the finite-dimensional distributions of the process $(Z_{\delta}(t, z), t \in [0,T])$ are identical to those of $(\underline{Z}(t, z), t \in [0,T])$, for any $0 < \delta < \delta_T$. 
\end{enumerate}
\end{proposition}
\begin{proof}[Proof of \Cref{result:scaling_limit}]
\label{proof:scaling_limit}
\begin{step}
Clearly, by \Cref{result:covariance_structure} and \Cref{result:rescaling_trigonometric_operator_families} the process $Z_{\delta}$ is a centred Gaussian process with the covariance function
\begin{equation}
\label{eq:covariance_structure_scaling_limit}
\begin{aligned}
&\mathbb{E}[Z_{\delta}(t,z_1)Z_{\delta}(s,z_2)]\\
&= \delta^{-3} \int_{0}^{\delta(t \land s)} \langle S_{\vartheta}(\delta t-r)(P_{\delta}z_1)_{\delta},S_{\vartheta}(\delta s-r)(P_{\delta}z_2)_{\delta} \rangle_{L^{2}(\Lambda)} \mathrm{d}r\\
&= \delta^{-3}\int_{0}^{\delta (t \land s)} \langle \delta S_{\vartheta,\delta}(\delta^{-1}(\delta t-r)) P_{\delta}z_1,\delta S_{\vartheta,\delta}(\delta^{-1}(\delta s-r))P_{\delta} z_2 \rangle_{L^{2}(\Lambda_{\delta})} \mathrm{d}r\\
&= \int_{0}^{t \land s}\langle S_{\vartheta,\delta}(t-r)P_{\delta} z_1, S_{\vartheta,\delta}(s-r)P_{\delta}z_2\rangle_{L^{2}(\Lambda_{\delta})} \mathrm{d}r, \quad t,s \geq 0, \quad z_1,z_2 \in L^{2}(\mathbb{R}^{d}).\\
\end{aligned}
\end{equation}
Similarly, the covariance function of the process $\underline{Z}$ is given by
\begin{equation}
\label{eq:covariance_structure_scaling_limit_unbounded}
\begin{aligned}
&\mathbb{E}[\underline{Z}(t,z_1)\underline{Z}(s,z_2)]\\
&= \int_{0}^{t \land s}\langle S_{\vartheta(0)}(t-r)z_1, S_{\vartheta(0)}(s-r)z_2\rangle_{L^{2}(\mathbb{R}^{d})} \mathrm{d}r, \quad t,s \geq 0, \quad z,z' \in L^{2}(\mathbb{R}^{d}).
\end{aligned}
\end{equation}
By \eqref{eq:operator_convergence_varying_domain} in \Cref{result:Limits_operator_cosine_sine}, we obtain the convergence $S_{\vartheta,\delta}(\tau) P_{\delta}z_1 \rightarrow S_{\vartheta(0)}(\tau)z_1$ in $L^{2}(\mathbb{R}^{d})$ for any $\tau \geq 0$ as $\delta \rightarrow 0$. The same is true for $z_2$. In particular, by the representation $S_{\vartheta,\delta}(\tau)z=\int_{0}^{\tau}C_{\vartheta,\delta}(s)z \mathrm{d}s$ and the boundedness of the cosine, we have with the functional calculus the upper bound 
\begin{equation*}
\sup_{0 < \delta \leq 1} \sup_{0 \leq \tau \leq T} \Vert S_{\vartheta,\delta}(\tau)z \Vert_{L^{2}(\Lambda_{\delta})} \leq T \Vert 
z \Vert_{L^{2}(\mathbb{R}^{d})}< \infty.
\end{equation*}
The result follows as \eqref{eq:covariance_structure_scaling_limit} converges to \eqref{eq:covariance_structure_scaling_limit_unbounded} by the dominated convergence theorem.
\end{step}
\begin{step}
As we have assumed $\vartheta(x) = \vartheta_*$ for some constant $\vartheta_*>0$ and all $x \in \Lambda$, it makes sense write $S_{\vartheta, \delta}=S_{\vartheta_*,\delta}$. As the wave equation has a finite propagation speed, see \citet*{evansPartialDifferentialEquations2010}, and $z$ is compactly supported, there exists some $\delta_T =\delta_{T}(\vartheta_*, z)>0$ such that for all $0 < \delta < \delta_{T}$ we have
\begin{equation}
\label{eq:finite_time_horizon_identitiy}
S_{\vartheta_*,\delta}(\tau)z = S_{\vartheta_*}(\tau)z,\quad \tau \in [0,T].
\end{equation}
Indeed, if $\delta_{T}>0$ is chosen so small that the outer edge of the light cone induced by $S_{\vartheta_*,\delta}(\tau)z$ does not reach the boundary of the domain $\Lambda_{\delta}$ up until the finite time horizon $T$, $S_{\vartheta_*,\delta}(\tau)z$ cannot be differentiated from $S_{\vartheta_*}(\tau)z$ for any $\tau \in [0,T]$. The identity \eqref{eq:finite_time_horizon_identitiy} then follows by the uniqueness of solutions to the wave equation; see \citet*[Corollary 3.14.8]{arendtVectorvaluedLaplaceTransforms2001}. As a consequence, combining the representations \eqref{eq:covariance_structure_scaling_limit} and \eqref{eq:covariance_structure_scaling_limit_unbounded} with the identity \eqref{eq:finite_time_horizon_identitiy}, we have
\begin{equation}
\label{eq:identity_for_covariance_kernels}
\mathbb{E}[Z_{\delta}(t,z)Z_{\delta}(s,z)] = \mathbb{E}[\underline{Z}(t,z)\underline{Z}(s,z)], \quad t,s \in [0,T], \quad 0 < \delta < \delta_{T}.
\end{equation}
The result follows from \eqref{eq:identity_for_covariance_kernels} as two centred Gaussian processes with the same covariance function have the same law. \qedhere
\end{step}
\end{proof}
\begin{remark}[Finite propagation speed]
\begin{enumerate}
\item[]
\item Note that in the proof of \Cref{result:scaling_limit} (ii) the identity \eqref{eq:finite_time_horizon_identitiy} does not hold in the non-parametric situation. Even if the outer edges of the associated light cones do not reach the boundary of the domain $\Lambda_{\delta}$, the associated PDEs are only identical in the limit as $\delta \rightarrow 0$ in general as $\vartheta(\delta \cdot)$ approximates $\vartheta(0)$ but is never quite identical to it provided that $\vartheta$ is not locally constant at zero.
\item In the parametric case $\vartheta \equiv \vartheta_0>0$, \Cref{result:scaling_limit} (ii) implies that the Hellinger-distance between the laws of processes $(Z_{\delta}(t,z), t \in [0,T])$ and $(\underline{Z}(t, z), t \in [0,T])$ is zero for $\delta \in (0, \delta_T)$. Note that $\delta_T>0$ depends on the time horizon. Thus, the result is no longer accessible if, after rescaling, the time horizon depends on the resolution level $\delta$ itself. Thus, it is important to also understand the energetic behaviour of the rescaled trigonometric operator families as time increases, which will be the topic of the next section.
\end{enumerate}
\end{remark}

\section{Asymptotic energetic behaviour}
\label{section:asymptotic_energetic_behaviour}
This section is devoted to understanding the energetic behaviour of the stochastic wave equation under rescaling. 
The remaining proofs and more technical results are postponed to \Cref{section:energy_results}. \par
Consider a complex Hilbert space $(\mathscr{H}, \langle \cdot, \cdot\rangle_{\mathscr{H}})$ carrying a self-adjoint, negative, linear unbounded operator $(\mathscr{A}, D(\mathscr{A}))$.  \citet{goldsteinAsymptoticPropertySolutions1969} studied the energetic behaviour of abstract wave equations of the form
\begin{equation}
\label{eq:abstract_cauchy}
w''(t)=\mathscr{A}w(t), \quad w(0)=w_0 \in D(\mathscr{A}), \quad \dot{w}(0)= w_1 \in D((-\mathscr{A})^{1/2}).
\end{equation}
The abstract Cauchy problem \eqref{eq:abstract_cauchy} is well-posed, and the total energy
\begin{equation}
\label{eq:equipartition_of_energy}
\mathscr{E}_{\mathscr{A}} \equiv P(t) + K(t) \coloneqq \Vert (-\mathscr{A})^{1/2}w(t) \Vert^{2}_{\mathscr{H}} + \Vert \dot{w}(t) \Vert^{2}_{\mathscr{H}}, \quad t \in \mathbb{R}. 
\end{equation}
is preserved and does not depend on time $t \in \mathbb{R}$. \citet{goldsteinAsymptoticPropertySolutions1969} discovered that the asymptotic equipartition of energy
\begin{equation}
\label{eq:abstract_equipartition}
\lim_{|t| \rightarrow \infty} P(t)= \lim_{|t| \rightarrow \infty} K(t)=\frac{\mathscr{E}_{\mathscr{A}}}{2}
\end{equation}
holds if $\mathscr{A}$ satisfies the Riemann-Lebesgue property defined through
\begin{equation}
\label{eq:riemann_lebesgue}
\langle e^{it(-\mathscr{A})^{1/2}}z_1, z_2\rangle_{\mathscr{H}} \rightarrow 0, \quad |t|\rightarrow \infty, \quad z_1, z_2 \in \mathscr{H},
\end{equation}
where the convergence is exactly the convergence $e^{it(-\mathscr{A})^{1/2}} \xrightarrow{w} 0$ in the weak operator topology as $|t| \rightarrow \infty$. Thus, asymptotically, the kinetic and potential energy contribute equally to the entire energy within the system. Operators satisfying condition \eqref{eq:riemann_lebesgue} are called Riemann-Lebesgue operators.
\citet{goldsteinAsymptoticPropertySolutions1970} noticed that assumption \eqref{eq:riemann_lebesgue} is a condition on the reso\-lution of the identity $(E(\lambda), \lambda \in \mathbb{R})$ of $\mathscr{A}$, which lies strictly between continuity and absolute continuity of ${E}(\cdot)$. Thus, an operator like the Laplace operator on the unbounded spatial domain, which has a fully absolutely continuous spectrum as defined in \citet*[Chapter 9.1]{schmudgenUnboundedSelfadjointOperators2012}, is a Riemann-Lebesgue operator. \par The concept of Riemann-Lebesgue operators is not specific to the abstract wave equation and also extends to a large class of energy-preserving and hyperbolic equations like the abstract Schrödinger equation, see 
for instance \citet*{goldsteinAbstractEquipartitionEnergy1979, goldsteinEquipartitionEnergyHigher1982,sandefurAsymptoticEquipartitionEnergy1983,picardAsymptoticEquipartitionEnergy1987,picardAsymptoticPartitionEnergy1991}. By \citet{goldsteinAutocorrelationEquipartitionEnergy1982}, the equipartition of energy \eqref{eq:equipartition_of_energy} also holds in the sense of Cèsaro limits and for auto-correlations.  \citet*{marcellodabbiccoEquipartitionEnergyNonautonomous2021} and \citet*{bilerPartitionEnergyStrongly1990} employed similar techniques in the analysis of the behaviour of energies within abstract wave equations with certain types of damping.
\begin{remark}
Depending on the context, sometimes the Riemann-Lebesgue operator property \eqref{eq:riemann_lebesgue} is defined directly using the operator $e^{it\mathscr{A}}$ and not using the root $(-\mathscr{A})^{1/2}$, see for instance \citet{goldsteinAbstractEquipartitionEnergy1979}. We will use the above convention because we consider the stochastic wave equation, and the operator root turns up naturally.
\end{remark}
\noindent Consider the Laplace operator $\Delta$ on $L^{2}(\mathbb{R}^{d})$ with the domain $D(\Delta)=H^{2}(\mathbb{R}^{d})$. Using the Riemann-Lebesgue lemma, we show that the Laplace operator $\Delta$ is a Riemann-Lebesgue operator. 
\begin{lemma}[The Laplace operator is a Riemann-Lebesgue operator]
  \label{result:riemann_lebesgue_lemma}
We have $e^{it(-\Delta)^{1/2}} \xrightarrow{w} 0$ as $|t|\rightarrow \infty$, i.e.,
\begin{equation}
  \label{eq:riemann_lebesque_laplace_operator}
  \langle e^{it(-\Delta)^{1/2}}z_1, z_2\rangle_{L^{2}(\mathbb{R}^{d})} \rightarrow 0, \quad |t| \rightarrow \infty, \quad z_1,z_2 \in L^{2}(\mathbb{R}^{d}).
  \end{equation}
  \end{lemma}
  \begin{proof}[Proof of \Cref{result:riemann_lebesgue_lemma}] 
  Since $\sigma(-\Delta)=[0, \infty)$ we observe with $g_t(x)=e^{it\sqrt{x}}$, that $e^{it(-\Delta)^{1/2}}z = \mathcal{F}^{-1}( M_{g_t \circ p}\mathcal{F}(z))$, where $p(x)=|x|^{2}$ and 
  \begin{equation*}
     M_{g_t \circ p}[\tilde{z}](\fv)\coloneqq (g_t \circ p)(\fv) \tilde{z}(\fv), \quad \fv \in \mathbb{R}^{d}, \quad \tilde{z}\in L^{2}(\mathbb{R}^{d}),
  \end{equation*}
  is the standard multiplication operator; see \citet[Proposition 8.2]{schmudgenUnboundedSelfadjointOperators2012}. By polarisation, it is sufficient to prove the result for $z = z_1 = z_2 \in L^{2}(\mathbb{R}^{d})$. Using Plancherel's identity, we obtain
  \begin{equation}
  \begin{aligned}
\langle e^{it(-\Delta)^{1/2}}z,z \rangle_{L^{2}(\mathbb{R}^{d})}&=(2\pi)^{d}\langle M_{g_t \circ p}\mathcal{F}(z), \mathcal{F}(z) \rangle_{L^{2}(\mathbb{R}^{d})} \\&= \int_{\mathbb{R}^{d}} e^{it|\fv|} |\mathcal{F}(z)|^{2}(\fv) \mathrm{d}\fv \rightarrow 0, \quad |t|\rightarrow \infty,
  \end{aligned}
  \end{equation}
  where the last convergence follows from the generalised Riemann-Lebesgue Lemma in \citet{kahaneGeneralizationsRiemannLebesgueCantorLebesgue1980} as $|\mathcal{F}(z)|^{2}(\cdot) \in L^{1}(\mathbb{R}^{d})$. 
\end{proof}
\begin{remark}[Spectrum of Riemann-Lebesgue operators]
\begin{enumerate}
\item[]
\item We have already mentioned that for a self-adjoint operator to be a Riemann Lebesgue operator, it is sufficient that it has a purely absolutely continuous spectrum. Any self-adjoint differential operator on the unbounded domain $\mathbb{R}^{d}$ that can be diagonalised using the Fourier transform has a purely absolutely continuous spectrum; see \citet*[Corollary 9.4]{schmudgenUnboundedSelfadjointOperators2012}. However, any Riemann-Lebesgue operator must have at least a fully continuous spectrum. The Laplace operator on a bounded spatial domain $\Lambda$ naturally has eigenvalues. At these eigenvalues, the resolution of the identity has jumps and is discontinuous. Therefore, the spectrum is not fully continuous, and the weak convergence cannot hold for all elements of $L^{2}(\Lambda)$. Thus, the Laplace operator on a bounded spatial domain is not a Riemann-Lebesgue operator in the sense of \eqref{eq:riemann_lebesgue}. See also the proof of (iii) in \citet{goldsteinAsymptoticPropertySolutions1969}.
\item Even if a generator $\mathscr{A}$, for example, the Laplace operator on a bounded domain, is not a Riemann-Lebesgue operator, the equipartition of energy may still hold in the Cesaro sense
\begin{equation*}
\lim_{|T| \rightarrow \infty}\frac{1}{T}\int_{0}^{T}P(t)\mathrm{d}t = \lim_{|T|\rightarrow \infty} \frac{1}{T}\int_{0}^{T}K(t)\mathrm{d}t = \frac{\mathscr{E}_{\mathscr{A}}}{2},
\end{equation*}
see for instance \citet{goldsteinAsymptoticPropertySolutions1970}. This is intuitive as the potential and kinetic energy may oscillate and may never converge, while these oscillations do not interfere with a Cesaro limit.  
\end{enumerate}
\end{remark}
Since the Laplace operator on $L^{2}(\mathbb{R}^{d})$ with domain $D(\Delta)=H^{2}(\mathbb{R}^{d})$ is a Riemann-Lebesgue operator, \citet{goldsteinAsymptoticPropertySolutions1969} shows for $\mathscr{A}=\Delta$ in \eqref{eq:abstract_cauchy} and \eqref{eq:abstract_equipartition} that we have the asymptotic equipartition of energy
\begin{equation}
\label{eq:asymptotic_equipartition_laplace}
\lim_{|t|\rightarrow \infty}\Vert (-\Delta)^{1/2}w(t) \Vert^{2}_{L^{2}(\mathbb{R}^{d})}=\lim_{|t|\rightarrow \infty}\Vert w'(t) \Vert^{2}_{L^{2}(\mathbb{R}^{d})}=\frac{\mathscr{E}_{\Delta}}{2}.
\end{equation}
\begin{remark}[Different notions of energy]
\label{remark:notions_of_energy}
In \eqref{eq:equipartition_of_energy}, the total energy (in $\mathscr{H}$) within the system of interest is preserved and can be decomposed into the total potential energy $\Vert (-\mathscr{A})^{1/2}w_0 \Vert_{\mathscr{H}}^{2}$ and the total kinetic energy $\Vert w_1 \Vert_{\mathscr{H}}^{2}$. If zero belongs to the resolvent set of the operator $\mathscr{A}$, one can also define the energy in the dual space $D((-\mathscr{A})^{1/2})^{*}$:
\begin{equation}
\label{eq:modified_energy}
\tilde{\mathscr{E}}_{\mathscr{A}} = \Vert w(t) \Vert^{2}_{\mathscr{H}} + \Vert (-\mathscr{A})^{-1/2}\dot{w}(t) \Vert^{2}_{\mathscr{H}} \equiv \tilde{P}(t) + \tilde{K}(t),
\end{equation}
which also does not depend on $t \in \mathbb{R}$ and leads to the same type of asymptotic equipartition of energy.
\end{remark}
By Euler's formula, the function $z \mapsto e^{iz}$ can be represented using the cosine and the sine functions, leading in turn to representations of the sine and cosine in terms of the exponential function. The functional calculus then reveals a natural relation between operator cosine and sine functions and the Riemann-Lebesgue property  \eqref{eq:riemann_lebesgue}, given by
\begin{equation}
\label{eq:eulers_formula}
\cos(t (-\mathscr{A})^{1/2})= \frac{e^{it(-\mathscr{A})^{1/2}} + e^{-it(-\mathscr{A})^{1/2}}}{2}, \quad t \geq 0,
\end{equation}
and 
\begin{equation*}
(-\mathscr{A})^{-1/2}\sin(t (-\mathscr{A})^{1/2})= \frac{(-\mathscr{A})^{-1/2}(e^{it(-\mathscr{A})^{1/2}} - e^{-it(-\mathscr{A})^{1/2}})}{2i}, \quad t \geq 0.
\end{equation*}
This realisation is fundamental to the proof of the equipartition result presented in \citet{goldsteinAsymptoticPropertySolutions1969} and is also essential in this work. \par 
The rescaled versions of the operator cosine and sine functions characterised by \Cref{result:rescaling_trigonometric_operator_families} incorporate the scaling $\delta^{-1}$ in time and are associated with the rescaled generator $A_{\vartheta,\delta}$. Therefore, it seems natural to expect an energetic behaviour similar to \eqref{eq:riemann_lebesgue} and \eqref{eq:asymptotic_equipartition_laplace} for the trigonometric families acting on rescaled functions. To this end, we show the following asymptotic version of the Riemann-Lebesgue property. 
\begin{proposition}[Approximate Riemann-Lebesgue property]
\label{result:approximate_riemann_lebesgue} 
For any $z_1, z_2 \in L^{2}(\mathbb{R}^{d})$, we have 
\begin{equation*}
\langle e^{i \delta^{-1}t(-A_{\vartheta,\delta})^{1/2}}P_{\delta} z_1, P_{\delta} z_2 \rangle_{L^{2}(\Lambda_{\delta})} \rightarrow 0, \quad \delta \rightarrow 0, \quad t \in \mathbb{R},
\end{equation*}
where the orthogonal projection $P_{\delta}$ is defined through \eqref{eq:orthogonal_projection}.
\end{proposition}
\begin{proof}[Proof of \Cref{result:approximate_riemann_lebesgue}] The result is proved on page \pageref{proof:asymptotic_riemann_lebesgue}.
\end{proof}
\begin{remark}[Strongly continuous groups associated with varying domains]
\label{remark:behaviour_strongly_continuous_groups}
Analysing strongly continuous families of operators associated with varying domains leads to many technical challenges, which were solved by \citet{altmeyerNonparametricEstimationLinear2021} using the Feynman-Kac theorem. Given a fixed Hilbert space, the Trotter-Kato approximation theorem, for instance \citet[Theorem 4.8]{engelOneparameterSemigroupsLinear2000}, characterises the relation between the convergence of semigroups, generators and resolvents. In \citet*{itoTrotterKatoTheoremApproximation1998} and  \citet*[Chapter 4]{itoEvolutionEquationsApproximations2002}, a Trotter-Kato approximation theorem is proved for varying Banach spaces. These results can also be applied to the convergence of semigroups associated with varying domains. Indeed, \citet*{itoTrotterKatoTheoremApproximation1998} show that the convergence of semigroups associated with a varying spatial domain is equivalent to an adapted version of the strong resolvent convergence. \citet*[Proposition 3.1]{itoEvolutionEquationsApproximations2002} and \citet[Theorem 1]{weidmannStrongOperatorConvergence1997} provide concrete conditions under which this resolvent convergence holds.
\end{remark}
Now that we have shown an asymptotic version of the Riemann-Lebesgue property, we are ready to analyse the asymptotic behaviour of the operator cosine and sine under rescaling. In our situation, the analysis is more involved as the spatial domain $\Lambda_\delta$ also grows in parallel to the temporal increase as $\delta \rightarrow 0$, and the asymptotic Riemann-Lebesgue property \Cref{result:approximate_riemann_lebesgue} is required to prove an asymptotic characterisation of the energetic behaviour within the wave equation. The following result is an instructive example of how to extend the results by \citet{goldsteinAsymptoticPropertySolutions1969} to varying domains. 
\begin{proposition}[Asymptotic energetic behaviour for the scaled cosine]
\label{result:presentable_equipartition_of_energy}
Consider $z_1, z_2 \in L^{2}(\mathbb{R}^{d})$. Then, we have
\begin{equation}
\label{eq:presentable_equipartition_of_energy}
\langle C_{\vartheta, \delta}(\delta^{-1}t) P_{\delta} z_1, C_{\vartheta, \delta}(\delta^{-1}t) P_{\delta} z_2 \rangle_{L^{2}(\Lambda_{\delta})} \rightarrow \frac{1}{2}\langle z_1, z_2 \rangle_{L^{2}(\mathbb{R}^{d})}, \quad \delta \rightarrow 0, \quad t \in \mathbb{R},
\end{equation}
and
\begin{equation*}
\langle C_{\vartheta,\delta}(\delta^{-1}t) P_{\delta} z, C_{\vartheta,\delta}(\delta^{-1}s) P_{\delta} z\rangle_{L^{2}(\Lambda_{\delta})} \rightarrow 0, \quad \delta \rightarrow 0, \quad t \neq s \in \mathbb{R}\setminus \{0\},
\end{equation*}
where the orthogonal projection $P_{\delta}$ is defined through \eqref{eq:orthogonal_projection}.
\end{proposition}
\begin{proof}[Proof of \Cref{result:presentable_equipartition_of_energy}] 
We abbreviate by $R_{\delta}(t)=e^{it(-A_{\vartheta,\delta})^{1/2}}$ the unitary group on the complex Hilbert space $L^{2}(\Lambda_{\delta})$ for $t \in \mathbb{R}$. As in \eqref{eq:eulers_formula}, the operator cosine can be represented using $R_\delta(t)$ through
\begin{equation*}
C_{\vartheta, \delta}(\delta^{-1} t)P_{\delta}z=\frac{1}{2}(R_{\delta}(\delta^{-1}t) + R_{\delta}(-\delta^{-1}t))P_{\delta}z, \quad t \in \mathbb{R}.
\end{equation*}
Let $t,s \in \mathbb{R}\setminus \{0\}$ be arbitrary. Then, we observe
\begin{equation*}
\begin{aligned}
&\langle C_{\vartheta,\delta}(\delta^{-1}t) P_{\delta} z_1, C_{\vartheta,\delta}(\delta^{-1}s) P_{\delta} z_2\rangle_{L^{2}(\Lambda_{\delta})}\\
&= \frac{1}{4} \langle (R_{\delta}(\delta^{-1}t) + R_{\delta}(-\delta^{-1}t))P_{\delta}z_1, (R_{\delta}(\delta^{-1}s) + R_{\delta}(-\delta^{-1}s))P_{\delta}z_2\rangle_{L^{2}(\Lambda_{\delta})}\\
&=\frac{1}{4} \langle (R_{\delta}(\delta^{-1}s) + R_{\delta}(-\delta^{-1}s))^{*} \circ(R_{\delta}(\delta^{-1}t) + R_{\delta}(-\delta^{-1}t))P_{\delta}z_1, P_{\delta}z_2\rangle_{L^{2}(\Lambda_{\delta})}\\
&= \frac{1}{4}\langle (R_{\delta}(-\delta^{-1}s) + R_{\delta}(\delta^{-1}s)) \circ(R_{\delta}(\delta^{-1}t) + R_{\delta}(-\delta^{-1}t))P_{\delta}z_1, P_{\delta}z_2\rangle_{L^{2}(\Lambda_{\delta})}\\
&= \frac{1}{4}\Big( \langle R_{\delta}(\delta^{-1}(t-s))P_{\delta}z_1, P_{\delta}z_2\rangle_{L^{2}(\Lambda_{\delta})} + \langle R_{\delta}(\delta^{-1}(t+s))P_{\delta}z_1, P_{\delta}z_2\rangle_{L^{2}(\Lambda_{\delta})} \\
&\quad + \langle R_{\delta}(-\delta^{-1}(t+s))P_{\delta}z_1, P_{\delta}z_2\rangle_{L^{2}(\Lambda_{\delta})}+ \langle R_{\delta}(\delta^{-1}(s-t))P_{\delta}z_1, P_{\delta}z_2\rangle_{L^{2}(\Lambda_{\delta})}\Big). 
\end{aligned}
\end{equation*}
Notice that for $t \neq s$, the group $R_{\delta}(\cdot)$ is evaluated at non-zero times. Hence, by \Cref{result:approximate_riemann_lebesgue}, each summand converges to zero as $\delta \rightarrow 0$. If $t = s$, we have
\begin{equation*}
\begin{aligned}
&\langle C_{\vartheta,\delta}(\delta^{-1}t) P_{\delta} z_1, C_{\vartheta,\delta}(\delta^{-1}s) P_{\delta} z_2\rangle_{L^{2}(\Lambda_{\delta})} 
= \frac{1}{2} \langle P_{\delta}z_1, P_{\delta}z_2\rangle_{L^{2}(\Lambda_\delta)} \\
&\quad + \langle R_{\delta}(-2t\delta^{-1})P_{\delta}z_1, P_{\delta}z_2\rangle_{L^{2}(\Lambda_{\delta})} + \langle R_{\delta}(2t\delta^{-1})P_{\delta}z_1, P_{\delta}z_2\rangle_{L^{2}(\Lambda_{\delta})} \\&\rightarrow \frac{1}{2}\langle z_1, z_2\rangle_{L^{2}(\mathbb{R}^{d})}, \quad \delta \rightarrow 0, 
\end{aligned}
\end{equation*}
Thus, we have shown both of the desired convergences. 
\end{proof}
Similar results to \Cref{result:presentable_equipartition_of_energy} for the operator cosine are proved for the operator sine in \Cref{section:energy_results}. \par
The covariance structure of the stochastic wave equation is characterised using the operator cosine and sine functions, see \Cref{result:covariance_structure}.
The following example illustrates how the Riemann-Lebesgue lemma and the asymptotic behaviour of energy within the associated deterministic equations emerge naturally when analysing the covariance structure of the stochastic wave equation tested against localised functions. 
\begin{example}[Wave equation on the unbounded domain]
\label{example:unbounded_domain_main_example}
For some $\vartheta_0>0$ and a fixed time horizon $T>0$, consider the one-dimensional stochastic wave equation
\begin{equation}
\label{SOL1}
\partial^{2}_{tt}\underline{u}(t,x)=\vartheta_0 \partial_{xx}^{2}\underline{u}(t,x)+\mathscr{W}(t,x), \quad t \in (0,T], \quad x \in \mathbb{R},
\end{equation}
with zero initial conditions, driven by space-time white noise $\mathscr{W}$ on a complete filtered probability space $(\Omega, \mathcal{F}, (\mathcal{F}_t)_{t \in [0,T]}, \mathbb{P})$. We interpret the driving white noise as a centred Gaussian random field $(\mathscr{W}(A), A \in \mathcal{B}_b([0,T]\times \mathbb{R}))$ with the covariance structure $\mathbb{E}[\mathscr{W}(A)\mathscr{W}(B)]=\lambda_{[0,T]\times \mathbb{R}}(A \cap B)$, where $\mathcal{B}_b([0,T]\times \mathbb{R})$ denotes the bounded Borel subsets  and $\lambda_{[0,T]\times \mathbb{R}}$ the Lebesgue measure of $[0,T]\times \mathbb{R}$. In view of Walsh \cite{walshIntroductionStochasticPartial1986}
a continuous mild solution to \eqref{SOL1} exists as an $(\mathcal{F}_t)_{t \in [0,T]}$-predictable random field given by 
\begin{equation*}
\underline{u}(t,x)=\int_{0}^{t}\int_{\mathbb{R}}G_{t-s}(x-y)\mathscr{W}(\mathrm{d}s,\mathrm{d}y), \quad t \in[0,T], \quad x \in \mathbb{R},
\end{equation*}
where 
\begin{equation}
\label{eq:green_function}
G_{t}(x)\coloneqq \frac{1}{2\sqrt{\vartheta_0}}\mathbbm{1}(|x| \leq \sqrt{\vartheta_0} t), \quad x \in \mathbb{R},\quad  t \in [0,T],
\end{equation}
is the fundamental solution (Green's function) to the associated deterministic problem. The Fourier transform of the Green's function \eqref{eq:green_function} is given by
\begin{equation*}
\mathcal{F}(G_t)(\fv)=\frac{\sin(t \sqrt{\vartheta_0} |\fv|)}{\sqrt{\vartheta_0}|\fv|}, \quad t \in [0,T], \quad \fv \in \mathbb{R}\setminus \{0\}.
\end{equation*}
For some kernel $K \in C_c^{\infty}(\mathbb{R})$, we have
\begin{align}
&\mathbb{E}[\langle \underline{u}(t), \delta^{-2}(K'')_{\delta} \rangle_{L^{2}(\mathbb{R})} \langle \underline{u}(s),\delta^{-2}(K'')_{\delta}\rangle_{L^{2}(\mathbb{R})}] \nonumber\\
&= \delta^{-4}\int_{0}^{t \land s} \int_{\mathbb{R}} \langle G_{t-r}(\cdot - y), (K'')_{\delta} \rangle_{L^{2}(\mathbb{R})}\langle G_{s-r}(\cdot - y), (K'')_{\delta} \rangle_{L^{2}(\mathbb{R})} \mathrm{d}y \mathrm{d}r \nonumber\\
&=  \frac{\delta^{-4}}{2\pi}\int_{0}^{t \land s} \int_{\mathbb{R}} \frac{\sin((t-r)\sqrt{\vartheta_0}|\fv|)\sin((s-r)\sqrt{\vartheta_0}|\fv|)}{\vartheta_0|\fv|^{2}}|\mathcal{F}((K'')_{\delta})|^{2}(\fv) \mathrm{d}\fv\mathrm{d}r \nonumber\\
&= \label{eq:fourier_inner_integral} \frac{\delta^{-2}}{2 \pi \vartheta_0}\int_{0}^{t \land s}\int_{\mathbb{R}} \sin(\delta^{-1}(t-r)\sqrt{\vartheta_0}|\fv|)\sin(\delta^{-1}(s-r)\sqrt{\vartheta_0}|\fv|) |\mathcal{F}(K')|^{2}(\fv)\mathrm{d}\fv \mathrm{d}r
\end{align}
where we have used the stochastic Fubini theorem and Plancherel's identity. Note that rescaling in this way via the Fourier transform and the transformation theorem is analogous to \Cref{example:parametric_rescaling}. Note that $|\mathcal{F}(K')|^{2} \in L^{1}(\mathbb{R}^{d})$. We observe
\begin{equation}
\label{eq:post_fubini_expression}
\begin{aligned}
\int_{0}^{t \land s} &\sin(\delta^{-1}(t-r)\sqrt{\vartheta_0}|\fv|)\sin(\delta^{-1}(s-r)\sqrt{\vartheta_0}|\fv|)\mathrm{d}r \\ 
&= \frac{(t \land s)}{2}\cos(\delta ^{-1}\left(t-s\right) \sqrt{\vartheta_0}\left|\fv\right|)\\
&+ \delta\dfrac{\sin\left(\delta^{-1}|t-s|\sqrt{\vartheta_0}\left|\fv\right|\right)-\sin\left(\delta^{-1}\left(t+s\right)\sqrt{\vartheta_0}\left|\fv\right|\right)}{4 \sqrt{\vartheta_0}\left|\fv\right|}. \\
\end{aligned}
\end{equation}
Thus, if $t\neq s$ the generalised Riemann-Lebesgue lemma in \citet{kahaneGeneralizationsRiemannLebesgueCantorLebesgue1980} applied to the inner integral in \eqref{eq:fourier_inner_integral}, shows 
\begin{equation*}
\delta^{2}\mathbb{E}[\langle \underline{u}(t), \delta^{-2}(K'')_{\delta} \rangle_{L^{2}(\mathbb{R})} \langle \underline{u}(s),\delta^{-2}(K'')_{\delta}\rangle_{L^{2}(\mathbb{R})}] \rightarrow 0.
\end{equation*}
Otherwise, if $t=s$ Fubini's theorem yields with \eqref{eq:post_fubini_expression}:
\begin{equation*}
\delta^{2}\mathbb{E}[\langle \underline{u}(t), \delta^{-2}(K'')_{\delta} \rangle_{L^{2}(\mathbb{R})} \langle \underline{u}(s),\delta^{-2}(K'')_{\delta}\rangle_{L^{2}(\mathbb{R})}]  \rightarrow \frac{t}{2 \vartheta_0} \Vert K' \Vert_{L^{2}(\mathbb{R})}^{2}. 
\end{equation*}
Consequently, the observed Fisher information satisfies
\begin{equation}
\label{eq:exp_obs_fisher_unbounded}
\delta^{2}\mathbb{E}\left(\int_{0}^{T}\langle \underline{u}(t), \delta^{-2}(K'')_{\delta} \rangle_{L^{2}(\mathbb{R})}^{2}\mathrm{d}t\right) \rightarrow \frac{T^{2}\Vert K' \Vert_{L^{2}(\mathbb{R})}^{2}}{4 \vartheta_0}, \quad \delta \rightarrow 0,
\end{equation}
and 
\begin{equation}
\label{eq:var_obs_fisher_unbounded}
\mathrm{Var}\left(\delta^{2}\int_{0}^{T}\langle \underline{u}(t), \delta^{-2}(K'')_{\delta} \rangle_{L^{2}(\mathbb{R})}^{2}\mathrm{d}t \right)\rightarrow 0, \quad \delta \rightarrow 0.
\end{equation}
The convergences \eqref{eq:exp_obs_fisher_unbounded} and \eqref{eq:var_obs_fisher_unbounded} are already sufficient for obtaining a central limit theorem for an augmented MLE in the parametric case on an unbounded domain. 
\end{example}
In the next section, we formalise the estimation procedure and extend \Cref{example:unbounded_domain_main_example} to a bounded spatial domain and the spatially varying wave-speed $\vartheta$. 

\section{Estimating the wave speed}
\label{section:asymptotics_augmented_MLE_main}
Based on \citet{altmeyerNonparametricEstimationLinear2021}, we will study an augmented maximum likelihood estimator $\hat{\vartheta}_{\delta}$, which was adapted to the setting of the stochastic wave equation. In contrast to the stochastic heat equation, the augmented MLE will involve both local measurements of the amplitude and the velocity. \par
We begin by stating the following regularity assumption. 
\begin{assumptionKernel}{z}
\label{asskernel}
Suppose that $z \in H^{2}(\mathbb{R}^{d})$ has compact support in $\Lambda_\delta$  for some $\delta>0$ and satisfies $\int_{\mathbb{R}^{d}}z(x)\mathrm{d}x=0$ for $d=1$. 
\end{assumptionKernel}
Fix some kernel $K$ satisfying \asskernel{}. By inserting the localised kernel $K_{\delta}$ into the dynamic representation of the weak solution from \Cref{result:Gaussian_process_solution}, we obtain the real-valued processes $(u_{\delta}(t), t \in [0,T])$,  $(u_{\delta}^{\Delta}(t), t \in [0,T])$ and $(v_{\delta}(t), t \in [0,T])$, called local measurements:
\begin{equation}
\label{eq:local_measurements_definition}
\begin{aligned}
u_\delta(t)&\coloneqq \langle u(t),K_{\delta} \rangle_{L^{2}(\Lambda)}, \quad v_{\delta}(t)\coloneqq \langle v(t), K_{\delta}\rangle,\\
u^{\Delta}_{\delta}(t)&\coloneqq \langle u(t),\delta^{-2}(\Delta K)_{\delta} \rangle_{L^{2}(\Lambda)}, \quad t \in [0,T].
\end{aligned}
\end{equation}
We observe $u_{\delta}(t)$ continuously in time $t \in [0,T]$ for a known kernel $K$ chosen by the statistician. By \Cref{result:Gaussian_process_solution}, the local measurements \eqref{eq:local_measurements_definition} satisfy the dynamic
\begin{equation}
\label{eq:dynamic_equations_local_measurements}
u_{\delta}(t)=\int_{0}^{t}v_\delta(s)\mathrm{d}s, \quad v_\delta(t)=\int_{0}^{t}\langle u(s),A_{\vartheta}K_{\delta} \rangle_{L^{2}(\Lambda)}\mathrm{d}s+ \Vert K \Vert_{L^{2}(\mathbb{R}^{d})} \overline{W}(t),
\end{equation}
where $(\overline{W}(t), t \in [0,T])$ defined through $\overline{W}(t)=\langle W(t), K_{\delta}\rangle_{L^{2}(\Lambda)}/ \Vert K_\delta \Vert_{L^{2}(\Lambda)}$ is a scalar Brownian motion and $t \in [0,T]$. 
\begin{remark}
Note that the local observations of the velocity, i.e.\ $(v_{\delta}(t), t \in [0,T])$ do not need to be observed as $(v_{\delta}(t), t \in [0,T])$ can be recovered as a limit of the corresponding difference quotient of $(u_{\delta}(t),t \in [0,T])$ because $u_\delta(t)=\int_{0}^{t}v_\delta(s)\mathrm{d}s$. Furthermore, we can approximate the measurements $(u_{\delta}^{\Delta}(t), t \in [0,T])$ based on $(u_{\delta}(t),t \in [0,T])$ observed in a neighbourhood of zero, as discussed in the introduction of the local observations in \citet*{altmeyerNonparametricEstimationLinear2021}.
\end{remark} 
In the deterministic parametric situation $A_{\vartheta}=\vartheta \Delta$ of \eqref{eq:dynamic_equations_local_measurements} without any noise, the parameter could be recovered through the local measurements satisfying $\dot{v}_{\delta}(t)=\vartheta u_{\delta}^{\Delta}(t)$ for $t \in [0,T]$. Thus, in the situation with noise, a least squares ansatz suggests the heuristic minimisation problem
\begin{equation*}
\hat{\vartheta}_{\delta}=\mathrm{argmin}_{\vartheta} \int_{0}^{T}(\dot{v}_{\delta}(t)- \vartheta u_{\delta}^{\Delta}(t))^{2}\mathrm{d}t,
\end{equation*}
leading via the corresponding normal equations to the estimator
\begin{equation}
\label{eq:augmented_MLE}
\hat{\vartheta}_\delta \coloneqq \frac{\int_{0}^{T}u_\delta^{\Delta}(t)\mathrm{d}v_{\delta}(t)}{\int_{0}^{T}(u_{\delta}^{\Delta}(t))^{2}\mathrm{d}t}, \quad \delta > 0.
\end{equation}
The estimator can also be motivated by the Girsanov-type arguments as outlined in \citet[Section 4.1]{altmeyerNonparametricEstimationLinear2021} and will therefore also be called augmented MLE. Using the dynamic representation \eqref{eq:dynamic_equations_local_measurements}, the error decomposition for the augmented MLE is given by
\begin{equation}
\label{eq:error_decomposition}
\hat{\vartheta}_{\delta} - \vartheta(0) = \Vert K \Vert_{L^{2}(\mathbb{R}^{d})} {I}_{\delta}^{-1}{M}_{\delta}+ I_{\delta}^{-1}R_{\delta},
\end{equation}
where we introduce the following notations.
\begin{itemize}
\item \textit{Observed Fisher information}: We define the observed Fisher information as
\begin{equation}
\label{eq:observed_fisher_information}
I_{\delta} \coloneqq \int_{0}^{T}u_{\delta}^{\Delta}(t)^{2}\mathrm{d}t, \quad \delta>0.
\end{equation}
\item \textit{Martingale part}: The martingale part of $\hat{\vartheta}_{\delta}$ is given by
\begin{equation*}
M_{\delta}\coloneqq \int_{0}^{T}u_{\delta}^{\Delta}(t) \mathrm{d}\overline{W}(t), \quad \delta>0.
\end{equation*}
\item \textit{Remaining bias:} The remaining bias is given by
\begin{equation}
\label{eq:remaining_bias}
R_{\delta} \coloneqq \int_{0}^{T}u_{\delta}^{\Delta}(t)\langle u(t),(A_{\vartheta}-\vartheta(0)\Delta)K_{\delta} \rangle_{L^{2}(\Lambda)} \mathrm{d}t. 
\end{equation}
\end{itemize}
Next, we show that the observed Fisher information and the remaining bias converge to deterministic constants. All propositions are proved in \Cref{section:asymptotics_augmentedMLE}.
\begin{proposition}[Asymptotics for the observed Fisher information]
\label{result:asymptotics_observed_fisher_information}
Grant Assumptions \asskernel{} and \assini{}. The expectation and covariance of the observed Fisher information satisfy
\begin{align*}
\mathbb{E}[\delta^{2}I_{\delta}] &\rightarrow \frac{T^{2} }{4 \vartheta(0)}\Vert \nabla K \Vert_{L^{2}(\mathbb{R}^{d})}^{2}, \quad
\quad\mathrm{Var}(\delta^{2}I_{\delta}) \rightarrow 0, \quad \delta \rightarrow 0. 
\end{align*}
\end{proposition}
\begin{proof}[Proof of \Cref{result:asymptotics_observed_fisher_information}] The result is proved on page \pageref{proof:obs_fisher}.
\end{proof}
\begin{remark}
Consider the second-order abstract Cauchy problem
\begin{equation*}
\label{eq:abstract_cauchy_problem}
w(t)=A_{\vartheta,\delta}w(t), \quad  w(0)=0, \quad \dot{w}(0)=\Delta K, \quad t \in \mathbb{R}.
\end{equation*}
As the first-order initial condition is zero, the associated total energy in $H_0^{-1}(\Lambda)$ from \Cref{remark:notions_of_energy} corresponds to the total kinetic energy in $H_0^{-1}(\Lambda)$, satisfying
\begin{equation*}
\tilde{\mathscr{E}}_{A_{\vartheta,\delta}}= \Vert (-A_{\vartheta,\delta})^{-1/2} \Delta K \Vert^{2}_{L^{2}(\Lambda_{\delta})} \rightarrow \frac{1}{\vartheta(0)} \Vert \nabla K \Vert^{2}_{L^{2}(\mathbb{R}^{d})}, \quad \delta \rightarrow 0,
\end{equation*}
where the convergence follows from \Cref{result:fractional_and_inverse_limits} as $K$ satisfies \asskernel{}. 
Thus, the limiting expectation of the observed Fisher information scaled by $\delta^{2}$ is proportional to the limiting total kinetic energy in $H^{-1}(\mathbb{R}^{d})$ within the Cauchy problem \eqref{eq:abstract_cauchy_problem}.
\end{remark}
\begin{proposition}[Asymptotics for the remaining bias]
\label{result:remaining_bias}
Grant Assumptions \asskernel{} and \assini{}. The remaining bias satisfies
\begin{equation}
\label{eq:asymptotic_bias}
\delta^{-1}(I_{\delta})^{-1}R_{\delta} \xrightarrow{\mathbb{P}}  \frac{\langle \nabla K, \nabla\beta^{(0)}\rangle_{L^{2}(\mathbb{R}^{d})}}{\Vert \nabla K \Vert^{2}_{L^{2}(\mathbb{R}^{d})}}, \quad \delta \rightarrow 0,
\end{equation}
where $\beta^{(0)}$ is defined through \eqref{eq:beta_delta}.
\end{proposition}
\begin{proof}[Proof of \Cref{result:remaining_bias}] The result is proved on page \pageref{proof:bias}. 
\end{proof}
\noindent Note that in the parametric case $\vartheta \equiv \vartheta_0 > 0$, the bias term is just zero, and the asymptotic bias derived through \eqref{eq:asymptotic_bias} is also zero. 
\begin{theorem}[Asymptotic normality of the augmented MLE] Grant Assumptions \asskernel{} and  \assini{}. Then, the augmented MLE \eqref{eq:augmented_MLE} satisfies
\label{result:asymptotic_normality}
\begin{equation*}
\delta^{-1}(\hat{\vartheta}_\delta - \vartheta(0)) \xrightarrow{d} \mathcal{N}\left(  \frac{\langle \nabla K, \nabla \beta^{(0)}\rangle_{L^{2}(\mathbb{R}^{d})}}{\Vert \nabla K \Vert^{2}_{L^{2}(\mathbb{R}^{d})}}, \frac{ 4 \vartheta(0)\Vert K \Vert^{2}_{L^{2}(\mathbb{R}^{d})}}{T^{2} \Vert \nabla K \Vert^{2}_{L^{2}(\mathbb{R}^{d})}} \right), \quad \delta \rightarrow 0,
\end{equation*}
where $\beta^{(0)}$ is defined through \eqref{eq:beta_delta}.
\end{theorem}
\begin{remark}[Rate of convergence]
The rate of convergence of order $\delta$ is the same as in the case of the stochastic heat equation obtained by \citet[Proposition 5.2]{altmeyerNonparametricEstimationLinear2021}.  In the spectral observation scheme based on the first $N$-Fourier modes of the solution process, the rate of convergence is $N^{-3/2}$ for the maximum likelihood estimator analysed by \citet[Theorem 3.1]{liuEstimatingSpeedDamping2008}. Heuristically, the convergence rate $\delta$ of the augmented MLE $\hat{\vartheta}_{\delta}$ can be regarded as a single-mode version of the maximum likelihood estimator in the spectral case. 
\end{remark}
\begin{proof}[Proof of \Cref{result:asymptotic_normality}]
In view of \eqref{eq:error_decomposition}, we begin by introducing the error decomposition 
\begin{align*}
\delta^{-1}\left(\hat{\vartheta}_{\delta}- \vartheta(0) \right) &= \delta^{-1}\Vert K \Vert_{L^{2}(\mathbb{R}^{d})} ({I}_{\delta})^{-1}{M}_{\delta} + \delta^{-1} I_{\delta}^{-1}R_{\delta}\\
&= \left(\frac{{M}_{\delta}}{\sqrt{\mathbb{E}[{I}_{\delta}]}}\right)\left(\frac{{I}_{\delta}}{\mathbb{E}[{I}_{\delta}]}\right)^{-1}\left( \delta^{2}\mathbb{E}[{I}_{\delta}]\right)^{-1/2}\Vert K \Vert_{L^{2}(\mathbb{R}^{d})}+ \delta^{-1} I_{\delta}^{-1}R_{\delta}.
\end{align*}
The quadratic variation of the martingale $Y^{(\delta)}\coloneqq M_{\delta}/({\mathbb{E}[I_{\delta}]})^{1/2}$ is given by $\langle Y^{(\delta)} \rangle = I_{\delta}/\mathbb{E}[I_{\delta}]$. The standard continuous martingale central limit theorem, see for instance \citet*[Theorem A.1]{pasemannParameterEstimationSemilinear2021}, shows that $Y^{(\delta)} \xrightarrow{d} \mathcal{N}(0,1)$ if $\langle Y^{(\delta)}\rangle \xrightarrow{\mathbb{P}} 1$ as $\delta \rightarrow 0$. In view of \Cref{result:asymptotics_observed_fisher_information} and Chebyshev's inequality we obtain
$I_{\delta}/\mathbb{E}[I_{\delta}]\xrightarrow{\mathbb{P}} 1$ and thus $Y^{(\delta)} \xrightarrow{d} \mathcal{N}(0,1)$ as $\delta \rightarrow 0$. The result follows with \Cref{result:remaining_bias} and using Slutsky's theorem.
\end{proof}
\begin{remark}[Dependence of the asymptotic variance on the time horizon]
\label{remark:dependence_on_time_horizon}
In the case of the stochastic heat equation, \citet[Proposition 5.2]{altmeyerNonparametricEstimationLinear2021} shows that increasing the time horizon $T$ will lead to a decrease of the asymptotic variance of the limiting normal distribution in the associated central limit theorem of order $T^{-1}$. This effect is even more prevalent for the stochastic wave equation, as the asymptotic variance scales with $T^{-2}$, which is also visible in the spectral approach. Indeed, the asymptotic variances from \citet[Theorem 1.1]{lototskyStatisticalInferenceStochastic2009} and \citet[Theorem 3.1]{liuEstimatingSpeedDamping2008}  depend on the time horizon through the factors $2T^{-1}$ and $4T^{-2}$, respectively. \par The rate of convergence of the MLE in the case of the ordinary Ornstein-Uhlenbeck process or  the harmonic oscillator is $\sqrt{T}$ in the ergodic and $T$ in the energetically stable case; see \citet*[Proposition 3.46]{kutoyantsStatisticalInferenceErgodic2004} and \citet*{linUndampedHarmonicOscillator2011}. The rate of estimation associated with the Fourier modes of the corresponding SPDE is then inherited by the augmented MLE through the dependence of asymptotic variance on the time horizon.
\end{remark}
\noindent Up to a constant, the asymptotic bias is determined by 
\begin{equation}
\label{eq:partial_integration_property}
\langle \nabla K, \nabla \beta^{(0)} \rangle_{L^{2}(\mathbb{R}^{d})} = -\langle \Delta K, \beta^{(0)}\rangle_{L^{2}(\mathbb{R}^{d})},
\end{equation}
where $\beta^{(0)}$ is defined through \eqref{eq:beta_delta}.
Suppose that $K \in H^{4}(\mathbb{R}^{d})$. Then, in view of \citet[Lemma A.3]{altmeyerNonparametricEstimationLinear2021}, \eqref{eq:partial_integration_property} can be rewritten as 
\begin{equation*}
-\langle \Delta K, \beta^{(0)}\rangle_{L^{2}(\mathbb{R}^{d})} = \langle \langle \nabla \vartheta (0), x\rangle_{\mathbb{R}^{d}}, |\nabla \Delta K(x)|^{2}_{\mathbb{R}^{d}}\rangle_{L^{2}(\mathbb{R}^{d})}.
\end{equation*}
If $\nabla \Delta K$ is symmetric, i.e.\ $| \nabla \Delta K (-x)|_{\mathbb{R}^{d}} = |\nabla \Delta K(x)|_{\mathbb{R}^{d}}$ for $x \in \mathbb{R}^{d}$, then the asymptotic bias vanishes. 
Note that the asymptotic bias is different in the case of the heat equation and involves the term $- \frac{1}{2}\langle K, \beta^{(0)}\rangle_{L^{2}(\mathbb{R}^{d})}$ and not \eqref{eq:partial_integration_property}. As described by \citet[Lemma A.3]{altmeyerNonparametricEstimationLinear2021}, this leads to the requirement that $\nabla K$ is symmetric in contrast to our assumption that $\nabla \Delta K$ is symmetric. 
\begin{corollary}[Confidence interval]
\label{result:confidence_interval}
Assume that the asymptotic bias is zero in the setting of \Cref{result:asymptotic_normality}. For some $\overline{\alpha} \in (0,1)$ the confidence interval around $\vartheta(0)$, given by
\begin{equation*}
I_{1-\overline{\alpha}} \coloneqq \left[\hat{\vartheta}_{\delta} - \delta \sqrt{\hat{\vartheta}_{\delta}} \frac{2 \Vert K \Vert_{L^{2}(\mathbb{R}^{d})}}{T \Vert \nabla K \Vert_{L^{2}(\mathbb{R}^{d})}}q_{1-\overline{\alpha}/2}, \hat{\vartheta}_{\delta} + \delta \sqrt{\hat{\vartheta}_{\delta}} \frac{2 \Vert K \Vert_{L^{2}(\mathbb{R}^{d})}}{T \Vert \nabla K \Vert_{L^{2}(\mathbb{R}^{d})}}q_{1-\overline{\alpha}/2}\right],
\end{equation*}
with the standard normal $(1-\overline{\alpha}/2)$-quantile $q_{1-\overline{\alpha}/2}$, has asymptotic coverage $1-\overline{\alpha}$ for $\delta \rightarrow 0$. 
\end{corollary}
\begin{proof}[Proof of \Cref{result:confidence_interval}]
By \Cref{result:asymptotic_normality} we not only obtain the asymptotic normality but also $\hat{\vartheta}_{\delta} \xrightarrow{\mathbb{P}} \vartheta(0)$. Thus, we may apply Slutsky's lemma and obtain
\begin{equation*}
\delta^{-1}\left(\hat{\vartheta}_{\delta} \frac{4 \Vert K \Vert^{2}_{L^{2}(\mathbb{R}^{d})}}{T^{2} \Vert \nabla K \Vert^{2}_{L^{2}(\mathbb{R}^{d})}}\right)^{-1/2}(\hat{\vartheta}_{\delta}- \vartheta(0)) \xrightarrow{d} \mathcal{N}(0,1), \quad \delta \rightarrow 0,
\end{equation*}
as we have assumed that the asymptotic bias is zero. Consequently, we obtain 
\begin{equation*}
\mathbb{P}(\vartheta(0) \in I_{1-\overline{\alpha}}) \rightarrow 1- \overline{\alpha}, \quad \delta \rightarrow 0. \qedhere
\end{equation*}
\end{proof}
\section{Numerical illustration}
\label{section:numerical_example}
\begin{figure}[t]
    \centering
    \subfloat[]{\label{subfigure:heatmap_1}\includegraphics[width=\textwidth]{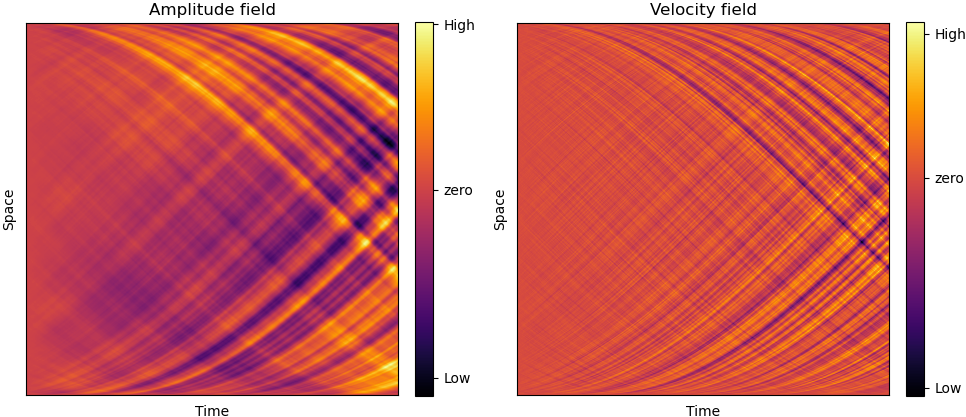}}\\
    \subfloat[]{\label{subfigure:heatmap_2}\includegraphics[width=\textwidth]{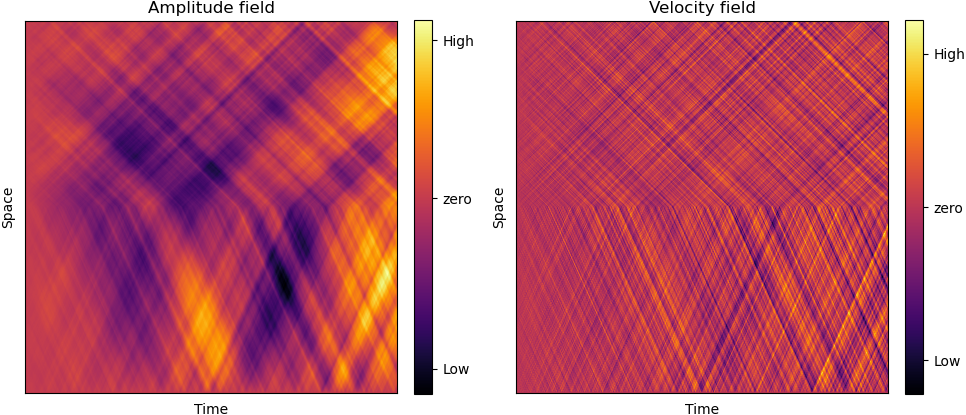}}
    \caption{Simulation of the amplitude and velocity field of \eqref{eq:1d_stochastic_wave} with spatially dependent wave speed \eqref{eq:vartheta_a} in (a) and \eqref{eq:vartheta_b} in (b)}
    \label{figure:heatmaps}
\end{figure}
This section is dedicated to the illustration of the main results. We consider the stochastic wave equations on the 1-dimensional bounded domain $\Lambda = (0,1)$ up to the time horizon $T=1$:
\begin{equation}
\label{eq:1d_stochastic_wave}
\partial_{tt}^{2} u(t,x)= \partial_x(\vartheta(x)\partial_x u(t,x)) + \dot{W}(t,x), \quad t \in [0,1], \quad x \in (0,1). 
\end{equation}
Unless stated otherwise, we assume zero first and second-order initial conditions. Based on the results of \citet[Section 10.5]{lordIntroductionComputationalStochastic2014} and the work of \citet{quer-sardanyonsSpaceSemiDiscretisationsStochastic2006}, we employ a semi-implicit Euler scheme with a finite difference approximation of the second spatial derivative on the uniform grid
\begin{equation*}
\label{eq:resolution}
    \left\{ (t_k, y_j) \colon t_k = k/N, y_j = j/M, \quad k=0, \dots, N, j=0,\dots, M\right\},
\end{equation*}
where the spatial and temporal resolutions are $M=10^3$ and $N=M^2$, respectively. \par
\textit{Smooth wave speed}: In \Cref{subfigure:heatmap_1}, we simulate the stochastic wave equation \eqref{eq:1d_stochastic_wave} with the spatially dependent wave speed
\begin{equation}
\label{eq:vartheta_a}
\vartheta_a (x)=4x(1-x)+0.01, \quad x \in (0,1).
\end{equation}
\begin{figure}
\subfloat[]{ \label{subfigure:loglog}
\includegraphics[scale=0.35]{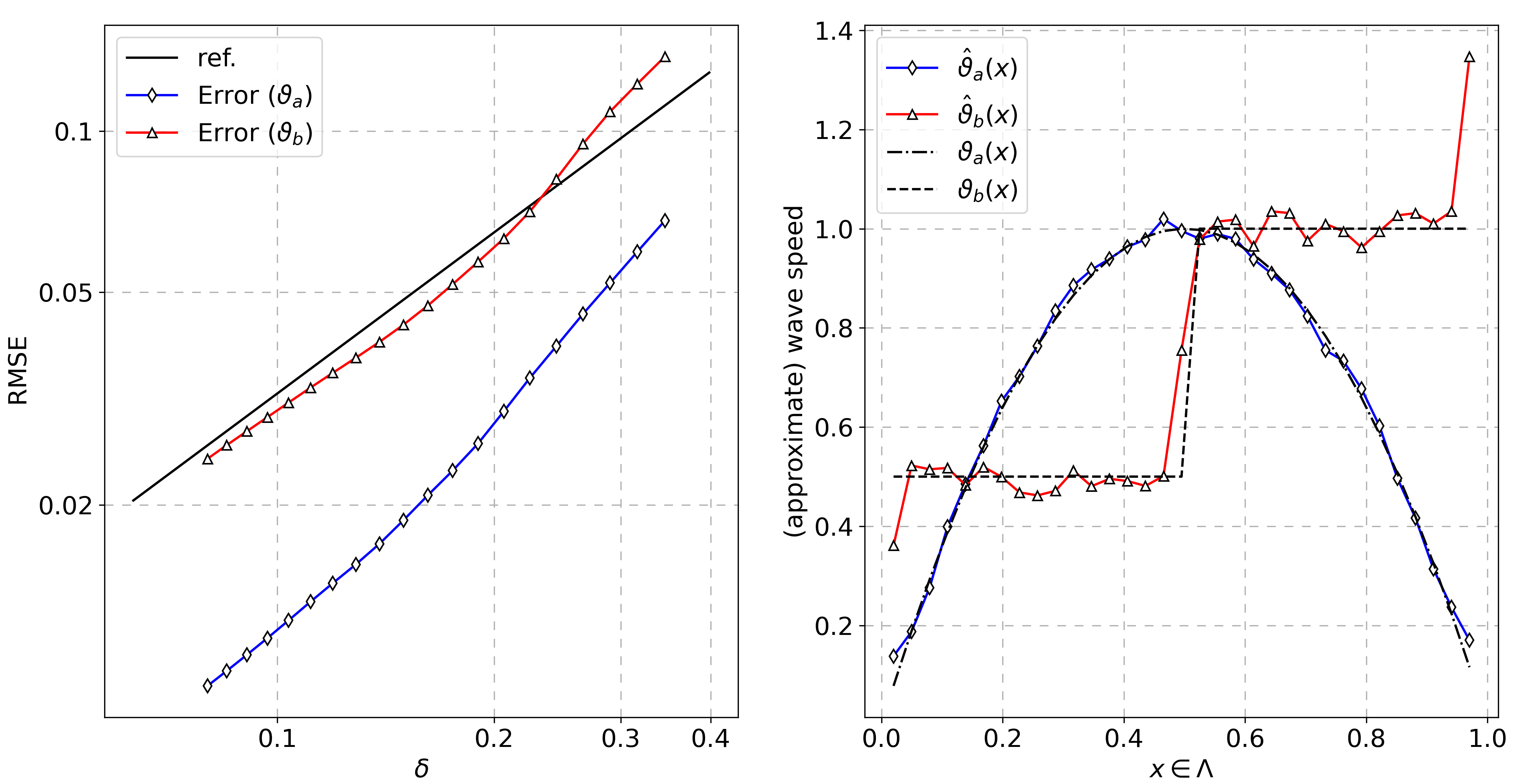}
}\\
\subfloat[]{\label{fig:montecarlo_time}\includegraphics[scale=0.35]{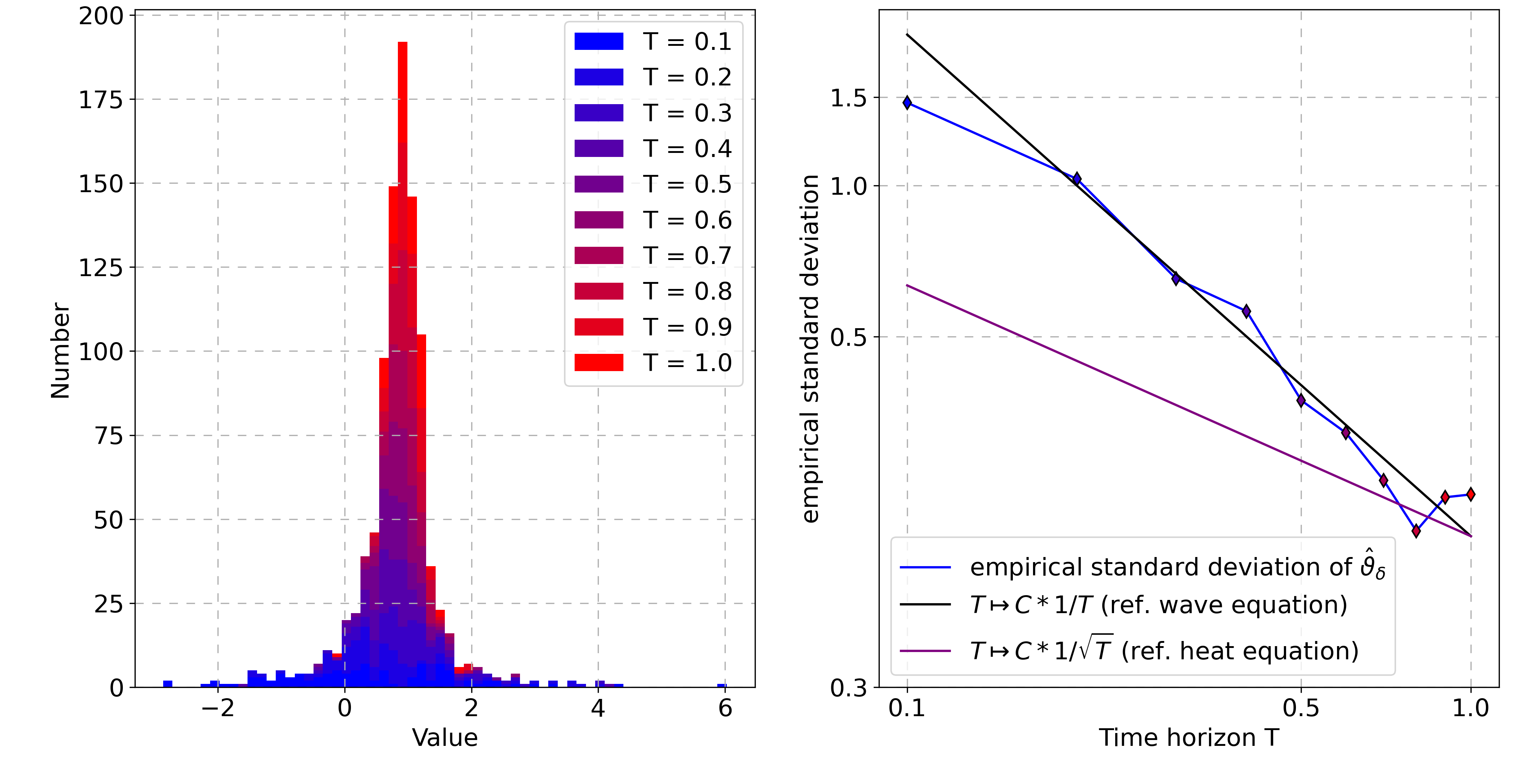} }
     \caption{(a) Monte-Carlo simulation of the augmented MLE. (left) $\log_{10}$-$\log_{10}$ plot of root mean squared estimation error at $x_0 = 0.6$. (b)(left) histogram plot of the estimator for different values of $T$. (right) Monte-Carlo simulation of the asymptotic variance associated with augmented MLE.}
     \label{figure:montecarlo_plots}
\end{figure}
In this case, the wave speed of the stochastic wave equation is increasing towards the centre of the interval. The amplitude field is much smoother than the velocity field as the white noise enters the stochastic wave equation through the velocity field. As both the first and second-order initial conditions are zero, energy is added to the system through the noise, and the observed pattern results from accumulating fluctuations.
\par \textit{Piecewise constant wave speed}: In \Cref{subfigure:heatmap_2}, we simulate the stochastic wave equation \eqref{eq:1d_stochastic_wave} with parameters
\begin{equation}
\label{eq:vartheta_b}
\vartheta_b (x)= \frac{1}{2}\mathbbm{1}_{(0, 1/2]}(x)+ \mathbbm{1}_{(1/2,1)}(x), \quad x \in (0,1).
\end{equation}
This situation arises if a wave travels from one medium (air) into another (water) and back again. The wave is partially transmitted as it passes from one medium into the other medium. Indeed, as the wave speed in both media is different, we observe a change of angle at the interface reminiscent of Snell's law. This case is not covered by the theory as $\vartheta_b$ is not differentiable in $1/2$. Still, it can be considered an intriguing limiting case reminiscent of the situation of change point detection discussed in \citet*{reissChangePointEstimation2023}.
\par As in \citet*{altmeyerNonparametricEstimationLinear2021}, we consider the kernels $K=\varphi'''$ based on the smooth bump function:
\begin{equation}
\varphi (x)=\exp \left(\frac{-12}{(1-x^2)}\right), \quad x \in (-1,1).
\end{equation}
For $\delta \in (0.05,0.5)$ and $x_0 \in (0,1)$, we can empirically approximate the local measurements $u^{\Delta}_{\delta,x_0}$ and $v_{\delta,x_0}$ based on the kernel $K$ and compute the associated augmented MLE accordingly. \par
\Cref{subfigure:loglog} provides a $\log_{10}$-$\log_{10}$ plot of the root mean squared estimation error as $\delta \rightarrow 0$ based on 1000 Monte-Carlo runs. The convergence rate of $\delta$ is achieved for both $\vartheta_a$ and $\vartheta_b$. On the other hand, the (right) side of \Cref{subfigure:loglog} displays the true wave speed compared to its approximation.
When approaching the discontinuity in the wave speed $\vartheta_b$, the accuracy of the estimation shrinks significantly. \par \Cref{fig:montecarlo_time} visualises in the case of \eqref{eq:vartheta_a} the asymptotic normality result of \Cref{result:asymptotic_normality} and the dependence of the asymptotic variance on the time horizon. The right-hand side of \Cref{fig:montecarlo_time} displays the empirical standard deviation of the asymptotic distribution approximated based on 500 Monte-Carlo runs for each displayed time horizon. In contrast, the left-hand side shows how the asymptotic normal distribution concentrates around the true value as the associated time horizon increases. As expected by \Cref{result:asymptotic_normality} and \Cref{remark:dependence_on_time_horizon}, the asymptotic variance in the case of the augmented MLE depends differs from the parabolic case and depends on the time horizon through ${1}/{T^{2}}$ and not through ${1}/{T}$ as in the case of the stochastic heat equation. \par 
We conclude this section with the following table, which illustrates the asymptotic coverage of the confidence interval $I_{1-\overline{\alpha}}$ derived in \Cref{result:confidence_interval} for different values of $\alpha$ and $\delta$ based on $500$ Monte-Carlo simulations.
\begin{center}
\begin{tabular}{ |p{2cm}|p{2cm}|p{2cm}|p{2cm}|  }
\hline
\multicolumn{4}{|c|}{Empirical probability of $\vartheta(0) \in I_{1-\overline{\alpha}}$} \\
\hline
/ & $\delta \geq 0.2$ & $\delta = 0.1$  & $\delta =0.09$ \\
\hline
$\overline{\alpha} = 0.1$ & $\approx 1$ & 0.93 & 0.89\\
\hline
$\overline{\alpha} = 0.05$ & $\approx 1$ & 0.98 & 0.97\\
\hline
\end{tabular}
\end{center}
As $\delta$ decreases, the observations are in line with the converge $\mathbb{P}(\vartheta(0) \in I_{1-\overline{\alpha}}) \rightarrow 1-\overline{\alpha}$ as $\delta \rightarrow 0$. For larger values of $\delta$, the confidence intervals have a greater size, resulting in larger empirical coverage probabilities. 

\section*{Acknowledgment}
This research has been partially funded by Deutsche Forschungsgemeinschaft (DFG)—SFB1294/1-318763901.
\newpage
\begin{appendix}
  \section{Remaining proofs}
  \subsection{Well-posedness proofs}
\label{section:well-posedness}
Based on page 415 ff. in \citet{arendtVectorvaluedLaplaceTransforms2001}, the following result summarises some crucial properties of the isometric isomorphism $A_{\vartheta}$ defined through \eqref{eq:weak_laplace}. 
\begin{lemma}[Some Gelfand properties]
\label{result:gelfand_properties}
\begin{enumerate}
\item[]
\item The triple $(H_{0}^{1}(\Lambda), L^{2}(\Lambda), H_0^{-1}(\Lambda))$ is a Gelfand triple, i.e. the inclusions $H_{0}^{1}(\Lambda) \hookrightarrow L^{2}(\Lambda)$ and $L^{2}(\Lambda) \hookrightarrow H_0^{-1}(\Lambda)$ are bounded linear operators. 
\item For $z_2 \in L^{2}(\Lambda)$ and $z_1 \in H_0^{1}(\Lambda)$, we have
\begin{equation*}
\langle A_{\vartheta}z_1, z_2\rangle_{H_0^{-1}(\Lambda)}=-\langle z_2, z_1\rangle_{L^{2}(\Lambda)}.
\end{equation*}
\item For $z \in L^{2}(\Lambda)$ and $l \in H_0^{-1}(\Lambda)$ we have 
\begin{equation}
\label{eq:dual_pairing_property_2}
\langle l, z\rangle_{H_0^{-1}(\Lambda)}= -\langle A_{\vartheta}^{-1}l, z\rangle_{H_0^{-1}(\Lambda), H_0^{1}(\Lambda)} = -\langle z, A_{\vartheta}^{-1}l\rangle_{L^{2}(\Lambda)}.
\end{equation}
\item For $l \in H_0^{-1}(\Lambda)$ and $z \in H_0^{1}(\Lambda)$, we have
\begin{equation*}
\langle l,-A_{\vartheta}z \rangle_{H_{0}^{-1}(\Lambda)}=\langle l, z\rangle_{H_{0}^{-1}(\Lambda), H_0^{1}(\Lambda)}.
\end{equation*}
\end{enumerate}
\end{lemma}
\begin{proof}[Proof of \Cref{result:gelfand_properties}] The result is proved in \citet*[415 ff. and Proposition 7.1.5]{arendtVectorvaluedLaplaceTransforms2001}.
\end{proof}
\begin{lemma}[Properties of the wave generator]    \label{lemma:adjoint_generator_wave_equation}
    The generator $\mathcal{A}_{\vartheta}$ is skew-adjoint, i.e.\ $\mathcal{A}^{*}_{\vartheta}=-\mathcal{A}_{\vartheta}$ and $D(\mathcal{A}^{*}_{\vartheta})=D(\mathcal{A}_{\vartheta})$.
    \end{lemma}
    \begin{proof}[Proof of \Cref{lemma:adjoint_generator_wave_equation}] 
    Let $U=(z_1,z_2)^{\top} \in H_{0}^{1}(\Lambda) \times L^{2}(\Lambda) $ and $\tilde{U}=(\tilde{z}_1,\tilde{z}_2)^{\top} \in H_{0}^{1}(\Lambda) \times L^{2}(\Lambda)$. By applying \Cref{result:gelfand_properties} (ii), we immediately obtain
    \begin{align*}
    \langle \mathcal{A}_{\vartheta}U,\tilde{U} \rangle_{L^{2}(\Lambda)\times H^{-1}_0(\Lambda)} &=\langle z_2, \tilde{z}_1\rangle_{L^{2}(\Lambda)}+\langle A_{\vartheta} z_1, {\tilde{z}_2}\rangle_{H^{-1}_0(\Lambda)} \\
    &= -\langle z_2, A_{\vartheta}\tilde{z}_1\rangle_{H_0^{-1}(\Lambda)}-\langle z_1,\tilde{z}_2 \rangle_{L^{2}(\Lambda)} \\
    &= \langle U,-\mathcal{A}_{\vartheta} \tilde{U}\rangle_{L^{2}(\Lambda)\times H_0^{-1}(\Lambda)}.
    \end{align*}
    Thus, $\mathcal{A}^{*}_{\vartheta}=-\mathcal{A}_{\vartheta}$ and $D(\mathcal{A}^{*}_{\vartheta})=H_{0}^{1}(\Lambda) \times L^{2}(\Lambda)$.
\end{proof}

\begin{lemma}[Unitary group]
\label{result:unitary_group}
The group $(\mathcal{J}_{\vartheta}(t), t \in \mathbb{R})$ is unitary and satisfies
\begin{equation*}
\mathcal{J}_{\vartheta}(t)^{*} = \mathcal{J}_{\vartheta}(-t) = \begin{pmatrix}
C_{\vartheta}(t) & -S_{\vartheta}(t)\\
-C'_{\vartheta}(t) & C_{\vartheta}(t)
\end{pmatrix}, \quad t \in \mathbb{R}.
\end{equation*}
\end{lemma}
\begin{proof}[Proof of \Cref{result:unitary_group}]
We have already shown in \Cref{lemma:adjoint_generator_wave_equation} that the generator $\mathcal{A}_{\vartheta}$ is skew-adjoint. Thus, we conclude that $(\mathcal{J}_{\vartheta}(t), t \in \mathbb{R})$ is unitary by \citet*[Theorem 3.24]{engelOneparameterSemigroupsLinear2000}. As an immediate consequence, we observe
\begin{equation*}
\mathcal{J}_{\vartheta}(t)^{*}=\mathcal{J}_{\vartheta}(t)^{-1}=\mathcal{J}_{\vartheta}(-t), \quad t \in \mathbb{R}.
\end{equation*}
By definition, we have $C_{\vartheta}(t)=C_{\vartheta}(-t)$ for all $t \in \mathbb{R}$, see \citet*[p.209]{arendtVectorvaluedLaplaceTransforms2001}. With this and the transformation theorem, we also observe $S_{\vartheta}(-t)=-S_{\vartheta}(t)$ and $C'_{\vartheta}(-t)=A_{\vartheta}S_{\vartheta}(-t)=-A_{\vartheta}S_{\vartheta}(t)=-C'_{\vartheta}(t)$ for $t \in \mathbb{R}$. 
\end{proof}
\begin{lemma}[Adjoint of the component inclusion]
\label{result:adjoint_component_inclusion}
The adjoint of the operator $B:L^{2}(\Lambda) \rightarrow L^{2}(\Lambda)\times H_0^{-1}(\Lambda)$, defined through $Bu \coloneqq (0,u)^{\top}$, is given by $B^{*}(z, l)^{\top}\coloneqq -A_{\vartheta}^{-1}l$ for $(z, l)^{\top}\in L^{2}(\Lambda) \times H_0^{-1}(\Lambda)$.
\end{lemma}
\begin{proof}[Proof of \Cref{result:adjoint_component_inclusion}]
Given \Cref{result:gelfand_properties}, the operator $B:L^{2}(\Lambda) \rightarrow L^{2}(\Lambda)\times H_0^{-1}(\Lambda)$ defined through $z \mapsto (0,z)^{\top}$ satisfies 
    \begin{equation}
    \label{eq:adjoint_B}
    \begin{aligned}
    \langle Bz, U\rangle_{L^{2}(\Lambda) \times H_0^{-1}(\Lambda)} = \langle z, l\rangle_{H_0^{-1}(\Lambda)}=\langle z, -A_{\vartheta}^{-1}l\rangle_{L^{2}(\Lambda)}=\langle z, B^{*} U\rangle_{L^{2}(\Lambda)}
    \end{aligned}
    \end{equation}
    for $U = (\tilde{z},l)^{\top}\in L^{2}(\Lambda) \times H_0^{-1}(\Lambda)$ and $z \in L^{2}(\Lambda)$, where $B^{*}:L^{2}(\Lambda) \times H_0^{-1}(\Lambda) \rightarrow L^{2}(\Lambda)$ defined through $B^{*}(U_1, U_2)^{\top}\coloneqq -A_{\vartheta}^{-1}U_2$ is the adjoint of $B$. 
\end{proof}

\begin{proof}[Proof of \Cref{result:system_of_equations_from_weak_solution}]
\label{proof:system_of_equations_from_weak_solution}
By the definition of the weak solution, we have 
    \begin{equation*}
    \begin{aligned}
        &\langle X(t), U \rangle_{L^{2}(\Lambda)\times H^{-1}_0(\Lambda)}\\
        &=\int_{0}^{t}\langle X(s),\mathcal{A}^{*}_{\vartheta}U \rangle_{L^{2}(\Lambda) \times H^{-1}_0(\Lambda)}\mathrm{d}s+ \langle BW(t), U\rangle_{L^{2}(\Lambda)\times H^{-1}_0(\Lambda)},
    \end{aligned}
    \end{equation*}
    for $U \in L^{2}(\Lambda) \times H_0^{-1}(\Lambda)$ and $ t \in [0,T]$. For functions $z \in H_0^{1}(\Lambda)\cap H^{2}(\Lambda)$, we find $A_{\vartheta} z \in L^{2}(\Lambda) \subset H^{-1}_0(\Lambda)$. In view of \Cref{result:gelfand_properties} and by setting $U = (z, -A_{\vartheta}z) \in L^{2}(\Lambda) \times (L^{2}(\Lambda))^{'}$, we have
    \begin{align}
    \langle X(t), U \rangle_{L^{2}(\Lambda)\times H^{-1}_0(\Lambda)} &= \langle u(t),z \rangle_{L^{2}(\Lambda)}+ \langle v(t), -A_{\vartheta}z \rangle_{H_0^{-1}(\Lambda)} \nonumber\\
    &= \langle u(t),z \rangle_{L^{2}(\Lambda)} + \langle v(t), z \rangle_{H_0^{-1}(\Lambda), H_0^{1}(\Lambda)}, \nonumber\\
    \langle X(s),\mathcal{A}^{*}_{\vartheta}U \rangle_{L^{2}(\Lambda) \times H^{-1}_0(\Lambda)} &= \langle X(s), -\mathcal{A}_{\vartheta}(z, -A_{\vartheta}z)^{\top} \rangle_{L^{2}(\Lambda)\times H_0^{-1}(\Lambda)} \nonumber\\
    &= -\langle u(t), -A_{\vartheta}z \rangle_{L^{2}(\Lambda)}  - \langle v(t), A_{\vartheta}z  \rangle_{H_0^{-1}(\Lambda)}\nonumber\\
    &= \langle u(t),A_{\vartheta}z \rangle_{L^{2}(\Lambda)} + \langle v(t), z \rangle_{H^{-1}_0(\Lambda), H^{1}_0(\Lambda)}, \nonumber\\
    \label{eq:noise_representation}\langle BW(t), U\rangle_{L^{2}(\Lambda)\times H^{-1}_0(\Lambda)} &= \langle {W(t)}, -A_{\vartheta}z \rangle_{H_0^{-1}(\Lambda)}\\
    &= \langle W(t), z \rangle_{L^{2}(\Lambda)}. \nonumber
    \end{align}
    The result is obtained through linear separation using $U = (z, 0)^{\top}$ and $U = (0, -A_{\vartheta}z)^{\top}$, respectively.
\end{proof}
 
\begin{proof}[Proof of \Cref{result:Gaussian_process_solution}]
\label{proof:Gaussian_process_solution}
\begin{step}[The Gaussian process]
We define the Gaussian process $\mathcal{V}(t,U)$ parameterised by $t \in [0,T]$ and  $U \in L^{2}(\Lambda) \times H_0^{-1}(\Lambda)$ through
    \begin{equation}
    \label{eq:gaussian_process_joint_solution}
    \mathcal{V}(t,U)\coloneqq \int_{0}^{t}\langle \mathcal{J}^{*}_{\vartheta}(t-s)U,B \mathrm{d}W(s) \rangle_{L^{2}(\Lambda) \times H_0^{-1}(\Lambda)},
    \end{equation}
for $U \in L^{2}(\Lambda)\times H_0^{-1}(\Lambda)$ and $t \in [0,T]$.
By definition, $\mathcal{V}$ is a well-defined centred Gaussian process, and Itô's isometry (\citet[Proposition 4.28]{dapratoStochasticEquationsInfinite2014}) shows \eqref{eq:covariance_gaussian_process_solution}.
The proof of \citet*[Proposition 6.7]{hairerIntroductionStochasticPDEs2023} reveals that the process \eqref{eq:gaussian_process_joint_solution} emerges when passing from an almost surely integrable mild solution to a weak solution. In particular, the process satisfies the dynamic behaviour
\begin{equation}
\label{eq:weak_dynamic_gaussian_process_solution}
\mathcal{V}(t, U)=\int_{0}^{t} \mathcal{V}(s, \mathcal{A}^{*}_{\vartheta} U) \mathrm{d}s + \langle BW(t), U \rangle_{L^{2}(\Lambda) \times H_0^{-1}(\Lambda)}, \quad U \in D(\mathcal{A}_{\vartheta}^{*}).
\end{equation}
\end{step}
\begin{step}[Decomposition into component processes] Using the adjoint of $B$ and the unitary group $(\mathcal{J}(t),t \in [0,T])$ computed in \Cref{result:adjoint_component_inclusion} and \Cref{result:unitary_group}, we may rewrite the process \eqref{eq:gaussian_process_joint_solution} as
\begin{equation}
\label{eq:computations_for_the_representation_using_component_processes}
\begin{aligned}
\mathcal{V}(t,U) 
&= \int_{0}^{t}\langle \mathcal{J}^{*}_{\vartheta}(t-s)U,B\cdot \rangle_{L^{2}(\Lambda) \times H_0^{-1}(\Lambda)} \mathrm{d}W(s)\\
&=\int_{0}^{t}\langle B^{*}\mathcal{J}^{*}_{\vartheta}(t-s)U,\cdot \rangle_{L^{2}(\Lambda)} \mathrm{d}W(s)\\
&=\int_{0}^{t} \langle -A_{\vartheta}^{-1}(-C'_{\vartheta}(t-s)z + C_{\vartheta}(t-s)l) , \mathrm{d}W(s)\rangle_{L^{2}(\Lambda)}\\
&= \int_{0}^{t}\langle S_{\vartheta}(t-s)z, \mathrm{d}W(s) \rangle_{L^{2}(\Lambda)} + \int_{0}^{t}\langle -A_{\vartheta}^{-1}C_{\vartheta}(t-s) l, \mathrm{d}W(s) \rangle_{L^{2}(\Lambda)},
\end{aligned}
\end{equation}
for any $U = (z, l)^{\top} \in L^{2}(\Lambda) \times H_0^{-1}(\Lambda)$. Thus, for $z \in L^{2}(\Lambda)$, we set
\begin{equation*}
u_{\mathcal{V}} (t, z)\coloneqq \int_{0}^{t}\langle S_{\vartheta}(t-s)z, \mathrm{d}W(s)\rangle_{L^{2}(\Lambda)}, \quad v_{\mathcal{V}}(t, z)=\int_{0}^{t}\langle C_{\vartheta}(t-s)z, \mathrm{d}W(s) \rangle_{L^{2}(\Lambda)},
\end{equation*}
and obtain $\mathcal{V}(t, U)\coloneqq u_{\mathcal{V}}(t, z) + v_{\mathcal{V}}(t, -A_{\vartheta}^{-1}l)$ from \eqref{eq:computations_for_the_representation_using_component_processes}, as the operator cosine commutes with its generator. 
\end{step}
\begin{step}[Dynamic behaviour] By \Cref{lemma:adjoint_generator_wave_equation}, the generator of the unitary group $(\mathcal{J}_{\vartheta}(t), t \in \mathbb{R})$ is skew-adjoint. Thus, representing the weak dynamic \eqref{eq:weak_dynamic_gaussian_process_solution} using the processes $u_{\mathcal{V}}$ and $v_{\mathcal{V}}$, we obtain for $U=(z_1, z_2)^{\top} \in D(\mathcal{A}_{\vartheta})=D(A_{\vartheta}) \times L^{2}(\Lambda)$:
\begin{equation*}
\begin{aligned}
&u_{\mathcal{V}}(t, z_1) + u_{\mathcal{V}}(t, -A^{-1}_{\vartheta}z_2) \\
&= \int_{0}^{t} u_{\mathcal{V}}(s, -z_2)+ v_{\mathcal{V}}(s, z_1)\mathrm{d}s + \langle BW(t), U \rangle_{L^{2}(\Lambda) \times H_0^{-1}(\Lambda)}. 
\end{aligned}
\end{equation*}
Because of \eqref{eq:noise_representation}, setting $U=(z_1, z_2)^{\top} = (0, -A_{\vartheta}z)^{\top}$ and $U=(z_1, z_2)^{\top} = (z, 0)^{\top}$ for $z \in H_0^{1}(\Lambda) \cap H^{2}(\Lambda)$ yields \eqref{eq:system_of_equations_gaussian_process_solution_eq_1} and \eqref{eq:system_of_equations_gaussian_process_solution_eq_2}, respectively.\qedhere
\end{step}
\end{proof}
  \subsection{Proof of spectral asymptotics}
\label{section:spectral_asymptotics}
For an overview or introduction to the theory of spectral measures, we refer to \citet{schmudgenUnboundedSelfadjointOperators2012} and \citet{rudinFunctionalAnalysis1991}. By the spectral theorem for unbounded operators, see for instance \citet[Theorem 5.1]{schmudgenUnboundedSelfadjointOperators2012}, there exists a unique spectral measure $\tilde{E}_{\mathscr{A}}(\cdot)$ associated with an unbounded self-adjoint operator $(\mathscr{A}, D(\mathscr{A}))$ such that 
\begin{equation*}
\mathscr{A} = \int_{\mathbb{R}}\lambda \mathrm{d} E_{\mathscr{A}}(\lambda),
\end{equation*}
with the resolution of the identity $E_{\mathscr{A}}(\lambda)=\tilde{E}_{\mathscr{A}}((-\infty, \lambda])$. By the functional calculus for unbounded operators, we may then define the operator 
\begin{equation*}
f(\mathscr{A})=\int_{\mathbb{R}}f(\lambda) \mathrm{d}E_{\mathscr{A}}(\lambda), \quad D(f(\mathscr{A}))=\Big\{x \in \mathscr{H}\colon \int_{\mathbb{R}}|f(\lambda)|^{2} \mathrm{d} \langle E_{\mathscr{A}}(\lambda)x, x\rangle< \infty\Big\}. 
\end{equation*}
For an overview of the properties of the functional calculus for unbounded operators, see \citet[Theorem 5.9]{schmudgenUnboundedSelfadjointOperators2012} and \citet[Theorem 13.24]{rudinFunctionalAnalysis1991}. 
\par Abbreviate by $(E_\delta(\lambda), \lambda \geq 0)$ and $(E(\lambda), \lambda \geq 0)$ the resolution of identity associated with the operator $-A_{\vartheta,\delta}$ and $-\vartheta(0)\Delta$, respectively. Recall further the orthogonal projection $P_{\delta}:L^{2}(\mathbb{R}^{d}) \rightarrow L^{2}(\Lambda_{\delta})$ defined in \eqref{eq:orthogonal_projection}. 
\begin{remark}[Projections, subspaces and rescaling]
\label{remark:subspace_projections}
\begin{enumerate}
\item[]
\item We can identify the space $L^{2}(\Lambda)$ with the subspace of $L^{2}(\mathbb{R}^{d})$ consisting of those functions which vanish a.e. on the complement of $\Lambda$. Note that this is also possible for Sobolev spaces. Indeed, we can identify $H_0^{1}(\Lambda)$ with a subspace of $H^{1}(\mathbb{R}^{d})$ by extending functions by zero. For $z \in H_0^{1}(\Lambda)$, we define
\begin{equation*}
\tilde{z}(x) \coloneqq \begin{cases}
z(x) \quad &\text{ if } x \in \Lambda,\\
0 \quad &\text{ else }.
\end{cases}
\end{equation*}
Then, in view of \citet*[Chapter 5]{adamsSobolevSpaces2003}, we have $\tilde{z}\in H^{1}(\mathbb{R}^{d})$ and $\nabla \tilde{z}=\widetilde{\nabla z}$. As a consequence, the $\sim$ can be omitted throughout.
\item Notice that by the convexity of $\Lambda$, we have $\Lambda_{\delta'} \subset \Lambda_{\delta}$ for all $\delta \in (0, \delta')$. In particular, if some function $z \in L^{2}(\mathbb{R}^{d})$ has compact support in $\Lambda_{\delta'}$, its support is thus also contained in $\Lambda_{\delta}$ for $0 < \delta < \delta'$. In fact, we have $z \in L^{2}(\Lambda_{\delta})$ and $z_\delta \in L^{2}(\Lambda)$ for $0< \delta < \delta'$. Crucially, the fact that $P_{\delta}z = z$ for $0 < \delta < \delta'$ will allow us to simplify notation considerably as we are starting out with some function in $L^{2}(\mathbb{R}^{d})$, which is already contained in $L^{2}(\Lambda_{\delta})$ for $\delta$ sufficiently small.
\item For any $z \in L^{2}(\mathbb{R}^{d})$, we have
\begin{equation}
\label{eq:strong_convergence_projection}
\begin{aligned}
\Vert z-P_{\delta}z \Vert_{L^{2}(\mathbb{R}^{d})}^{2} &=\int_{\mathbb{R}^{d}}|z(x)-\mathbbm{1}_{\Lambda_{\delta}}(x)z(x)|^{2}\mathrm{d}x \\
&= \int_{\mathbb{R}^{d}\setminus \Lambda_{\delta}}|z(x)|^{2}\mathrm{d}x \rightarrow 0, \quad \delta \rightarrow 0.
\end{aligned}
\end{equation}
\end{enumerate}
\end{remark}
\begin{lemma}[A strong convergence for the resolution of identities]
\label{result:convergence:partition_of_unity}
We have for any $z \in L^{2}(\mathbb{R}^{d})$:
\begin{equation}
\label{eq:spectral_convergence_1}
\Vert E_{\delta}(\lambda)P_{\delta}z - E(\lambda)z \Vert^{2}_{L^{2}(\mathbb{R}^{d})} \rightarrow 0, \quad \delta \rightarrow 0, \quad \lambda \geq 0. 
\end{equation}
\end{lemma}
\begin{proof}[Proof of \Cref{result:convergence:partition_of_unity}] By \citet*[Theorem 2]{weidmannStrongOperatorConvergence1997}, the convergence of the resolutions of identity in \eqref{eq:spectral_convergence_1} follows from the strong resolvent convergence
\begin{equation}
\label{eq:strong_resolvent_convergence}
\Vert (-A_{\vartheta,\delta}-\lambda)^{-1}P_{\delta} z - (-\vartheta(0)\Delta - \lambda)^{-1}z\Vert^{2}_{L^{2}(\mathbb{R}^{d})} \rightarrow 0, \quad \lambda \in \mathbb{C}\setminus \mathbb{R}.
\end{equation}
A sufficient condition for \eqref{eq:strong_resolvent_convergence} is given by \citet*[Theorem 1]{weidmannStrongOperatorConvergence1997}. Indeed, consider some $\varphi \in D_0 = C_c^{\infty}(\mathbb{R}^{d})$ in the core of the Laplace operator. Then, there exists some $\delta_0(z)$ such that for any $0 < \delta < \delta_0(z)$, we have $\varphi \in H_0^{1}(\Lambda_\delta) \cap H^{2}(\Lambda_\delta)$ and $-A_{\vartheta,\delta}\varphi \rightarrow -\vartheta(0) \Delta \varphi$ in $L^{2}(\mathbb{R}^{d})$. In particular, the assumptions of \citet*[Theorem 1 and Theorem 2]{weidmannStrongOperatorConvergence1997} are satisfied and \eqref{eq:spectral_convergence_1} follows. 
\end{proof}

\begin{corollary}
\label{result:weak_convergence_partitions_of_unity}
For $z \in L^{2}(\mathbb{R}^{d})$, we have
\begin{equation}
\label{eq:weak_convergence_distribution}
\int_{0}^{\infty}g(\lambda)\mathrm{d}\langle E_\delta(\lambda)P_{\delta}z - E(\lambda)z, z \rangle_{L^{2}(\mathbb{R}^{d})} \rightarrow 0, \quad \delta \rightarrow 0, 
\end{equation}
for any bounded and continuous function $g: \mathbb{R} \rightarrow \mathbb{C}$.
\end{corollary}
\begin{proof}[Proof of \Cref{result:weak_convergence_partitions_of_unity}]
Consider the finite measures
\begin{equation}
\label{eq:spectral_measures_delta}
E_{z,\delta}(B) \coloneqq \langle E_{\delta}(B)P_{\delta}z, P_{\delta}z\rangle_{L^{2}(\Lambda_{\delta})}, \quad B \in \mathcal{B}(\mathbb{R}), \quad \delta \geq 0,
\end{equation}  
with the distribution functions $E_{z, \delta}(\lambda) = \langle E_{\delta}(\lambda)P_{\delta}z, P_{\delta}z\rangle_{L^{2}(\Lambda_{\delta})}$ for $\delta \geq 0$. Using \eqref{eq:spectral_convergence_1} and the fact that $\Delta$ has fully absolutely continuous spectrum, we observe for any $\lambda \geq 0$ that 
\begin{equation}
\label{eq:convergence_distribution_functions}
\begin{aligned}
|E_{z,\delta}(\lambda)-E_{z,0}(\lambda)| &= |\langle E_{\delta}(\lambda)P_{\delta}z - E(\lambda)z,P_{\delta }z \rangle_{L^{2}(\mathbb{R}^{d})}| \\
&\leq \Vert E_{\delta}(\lambda)P_{\delta}z - E(\lambda)z \Vert_{L^{2}(\mathbb{R}^{d})} \Vert z \Vert_{L^{2}(\mathbb{R}^{d})} \rightarrow 0, \quad \delta \rightarrow 0,
\end{aligned}
\end{equation}
where the convergence follows immediately from \Cref{result:convergence:partition_of_unity}. 
The convergence of the distribution function yields a weak convergence of the associated measures \eqref{eq:spectral_measures_delta}. For any bounded and continuous function $g:\mathbb{R} \rightarrow \mathbb{C}$ the convergence \eqref{eq:convergence_distribution_functions} implies \eqref{eq:weak_convergence_distribution}. 
\end{proof}
\begin{proof}[Proof of \Cref{result:approximate_riemann_lebesgue}]
\label{proof:asymptotic_riemann_lebesgue}
We wish to show that for any $z_1, z_2 \in L^{2}(\mathbb{R}^{d})$, we have 
\begin{equation*}
\langle e^{i \delta^{-1}t(-A_{\vartheta,\delta})^{1/2}}P_{\delta} z_1, P_{\delta} z_2 \rangle_{L^{2}(\Lambda_{\delta})} \rightarrow 0, \quad \delta \rightarrow 0, \quad t \in \mathbb{R},
\end{equation*}
where the orthogonal projection $P_{\delta}$ is defined through \eqref{eq:orthogonal_projection}.
By polarisation, we may assume that $z=z_1=z_2$. \par The assumptions of \Cref{result:convergence:partition_of_unity} are satisfied and \Cref{result:weak_convergence_partitions_of_unity} is applicable. With $|e^{i\delta^{-1}t \sqrt{\lambda}}| \leq 1$ and the fact that the Laplace operator is a Riemann-Lebesgue operator, c.f. \eqref{eq:riemann_lebesque_laplace_operator} in \Cref{result:riemann_lebesgue_lemma}, we observe 
\begin{align*}
\langle e^{i \delta^{-1}t(-A_{\vartheta,\delta})^{1/2}}P_{\delta}z,P_{\delta}z \rangle_{L^{2}(\Lambda_{\delta})} &= \int_{0}^{\infty}e^{i\delta^{-1}t \sqrt{\lambda}}\mathrm{d}\langle E_{\delta}(\lambda)P_{\delta}z, P_{\delta}z\rangle_{L^{2}(\Lambda_{\delta})}\\
&=\int_{0}^{\infty}e^{i\delta^{-1}t \sqrt{\lambda}}\mathrm{d}\langle E_{\delta}(\lambda)P_{\delta}z, P_{\delta}z\rangle_{L^{2}(\mathbb{R}^{d})}\\
&=  \int_{0}^{\infty}e^{i\delta^{-1}t \sqrt{\lambda}} \mathrm{d}\langle E_{\delta}(\lambda)P_{\delta}z-E(\lambda)z, P_{\delta}z\rangle_{L^{2}(\mathbb{R}^{d})} \\
&+ \langle e^{i\delta^{-1}t(-\Delta)^{1/2}} P_{\delta}z, P_{\delta}z\rangle_{L^{2}(\mathbb{R}^{d})} \rightarrow 0, \quad \delta \rightarrow 0,
\end{align*}
where we have used \eqref{eq:strong_convergence_projection} in \Cref{remark:subspace_projections} and \eqref{eq:weak_convergence_distribution} from \Cref{result:weak_convergence_partitions_of_unity} in the last step.
\end{proof}
\begin{proof}[Proof of \Cref{result:Limits_operator_cosine_sine}]
\label{proof:Limits_operator_cosine_sine}
As $(C_{\vartheta,\delta}(t), t \in [0,T])$ and $(S_{\vartheta,\delta}(t), t \in [0,T])$ are the operator cosine and sine functions associated with the operator $A_{\vartheta,\delta}$ for $\delta \geq 0$, we can represent them using the resolutions of identities associated with  $-A_{\vartheta,\delta}$ and $-\vartheta(0)\Delta$, respectively:
\begin{equation}
\label{eq:difference_scaling_limit_objects}
\begin{aligned}
\Vert S_{\vartheta,\delta}(\tau) P_{\delta} z &- S_{\vartheta(0)}(\tau)z  \Vert_{L^{2}(\mathbb{R}^{d})}^{2} \\
&= \int_{0}^{\infty} \left| \frac{\sin(\tau \sqrt{\lambda})}{\sqrt{\lambda}} \right|^{2} \mathrm{d}\langle (E_{\delta}(\lambda)P_\delta -E(\lambda))z, z\rangle_{L^{2}(\mathbb{R}^{d})},\\
\Vert C_{\vartheta,\delta}(\tau)P_{\delta}z &- C_{\vartheta(0)}(\tau)z  \Vert_{L^{2}(\mathbb{R}^{d})}^{2} \\
&= \int_{0}^{\infty} | \cos(\tau \sqrt{\lambda})|^{2} \mathrm{d}\langle (E_{\delta}(\lambda)P_\delta -E(\lambda))z, z\rangle_{L^{2}(\mathbb{R}^{d})}.
\end{aligned}
\end{equation}
Thus, as both the cosine and the mapping $\lambda \rightarrow \sin(\tau \sqrt{\lambda})/\sqrt{\lambda}$ are bounded, we can apply \Cref{result:weak_convergence_partitions_of_unity} and both expressions in \eqref{eq:difference_scaling_limit_objects} converge to zero. 
\end{proof}

  \subsection{Proof of asymptotic energy results}
\label{section:energy_results}
Unless stated otherwise, all limits are for $\delta \rightarrow 0$. For $z \in L^{1}(\mathbb{R}^{d})\cap L^{2}(\mathbb{R}^{d})$, we define the norm
\begin{equation*}
\Vert z \Vert_{L^{1}\cap L^{2}(\mathbb{R}^{d})}\coloneqq \Vert z \Vert_{L^{1}(\mathbb{R}^{d})}+ \Vert z \Vert_{L^{2}(\mathbb{R}^{d})},
\end{equation*}
and for $z$ with partial derivatives up to second order in $L^{1}(\mathbb{R}^{d})\cap L^{2}(\mathbb{R}^{d})$ set 
\begin{equation*}
\Vert z \Vert_{W_{1,2}^{2}(\mathbb{R}^{d})} \coloneqq \Vert z + |\nabla z| + \Delta z \Vert_{L^{1}\cap L^{2}(\mathbb{R}^{d})}. 
\end{equation*}
In order to use \citet*[Lemma A.6]{altmeyerNonparametricEstimationLinear2021}, we require the following assumption. 
\begin{assumptionApprox}{w^{(\delta)}}{z}
\label{asskernel_approx} Suppose $z$ satisfies \asskernel[$z$].
Let $w^{(\delta)} \in L^{2}(\mathbb{R}^{d})$  have compact support in $\Lambda_{\delta'}$ for some $\delta'>0$ and 
\begin{equation*}
\Vert w^{(\delta)}- \Delta z \Vert_{L^{1} \cap L^{2}(\mathbb{R}^{d})} \leq C \delta^{\alpha} \Vert z \Vert_{W^{2}_{1,2}(\mathbb{R}^{d})}
\end{equation*}
for all $0 < \delta \leq \delta'$.  
\end{assumptionApprox}
We emphasize that we use \asskernelapprox[$w^{\delta}$]{$z$} as a declaration and assertion of a $w^{(\delta)}$ and $z$ such that \asskernelapprox[$w^{\delta}$]{$z$} is satisfied. 
\begin{lemma}[Limits for fractions and inverse]
\label{result:fractional_and_inverse_limits}
Fix $z$ and $w^{(\delta)}$ satisfying \asskernel[$z$] and  \asskernelapprox[$w^{\delta}$]{$z$}. Assume for $\gamma>0$ that $\gamma > 1 - d/4 - \alpha/2$. Then, as $\delta \rightarrow 0$ we have 
\begin{enumerate}
\item \label{enum_limits:1} $\Vert (-A_{\vartheta,\delta})^{-1/2}w^{(\delta)}-(-\vartheta(0)\Delta)^{-1/2}\Delta z \Vert^{2}_{L^{2}(\mathbb{R}^{d})} \rightarrow 0,$
\item \label{enum_limits:2} $\Vert (-A_{\vartheta,\delta})^{-1}\Delta z - (-{\vartheta(0)} \Delta)^{-1}\Delta z\Vert^{2}_{L^{2}(\mathbb{R}^{d})} \rightarrow 0,$
\item \label{enum_limits:3} $\Vert (-A_{\vartheta,\delta})^{-1 + \gamma}w^{(\delta)} - (-\vartheta(0)\Delta)^{-1+ \gamma}z  \Vert^{2}_{L^{2}(\mathbb{R}^{d})} \rightarrow 0,$
\item \label{enum_bounds:4} $\sup_{0 < \delta \leq 1}\Vert (-A_{\vartheta,\delta})^{-1} \Delta z \Vert_{L^{2}(\mathbb{R}^{d})} < \infty,$
\item \label{enum_bounds:5} $\sup_{0 < \delta \leq 1}\Vert (-A_{\vartheta,\delta})^{-1/2}w^{(\delta)} \Vert_{L^{2}(\mathbb{R}^{d})} < \infty.$
\end{enumerate}
\end{lemma}
\begin{proof}[Proof of \Cref{result:fractional_and_inverse_limits}]
\begin{step}[Semigroup representations]
Using the semigroup representation for fractional powers of operators in \citet[Chapter 2.6]{pazySemigroupsLinearOperators2010}, we observe
\begin{align*}
(-A_{\vartheta,\delta})^{-1/2}w^{(\delta)}&=\frac{1}{\Gamma(1/2)}\int_{0}^{\infty}t^{-1/2}T_{\delta}(t)w^{(\delta)} \mathrm{d}t,\\
(-A_{\vartheta,\delta})^{-1} \Delta z &= \frac{1}{\Gamma(1)} \int_{0}^{\infty} T_{\delta}(t) \Delta z
\mathrm{d}t,\\
(-A_{\vartheta,\delta})^{-1 + \gamma}w^{(\delta)} &= \frac{1}{\Gamma(1-\gamma)} \int_{0}^{\infty}t^{-\gamma}T_{\delta}(t)w^{(\delta)}\mathrm{d}t,
\end{align*}
where $(T_{\delta}(t), t \geq 0)$ is the strongly continuous semigroup generated by $A_{\vartheta,\delta}$ on $L^{2}(\Lambda_{\delta})$.
\end{step}
\begin{step}[Pointwise convergence of the integrand]
Using \citet[Proposition 3.5]{altmeyerNonparametricEstimationLinear2021}, we immediately obtain
\begin{align*}
\Vert T_{\delta}(t)w^{(\delta)} - T_{0}(t) \Delta z\Vert_{L^{2}(\mathbb{R}^{d})} \rightarrow 0, \quad \delta \rightarrow 0,\\
\Vert T_{\delta}(t)\Delta z - T_{0}(t) \Delta z\Vert_{L^{2}(\mathbb{R}^{d})} \rightarrow 0, \quad \delta \rightarrow 0.
\end{align*}
\end{step}
\begin{step}[Upper bounds for dominated convergence]
Given \asskernel[$z$] and \citet[Lemma A.6]{altmeyerNonparametricEstimationLinear2021} and the absence of a first and zeroth order perturbation term, we observe
\begin{align}
\label{eq:semigroup_bound_1}
\Vert t^{-1/2}T_{\delta}(t)w^{(\delta)} \Vert_{L^{2}(\Lambda_{\delta})} &\lesssim (t^{-1/2} \land t^{-1/2-d/4-\alpha/2}) \Vert z \Vert_{W_{1,2}^{2}(\mathbb{R}^{d})},\\
\label{eq:semigroup_bound_2}
\Vert T_{\delta}(t) \Delta z \Vert_{L^{2}(\Lambda_{\delta})} &\lesssim  (1 \land t^{-1/2-d/4-\alpha/2})\Vert z \Vert_{W_{1,2}^{2}(\mathbb{R}^{d})},\\
\label{eq:semigroup_bound_3}
\Vert t^{-\gamma}T_{\delta}(t) w^{(\delta)} \Vert_{L^{2}(\Lambda_{\delta})} &\lesssim (1 \land t^{-\gamma - d/4 - \alpha/2})\Vert z \Vert_{W_{1,2}^{2}(\mathbb{R}^{d})} . 
\end{align}
\end{step}
\begin{step}[Upper bounds and dominated convergence]
We will now argue that the latter upper bounds \eqref{eq:semigroup_bound_1}, \eqref{eq:semigroup_bound_2} and \eqref{eq:semigroup_bound_3} are integrable on $(0,\infty)$. 
\begin{case}[$d=1$]
In the one-dimensional case, we have
\begin{equation}
\begin{aligned}
\label{eq:integrability_condition_1}
-1/2-d/4 -\alpha/2 = -3/4 - \alpha/2 &< -1,\\
-\gamma - \alpha/2 - 1/4 &< -1,
\end{aligned}
\end{equation}
because $\alpha>\frac{1}{2}$ by \assth{} and $\gamma>\frac{3}{4}-\alpha/2$.     
\end{case}
\begin{case}[{$d>1$}]
In this case $\alpha>0$ is sufficient since $-1/2-d/4 \leq -1$ implies
\begin{equation}
\begin{aligned}
\label{eq:integrability_condition_2}
-1/2-d/4-\alpha/2 <-1,\\
-\gamma-\alpha/2-\frac{d}{4}<-1.
\end{aligned}
\end{equation}
Note that for $d\geq 4$, the assumption on $\gamma$ is always satisfied.     
\end{case}
By \eqref{eq:integrability_condition_1} and \eqref{eq:integrability_condition_2} the upper bounds \eqref{eq:semigroup_bound_1}, \eqref{eq:semigroup_bound_2} and \eqref{eq:semigroup_bound_3} are integrable on $(0,\infty)$. The upper bounds are independent of the variable $\delta$, which implies \ref{enum_bounds:4} and \ref{enum_bounds:5}. The convergences \ref{enum_limits:1}, \ref{enum_limits:2} and \ref{enum_limits:3} follow by the dominated convergence theorem using the upper bounds \eqref{eq:semigroup_bound_1}, \eqref{eq:semigroup_bound_2} and \eqref{eq:semigroup_bound_3} as dominants. \qedhere
\end{step}
\end{proof}
\begin{proposition}[Asymptotic behaviour of energy for the operator sine]
\label{result:energy_asymptotics}
Grant the Assumptions \asskernel[$z_j$] and \asskernelapprox[$w_j^{(\delta)}$]{$z_j$} for $j=1,2$. 
\begin{enumerate}
\item  \label{enum:asymptotic_equipartition} \textit{Asymptotic equipartition of energy}: For any $r \in \mathbb{R}$, we have
\begin{equation*}
\begin{aligned}
\langle S_{\vartheta,\delta}(\delta^{-1}r)w^{(\delta)}_1&,  S_{\vartheta,\delta}(\delta^{-1}r)w^{(\delta)}_2\rangle_{L^{2}(\Lambda_{\delta})}\rightarrow \frac{1}{2\vartheta(0)}\langle \nabla z_1,\nabla z_2 \rangle_{L^{2}(\mathbb{R}^{d})}, \quad \delta \rightarrow 0.
\end{aligned}
\end{equation*}
\item \label{enum:temporal_orthogonality} \textit{Slow-fast orthogonality}: For any $r_1, r_2 \in \mathbb{R}$ with $r_1 \neq r_2$, we have
\begin{equation*}
\langle S_{\vartheta,\delta}(\delta^{-1}r_1)w^{(\delta)}_1,  S_{\vartheta,\delta}(\delta^{-1}r_2)w^{(\delta)}_2\rangle_{L^{2}(\Lambda_{\delta})} \rightarrow 0, \quad \delta \rightarrow 0.
\end{equation*}
\end{enumerate}
\end{proposition}
\begin{proof}[Proof of \Cref{result:energy_asymptotics}]
\begin{step}[Representation using Riemann-Lebesgue operators]
We abbreviate by $R_{\delta}(t)=e^{it(-A_{\vartheta,\delta})^{1/2}}$ the unitary group on the complex Hilbert space $L^{2}(\Lambda_{\delta})$ for $t \in \mathbb{R}$. The operator sine can be represented using $R_\delta(t)$ through
\begin{equation*}
S_{\vartheta, \delta}(\delta^{-1} \tau)w_j^{(\delta)}=\frac{(-A_{\vartheta,\delta})^{-1/2}}{2i}\left(R_{\delta}(\delta^{-1}\tau)- R_{\delta}(-\delta^{-1}\tau)\right)w_j^{(\delta)}, \quad \tau \in \mathbb{R}, \quad j \in \{1,2\}.
\end{equation*}
Setting $\xi_{j}^{(\delta)} = (-A_{\vartheta,\delta})^{-1/2}w_j^{(\delta)}$, we observe
\begin{equation}
\label{eq:Riemann-Lebesgue-Representation}
\begin{aligned}
&4\langle S_{\vartheta,\delta}(\delta^{-1}r_1)\xi^{(\delta)}_1,  S_{\vartheta,\delta}(\delta^{-1}r_2)\xi^{(\delta)}_2\rangle_{L^{2}(\Lambda_{\delta})}\\
&= \langle (R_{\delta}(\delta^{-1}r_1)-R_{\delta}(-\delta^{-1}r_1))\xi_1^{(\delta)} , (R_{\delta}(\delta^{-1}r_2)-R_{\delta}(-\delta^{-1}r_2))\xi_2^{(\delta)} \rangle_{L^{2}(\Lambda_{\delta})}\\
&=\langle (R_{\delta}(\delta^{-1}r_2)-R_{\delta}(-\delta^{-1}r_2))^{*} \circ(R_{\delta}(\delta^{-1}r_1)-R_{\delta}(-\delta^{-1}r_1))\xi_1^{(\delta)} , \xi_2^{(\delta)} \rangle_{L^{2}(\Lambda_{\delta})}\\
&=\langle (R_{\delta}(-\delta^{-1}r_2)-R_{\delta}(\delta^{-1}r_2)) \circ (R_{\delta}(\delta^{-1}r_1)-R_{\delta}(-\delta^{-1}r_1))\xi_1^{(\delta)} , \xi_2^{(\delta)} \rangle_{L^{2}(\Lambda_{\delta})}\\
&= \langle R_{\delta}(\delta^{-1}(r_1-r_2))\xi_1^{(\delta)},\xi_2^{(\delta)} \rangle_{L^{2}(\Lambda_{\delta})}+ \langle R_{\delta}(\delta^{-1}(r_2-r_1))\xi_1^{(\delta)},\xi_2^{(\delta)} \rangle_{L^{2}(\Lambda_{\delta})}\\
& \quad - \langle R_\delta(\delta^{-1}(r_1+r_2))\xi_1^{(\delta)},\xi_2^{(\delta)}\rangle_{L^{2}(\Lambda_{\delta})}-\langle R_\delta(-\delta^{-1}(r_1+r_2))\xi_1^{(\delta)},\xi_2^{(\delta)}\rangle_{L^{2}(\Lambda_{\delta})}.
\end{aligned}
\end{equation}
\end{step}
\begin{step}[Convergences]
We have \asskernel[$z_j$], \asskernelapprox[$w^{(\delta)}_j$]{$z_j$} and \assth{}. Thus, all assumptions of \Cref{result:fractional_and_inverse_limits}\ref{enum_limits:1} are satisfied, and we obtain
\begin{equation}
\label{eq:convergence_in_norm}
\Vert \xi_j^{(\delta)} -  (-{\vartheta(0)^{-1/2}}(-\Delta)^{1/2})z_j\Vert^{2}_{L^{2}(\mathbb{R}^{d})} \rightarrow 0, \quad \delta \rightarrow 0, \quad j=1,2.
\end{equation}
Given \Cref{remark:subspace_projections} and using the convergence in norm \eqref{eq:convergence_in_norm}, we can reduce the desired convergence of the expression
\begin{equation*}
\langle R_\delta(\delta^{-1}\tau)\xi_{1}^{(\delta)}, \xi_2^{(\delta)}\rangle_{L^{2}(\Lambda_{\delta})} \rightarrow 0, \quad \delta \rightarrow 0,
\end{equation*}
for any real number $\tau \in \mathbb{R}\setminus \{0\}$ to the convergence
\begin{equation}
\label{eq:asymptotic_RL_principle_in_action}
\langle R_{\delta}(\delta^{-1}\tau) (-\Delta)^{1/2}z_{j_1}, (-\Delta)^{1/2}z_{j_2} \rangle_{L^{2}(\Lambda_{\delta})} \rightarrow 0, \quad \delta \rightarrow 0, \quad j_1, j_2 \in \{1,2\}.
\end{equation}
Clearly, \eqref{eq:asymptotic_RL_principle_in_action} follows from the asymptotic Riemann-Lebesgue principle  (\Cref{result:approximate_riemann_lebesgue}).
For \ref{enum:temporal_orthogonality} every summand in \eqref{eq:Riemann-Lebesgue-Representation} vanishes since $r_1 \neq r_2$. Similarly, for $r=r_1=r_2$, we obtain the remaining term
\begin{align*}
\frac{1}{2}\langle \xi_1^{(\delta)}, \xi_2^{(\delta)} \rangle_{L^{2}(\Lambda_{\delta})} &\rightarrow \frac{1}{2\vartheta(0)}\langle (-\Delta)^{1/2}z_1,(-\Delta)^{1/2}z_2 \rangle_{L^{2}(\mathbb{R}^{d})}\\
&= \frac{1}{2 \vartheta(0)}\langle \nabla z_1, \nabla z_2 \rangle_{L^{2}(\mathbb{R}^{d})},\quad \delta \rightarrow 0. 
\qedhere
\end{align*}
\end{step}
\end{proof}
For the analysis of the bias, we define the functions:
\begin{equation}
\label{eq:beta_delta}
\begin{aligned}
\beta^{(\delta)}(x) &\coloneqq \delta^{-1}(A_{\vartheta,\delta}- \vartheta(0)\Delta)K(x), \quad \delta > 0,\\ 
\beta^{(0)}(x)& \coloneqq \Delta( \langle \nabla\vartheta(0),x \rangle_{\mathbb{R}^{d}} K ) (x)- \langle \nabla\vartheta(0), \nabla K(x)\rangle_{\mathbb{R}^{d}},\quad x \in \mathbb{R}^{d}.
\end{aligned}
\end{equation}
Using \eqref{eq:beta_delta}, we determine the asymptotic behaviour of expressions emerging in the analysis of the observed Fisher information \eqref{eq:observed_fisher_information} and the bias \eqref{eq:remaining_bias}. 
\begin{proposition}[Asymptotics for the emerging energetic expressions]
\label{result:asymptotics_emerging_energetic_expressions}
Grant Assumption \asskernel{}. Let $\beta^{(0)}$ and $\beta^{(\delta)}$ be defined through \eqref{eq:beta_delta}. As $\delta \rightarrow 0$ we obtain the following convergences. 
\begin{enumerate}
\item \label{enum:concrete_asymptotic_equipartition} For $t \in \mathbb{R}\setminus \{0\}$, we have the asymptotic equipartitions
\begin{align}
\label{eq:emerging_equipartition_1}
\Vert S_{\vartheta, \delta}(\delta^{-1}t) \Delta K \Vert^{2}_{L^{2}(\Lambda_{\delta})} &\rightarrow \frac{1}{2\vartheta(0)}\Vert \nabla K \Vert^{2}_{L^{2}(\mathbb{R}^{d})},\\
\label{eq:emerging_equipartition_2}
\langle S_{\vartheta, \delta}(\delta^{-1}t)\Delta K,S_{\vartheta, \delta}(\delta^{-1}t)\beta^{(\delta)} \rangle_{L^{2}(\Lambda_{\delta})} &\rightarrow \frac{1}{2 \vartheta(0)}\langle \nabla K, \nabla\beta^{(0)} \rangle_{L^{2}(\mathbb{R}^{d})}.
\end{align}
\item \label{enum:concrete_asymptotic_orthogonality} For $s,t  \in \mathbb{R}$ with $s \neq t$ fixed, we have the slow-fast orthogonality
\begin{align}
\label{eq:emerging_orthogonality_1}
\langle S_{\vartheta,\delta}(\delta^{-1}t)\Delta K, S_{\vartheta,\delta}(\delta^{-1}s)\Delta K\rangle_{L^{2}(\Lambda_{\delta})} &\rightarrow 0, \\
\label{eq:emerging_orthogonality_2}
\langle S_{\vartheta,\delta}(\delta^{-1}t)\Delta K,S_{\vartheta, \delta}(\delta^{-1}s)\beta^{(\delta)} \rangle_{L^{2}(\Lambda_{\delta})} &\rightarrow 0,
\\
\label{eq:emerging_orthogonality_3}
\langle S_{\vartheta,\delta}(\delta^{-1}t)\beta^{(\delta)},S_{\vartheta, \delta}(\delta^{-1}s)\beta^{(\delta)} \rangle_{L^{2}(\Lambda_{\delta})} &\rightarrow 0.
\end{align}
\end{enumerate}
\end{proposition}
\begin{proof}[Proof of \Cref{result:asymptotics_emerging_energetic_expressions}]
\label{proof:asymptotics_emerging_energetic_expressions}
The result is a corollary of \Cref{result:energy_asymptotics} by setting $v^{(\delta)}_{j_1} = \beta^{(\delta)}$ or $v^{(\delta)}_{j_2} = \Delta K$ for $j_1, j_2 \in \{1,2\}$. In particular, the convergences \eqref{eq:emerging_equipartition_1} and \eqref{eq:emerging_equipartition_2}, and \eqref{eq:emerging_orthogonality_1}, \eqref{eq:emerging_orthogonality_2} and \eqref{eq:emerging_orthogonality_3} follow from the asymptotic equipartition \ref{enum:asymptotic_equipartition} and the asymptotic orthogonality \ref{enum:temporal_orthogonality} in \Cref{result:energy_asymptotics}, respectively. The conditions of \Cref{result:energy_asymptotics} are trivially satisfied for $\Delta K$ in view of \asskernel{}. On the other, we refer to \citet[Lemma A.5]{altmeyerNonparametricEstimationLinear2021} for the existence of a function $z$ such that \asskernelapprox[$\beta^{\delta}$]{$z$} is satisfied. 
\end{proof}

\begin{lemma}[Uniform upper bounds operator sine]
\label{result:uniform_upperbounds_operator_sine} 
Suppose that $K$ satisfies \asskernel{}. Then, we have 
\begin{align}
\label{eq:energy_upperbound_1}
\Vert S_{\vartheta,\delta}(\delta^{-1}r)\Delta K \Vert^{2}_{L^{2}(\Lambda_{\delta})} \leq \sup_{0 < \delta \leq 1} \Vert (-A_{\vartheta,\delta})^{-1/2} \Delta K \Vert^{2}_{L^{2}(\Lambda_{\delta})} < \infty,\\
\label{eq:energy_upperbound_2}
\Vert S_{\vartheta,\delta}(\delta^{-1}r)\beta^{(\delta)} \Vert^{2}_{L^{2}(\Lambda_{\delta})} \leq \sup_{0 < \delta \leq 1} \Vert (-A_{\vartheta,\delta})^{-1/2} \beta^{(\delta)} \Vert^{2}_{L^{2}(\Lambda_{\delta})} < \infty,
\end{align}
where $\beta^{(\delta)}$ is defined through \eqref{eq:beta_delta}.
\end{lemma}
\begin{proof}[Proof of \Cref{result:uniform_upperbounds_operator_sine}]
By \asskernel{} and \assth{}, both suprema are finite by (iv) and (v) in \Cref{result:fractional_and_inverse_limits}. The result now follows immediately using the functional calculus from the boundedness of the sine through
\begin{equation*}
\begin{aligned}
\Vert S_{\vartheta, \delta}(\delta^{-1}r)z \Vert^{2}_{L^{2}(\Lambda_{\delta})} &= \Vert \sin(\delta^{-1}r (-A_{\vartheta,\delta})^{1/2}) (-A_{\vartheta,\delta})^{-1/2}z \Vert^{2}_{L^{2}(\Lambda_{\delta})} \\
&\leq \Vert (-A_{\vartheta,\delta})^{-1/2}z \Vert_{L^{2}(\Lambda_{\delta})}^{2}
\end{aligned}
\end{equation*}
for any $z \in L^{2}(\Lambda_{\delta})$.
\qedhere
\end{proof}

  \subsection{Asymptotics for the augmented MLE}
\label{section:augmented_MLE}
We decompose the mild solution to the stochastic wave equation with non-zero initial conditions into a deterministic and a stochastic part given by 
\begin{align}
\label{eq:initial_decomp_1}
{u}(t) &= C_{\vartheta}(t)u_0 + S_{\vartheta}(t)v_0 + \tilde{u}(t), \quad t \in [0,T], \\
{v}(t) &= C'_{\vartheta}(t)u_0 + C_{\vartheta}(t)v_0 + \tilde{v}(t), \quad t \in [0,T],
\end{align}
where $(\tilde{u}(t), \tilde{v}(t))$ is a solution to the stochastic wave equation with zero-initial conditions. Based on the solution with zero-initial conditions, we introduce the notation
\begin{align*}
\tilde{u}_{\delta}^{\Delta}(t)&\coloneqq \langle \tilde{u}(t),\delta^{-2}(\Delta K)_{\delta} \rangle_{L^{2}(\Lambda)}, \\
\tilde{I _{\delta}} &\coloneqq \int_{0}^{T}\tilde{u}_{\delta}^{\Delta}(t)^{2}\mathrm{d}t,\\
\tilde{R}_{\delta} & \coloneqq \int_{0}^{T}\tilde{u}_{\delta}^{\Delta}(t)\langle \tilde{u}(t),(A_{\vartheta}-\vartheta(0)\Delta)K_{\delta} \rangle_{L^{2}(\Lambda)} \mathrm{d}t.
\end{align*}
Notice that the local measurements depend on the initial conditions through
\begin{equation}
\label{eq:initial_decomp_2}
{u}_{\delta}^{\Delta}(t)= \tilde{u}_{\delta}^{\Delta}(t) + \locu(t) + \locv(t), \quad t \in [0,T],
\end{equation}
where 
\begin{align*}
\locu(t) &\coloneqq \mathcal{L}_{\delta}^{C}(u_0, \Delta K)(t) \coloneqq \langle C(t)u_0, \delta^{-2}(\Delta K)_{\delta}\rangle_{L^{2}(\Lambda)}\\
\locv(t) &\coloneqq \mathcal{L}_{\delta}^{S}(v_0, \Delta K)(t)  \coloneqq  \langle S(t)v_0, \delta^{-2}(\Delta K)_{\delta}\rangle_{L^{2}(\Lambda)}.
\end{align*}
\begin{lemma}[Dependence of the observed Fisher information on the initial conditions]
\label{result:dependence_of_observed_fisher_information_on_initial_conditions} For any fixed $\delta>0$ and deterministic $(u_0,v_0) \in L^{2}(\Lambda) \times H_0^{-1}(\Lambda)$, we have:
\begin{enumerate}
\item The observed Fisher information ${I}_{\delta}$ depends on the initial conditions through 
\begin{equation}
\label{eq:observed_fisher_representation_with_initials}
\begin{aligned}
{I}_{\delta} = \tilde{I}_{\delta} &+  2(\langle \tilde{u}_\delta^{\Delta},\locu \rangle_{L^{2}([0,T])} 
+ \langle \tilde{u}_\delta^{\Delta},\locv \rangle_{L^{2}([0,T])}) \\
&+ \Vert \locu \Vert^{2}_{L^{2}([0,T])} + \Vert \locv \Vert^{2}_{L^{2}([0,T])} + 2 \langle \locu, \locv\rangle_{L^{2}([0,T])}.
\end{aligned}
\end{equation}
\item The expectation of the observed Fisher information ${I}_{\delta}$ satisfies \begin{equation}
\label{eq:expectation_observed_fisher}
\mathbb{E}[{I}_{\delta}]= \mathbb{E}[\tilde{I}_{\delta}]+ \Vert \locu \Vert^{2}_{L^{2}([0,T])} + \Vert \locv \Vert^{2}_{L^{2}([0,T])} + 2 \langle \locu, \locv\rangle_{L^{2}([0,T])}.
\end{equation}
\end{enumerate}
\end{lemma}
\begin{proof}[Proof of \Cref{result:dependence_of_observed_fisher_information_on_initial_conditions}] The result follows immediately using the Binomial formula's and noticing 
\begin{equation*}
\mathbb{E}[\langle \tilde{u}_\delta^{\Delta},\locu \rangle_{L^{2}([0,T])}] = \mathbb{E}[\langle \tilde{u}_\delta^{\Delta},\locv \rangle_{L^{2}([0,T])}]=0. \qedhere
\end{equation*}
\end{proof}
\begin{lemma}[Bounds for the deterministic parts]
\label{result:boundedness_initial_object}
Grant \assini{}. Then, we have 
\begin{equation}
\label{eq:boundedness_initial_object}
\Vert \mathcal{L}_{\delta}^{C} + \mathcal{L}_{\delta}^{S} \Vert^{2}_{L^{2}([0,T])} \in \mathcal{O}(1), \quad \delta \rightarrow 0.
\end{equation}
\end{lemma}
\begin{proof}[Proof of \Cref{result:boundedness_initial_object}]
\begin{step}[Rewriting $\mathcal{L}_{\delta}^{C}$ and $\mathcal{L}_{\delta}^{S}$]
Using properties of the rescaling (\Cref{result:rescaling_trigonometric_operator_families}), we observe
\begin{align*}
&\langle C_{\vartheta}(t)u_0, \delta^{-2}(\Delta K)_{\delta}\rangle_{L^{2}(\Lambda)}\\ &= \delta^{-2}\langle u_0, (C_{\vartheta,\delta}(\delta^{-1}t) \Delta K)_{\delta} \rangle_{L^{2}(\Lambda)}\\
&= \delta^{-2}\langle (u_0)_{\delta^{-1}}, C_{\vartheta,\delta}(\delta^{-1}t) \Delta K \rangle_{L^{2}(\Lambda_{\delta})}\\
&= \delta^{-2} \langle (-A_{\vartheta,\delta}) (u_0)_{\delta^{-1}}, C_{\vartheta,\delta}(\delta^{-1}t)(-A_{\vartheta,\delta})^{-1}\Delta K \rangle_{L^{2}(\Lambda_{\delta})}\\
&= \langle (-A_{\vartheta}u_0)_{\delta^{-1}}, C_{\vartheta,\delta}(\delta^{-1}t)(-A_{\vartheta,\delta})^{-1}\Delta K \rangle_{L^{2}(\Lambda_{\delta})} 
\end{align*}
and
\begin{align*}
&\langle S_{\vartheta}(t)v_0, \delta^{-2}(\Delta K)_{\delta} \rangle_{L^{2}(\Lambda)} \\&= \delta^{-2} \langle v_0, \delta(S_{\vartheta,\delta}(\delta^{-1}t) \Delta K)_{\delta} \rangle_{L^{2}(\Lambda)}\\
&= \delta^{-1}\langle (-A_{\vartheta,\delta})^{1/2}(v_0)_{\delta^{-1}}, S_{\vartheta,\delta}(\delta^{-1}t)  (-A_{\vartheta,\delta})^{-1/2}\Delta K\rangle_{L^{2}(\Lambda_{\delta})}\\
&= \langle ((-A_{\vartheta})^{1/2}v_0)_{\delta^{-1}},  S_{\vartheta,\delta}(\delta^{-1}t)(-A_{\vartheta,\delta})^{-1/2}\Delta K\rangle_{L^{2}(\Lambda_{\delta})}. 
\end{align*}
\end{step}
\begin{step}[Upper bound using Cauchy-Schwarz]
The Cauchy-Schwarz inequality yields 
\begin{align}
&|\langle C_{\vartheta}(t)u_0, \delta^{-2}(\Delta K)_{\delta} \rangle_{L^{2}(\Lambda)}| \nonumber\\&\leq \Vert (-A_{\vartheta} u_0)_{\delta^{-1}} \Vert_{L^{2}(\Lambda_{\delta})} \Vert C_{\vartheta,\delta}(\delta^{-1}t) \Vert_{\mathcal{L}(L^{2}(\Lambda_{\delta}), L^{2}(\Lambda_{\delta}))} \Vert (-A_{\vartheta,\delta})^{-1}\Delta K \Vert_{L^{2}(\Lambda_{\delta})} \nonumber\\
& \leq \Vert A_{\vartheta}u_0 \Vert_{L^{2}(\Lambda)} \Vert (-A_{\vartheta,\delta})^{-1}\Delta K \Vert_{L^{2}(\Lambda_{\delta})} \label{eq:upperbound_initial_delta_K_1}\\
& |\langle S_{\vartheta}(t)v_0, \delta^{-2}(\Delta K)_{\delta} \rangle_{L^{2}(\Lambda)}| \nonumber\\&\leq \Vert ((-A_{\vartheta})^{1/2}v_0)_{\delta^{-1}} \Vert_{L^{2}(\Lambda_{\delta})} \Vert S_{\vartheta,\delta}(\delta^{-1}t) \Vert_{\mathcal{L}(L^{2}(\Lambda_{\delta}), L^{2}(\Lambda_{\delta}))} \Vert (-A_{\vartheta,\delta})^{-1/2}\Delta K \Vert_{L^{2}(\Lambda_{\delta})} \nonumber\\
& \leq \Vert (-A_{\vartheta})^{1/2}v_0 \Vert_{L^{2}(\Lambda)} \Vert (-A_{\vartheta,\delta})^{-1/2}\Delta K \Vert_{L^{2}(\Lambda_{\delta})} T. \label{eq:upperbound_initial_delta_K_2} 
\end{align}
\end{step}
\begin{step}[Upper bounds remain finite]
Clearly both $A_{\vartheta}u_0$ and $(-A_{\vartheta})^{1/2}v_0$ are elements of the Hilbert space $L^{2}(\Lambda)$. Thus, their $\Vert \cdot \Vert_{L^{2}(\Lambda)}$ norm remains finite. Furthermore, by \Cref{result:fractional_and_inverse_limits} both $(-A_{\vartheta,\delta})^{-1}\Delta K$ and  $(-A_{\vartheta,\delta})^{-1/2}\Delta K$ converge in $L^{2}(\mathbb{R})$, implying
\begin{equation*}
\sup_{0 < \delta < \delta'} \Vert (-A_{\vartheta,\delta})^{-1}\Delta K \Vert_{L^{2}(\Lambda_{\delta})} < \infty, \quad \sup_{0 < \delta < \delta'} \Vert (-A_{\vartheta,\delta})^{-1/2}\Delta K \Vert_{L^{2}(\Lambda_{\delta})} < \infty.
\end{equation*}
The result follows from \eqref{eq:upperbound_initial_delta_K_1} and \eqref{eq:upperbound_initial_delta_K_2} as the time integral of the squared upper bound remains finite on a finite time horizon. \qedhere
\end{step}
\end{proof}
\label{section:asymptotics_augmentedMLE}
\begin{lemma}[Expectation and variance of observed Fisher information $\tilde{I}_{\delta}$]
\label{result:variance_and_covariance_representations}
Grant \asskernel{}. The expectation of the observed Fisher information satisfies
\begin{align}
\label{eq:rep_observed_fisher_information_expectation}
\delta^{2}\mathbb{E}[\tilde{I}_{\delta}] &= \int_{0}^{T} \int_{0}^{t}\Vert S_{\vartheta, \delta}(\delta^{-1}r) \Delta K \Vert^{2}_{L^{2}(\Lambda_{\delta})} \mathrm{d}r \mathrm{d}t, \quad \delta > 0,
\end{align}
and the variance of the observed Fisher information is given by
\begin{equation*}
\begin{aligned}
\label{eq:rep_observed_fisher_information_variance}
&\mathrm{Var}(\delta^{2}\tilde{I}_{\delta}) \\&= 2\int_{0}^{T}\int_{0}^{T}\left(\int_{0}^{t \land s} \langle S_{\vartheta,\delta}(\delta^{-1}(t-r))\Delta K,S_{\vartheta,\delta}(\delta^{-1}(s-r))\Delta K \rangle_{L^{2}(\Lambda_{\delta})} \mathrm{d}r\right)^{2}\mathrm{d}s\mathrm{d}t. 
\end{aligned}
\end{equation*}
\end{lemma}
\begin{proof}[Proof of \Cref{result:variance_and_covariance_representations}]
Using Fubini's theorem and \Cref{result:unitary_group}, we observe by \citet[Proposition 4.28]{dapratoStochasticEquationsInfinite2014}:
\begin{equation}
\label{eq:expectation_observed_fisher_zero_initial}
\begin{aligned}
\mathbb{E}[\delta^{2} \tilde{I}_{\delta}]&= \delta^{2}\int_{0}^{T} \mathrm{Var}(\tilde{u}_\delta^{\Delta}(t))\mathrm{d}t 
= \delta^{2} \int_{0}^{T} \int_{0}^{t}\Vert S_{\vartheta}(t-s) \Delta K_{\delta} \Vert^{2}_{L^{2}(\Lambda)} \mathrm{d}s \mathrm{d}t\\
&=  \delta^{-2} \int_{0}^{T} \int_{0}^{t}\Vert S_{\vartheta}(t-s) (\Delta K)_{\delta} \Vert^{2}_{L^{2}(\Lambda)} \mathrm{d}s \mathrm{d}t.
\end{aligned}
\end{equation}
Applying the rescaling of the operator sine function in \Cref{result:rescaling_trigonometric_operator_families} \ref{enum:rescaling} to \eqref{eq:expectation_observed_fisher_zero_initial} yields
\begin{align*}
\mathbb{E}[\delta^{2} \tilde{I}_{\delta}] &= \int_{0}^{T} \int_{0}^{t}\Vert (S_{\vartheta, \delta}(\delta^{-1}(t-s)) \Delta K)_{\delta} \Vert^{2}_{L^{2}(\Lambda)} \mathrm{d}s \mathrm{d}t\\
&=\int_{0}^{T} \int_{0}^{t}\Vert S_{\vartheta, \delta}(\delta^{-1}r) \Delta K \Vert^{2}_{L^{2}(\Lambda_{\delta})} \mathrm{d}r \mathrm{d}t.
\end{align*}
Similarly, using Wick's formula (\citet*[Theorem 1.28]{jansonGaussianHilbertSpaces1997}), we obtain for the covariance
\begin{equation}
\label{eq:covariance_observed_fisher_zero_initial}
\begin{aligned}
&\mathrm{Var}(\delta^{2}\tilde{I}_{\delta})\\ &= 2\delta^{4} \int_{0}^{T}\int_{0}^{T} \mathrm{Cov}(\tilde{u}_{\delta}^{\Delta}(s), \tilde{u}_{\delta}^{\Delta}(t))^{2} \mathrm{d}s \mathrm{d}t\\
&= 2\delta^{4}\int_{0}^{T}\int_{0}^{T} \left( \int_{0}^{t \land s} \langle S_{\vartheta}(t-r) \Delta K_{\delta}, S_{\vartheta}(s-r) \Delta K_{\delta}\rangle_{L^{2}(\Lambda)}\mathrm{d}r\right)^{2}\mathrm{d}s \mathrm{d}t\\
&= 2\delta^{-4}\int_{0}^{T}\int_{0}^{T} \left( \int_{0}^{t \land s} \langle S_{\vartheta}(t-r) (\Delta K)_{\delta}, S_{\vartheta}(s-r)  (\Delta K)_{\delta}\rangle_{L^{2}(\Lambda)}\mathrm{d}r\right)^{2}\mathrm{d}s \mathrm{d}t\\
\end{aligned}
\end{equation}
Another application of the rescaling of the operator sine function in \Cref{result:rescaling_trigonometric_operator_families} \ref{enum:rescaling} to \eqref{eq:covariance_observed_fisher_zero_initial} amounts to
\begin{align*}
&\mathrm{Var}(\delta^{2}\tilde{I}_{\delta}) \\
&= 2\delta^{-4}\int_{0}^{T}\int_{0}^{T} \left( \int_{0}^{t \land s} \langle S_{\vartheta}(t-r) (\Delta K)_{\delta}, S_{\vartheta}(s-r)  (\Delta K)_{\delta}\rangle_{L^{2}(\Lambda)}\mathrm{d}r\right)^{2}\mathrm{d}s \mathrm{d}t\\
&= 2\int_{0}^{T}\int_{0}^{T} \left( \int_{0}^{t \land s} \langle S_{\vartheta, \delta }( \delta^{-1}(t-r)) \Delta K, S_{\vartheta, \delta}(\delta^{-1}(s-r))\Delta K\rangle_{L^{2}(\Lambda_{\delta})}\mathrm{d}r\right)^{2}\mathrm{d}s \mathrm{d}t.
\end{align*}
\end{proof}

\begin{lemma}[Expectation and variance of the remaining bias $\tilde{R}_{\delta}$]
\label{result:variance_and_covariance_representations_remaining_bias}
Grant \asskernel{}. The expectation of the remaining bias satisfies
\begin{equation}
\label{eq:expectation_representation_remaining_bias}
\mathbb{E}[\delta \tilde{R}_{\delta}]= \int_{0}^{T}\int_{0}^{t}\langle S_{\vartheta,\delta}(\delta^{-1}s)\Delta K, S_{\vartheta, \delta}(\delta^{-1}s)\beta^{(\delta)} \rangle_{L^{2}(\Lambda_{\delta})}\mathrm{d}s\mathrm{d}t,
\end{equation}
and its variance is given by
\begin{equation}
\label{eq:variance_representation_remaining_bias}
\begin{aligned}
&\mathrm{Var}(\delta \tilde{R}_{\delta})=
\int_{0}^{T}\int_{0}^{T} (\mathrm{Cov}(\tilde{u}_{\delta}^{\Delta}(t), \langle \tilde{u}(s),\beta^{(\delta)} \rangle_{L^{2}(\Lambda)}))^{2} \\
&\quad+ \mathrm{Cov}(\tilde{u}_{\delta}^{\Delta}(t), \langle \tilde{u}(s),\beta^{(\delta)} \rangle_{L^{2}(\Lambda)})\mathrm{Cov}(\tilde{u}_{\delta}^{\Delta}(s), \langle \tilde{u}(t),\beta^{(\delta)} \rangle_{L^{2}(\Lambda)}) \mathrm{d}t\mathrm{d}s,
\end{aligned}
\end{equation}
where
\begin{align*}
&\mathrm{Cov}(\tilde{u}_{\delta}^{\Delta}(t), \langle \tilde{u}(s), \beta^{(\delta)} \rangle_{L^{2}(\Lambda)})\\
&= \int_{0}^{t \land s}\langle S_{\vartheta,\delta}(\delta^{-1}(t-r))\Delta K,S_{\vartheta,\delta}(\delta^{-1}(s-r))\beta^{(\delta)} \rangle_{L^{2}(\Lambda_{\delta})}\mathrm{d}r,
\end{align*}
with $\beta^{(\delta)}$ is defined through \eqref{eq:beta_delta}.
\end{lemma}
\begin{proof}[Proof of \Cref{result:variance_and_covariance_representations_remaining_bias}]
With \eqref{eq:beta_delta}, we observe
\begin{equation}
\label{eq:beta_delta_representation_bias_zero_initial}
\begin{aligned}
\mathbb{E}[\delta \tilde{R}_{\delta}] &= \mathbb{E}\left(\delta \int_{0}^{T}\tilde{u}_{\delta}^{\Delta}(t)\langle \tilde{u}(t),(A_{\vartheta}-\vartheta(0)\Delta)K_{\delta} \rangle_{L^{2}(\Lambda)} \mathrm{d}t\right)\\
&= \mathbb{E}\left(\delta \int_{0}^{T}\tilde{u}_{\delta}^{\Delta}(t)\langle \tilde{u}(t),\delta^{-2}((A_{\vartheta}-\vartheta(0)\Delta)K)_{\delta} \rangle_{L^{2}(\Lambda)} \mathrm{d}t\right)\\
&= \mathbb{E}\left( \int_{0}^{T}\tilde{u}_{\delta}^{\Delta}(t)\langle \tilde{u}(t),\delta^{-1}((A_{\vartheta}-\vartheta(0)\Delta)K)_{\delta} \rangle_{L^{2}(\Lambda)} \mathrm{d}t\right)\\
&= \mathbb{E}\left( \int_{0}^{T}\tilde{u}_{\delta}^{\Delta}(t)\langle \tilde{u}(t),\beta^{(\delta)}_{\delta} \rangle_{L^{2}(\Lambda)} \mathrm{d}t\right).
\end{aligned}
\end{equation}
As before, an application of the rescaling of the operator sine function in \Cref{result:rescaling_trigonometric_operator_families} \ref{enum:rescaling} to the representation \eqref{eq:beta_delta_representation_bias_zero_initial} implies
\begin{align*}
&\mathbb{E}[\delta \tilde{R}_{\delta}] \\
&= \int_{0}^{T}\int_{0}^{t}\langle S_{\vartheta}(t-s) \Delta K_{\delta}, S_{\vartheta}(t-s)\beta^{(\delta)}_{\delta} \rangle_{L^{2}(\Lambda)}\mathrm{d}s \mathrm{d}t\\
&= \delta^{-2}\int_{0}^{T}\int_{0}^{t}\langle S_{\vartheta}(t-s) (\Delta K)_{\delta}, S_{\vartheta}(t-s)\beta^{(\delta)}_{\delta} \rangle_{L^{2}(\Lambda)}\mathrm{d}s \mathrm{d}t\\
&= \delta^{-2}\int_{0}^{T}\int_{0}^{t}\langle \delta (S_{\vartheta,\delta}(\delta^{-1}(t-s)) \Delta K)_{\delta}, \delta (S_{\vartheta, \delta}(\delta^{-1}(t-s))\beta^{(\delta)})_{\delta} \rangle_{L^{2}(\Lambda)}\mathrm{d}s \mathrm{d}t\\
&= \int_{0}^{T}\int_{0}^{t} \langle S_{\vartheta,\delta}(\delta^{-1}(t-s))\Delta K,S_{\vartheta,\delta}(\delta^{-1}(t-s))\beta^{(\delta)} \rangle_{L^{2}(\Lambda_{\delta})} \mathrm{d}s \mathrm{d}t.
\end{align*}
Using Wick's formula (\citet*[Theorem 1.28]{jansonGaussianHilbertSpaces1997}), we further have
\begin{align*}
\mathrm{Var}(\delta \tilde{R}_{\delta})&= \mathrm{Var}\left( \int_{0}^{T}\tilde{u}_{\delta}^{\Delta}(t) \langle \tilde{u}(t), \beta_{\delta}^{(\delta)} \rangle_{L^{2}(\Lambda)} \mathrm{d}t\right)\\
&= \int_{0}^{T}\int_{0}^{T} (\mathrm{Cov}(\tilde{u}_{\delta}^{\Delta}(t), \langle \tilde{u}(s),\beta^{(\delta)} \rangle_{L^{2}(\Lambda)}))^{2} \\
&\quad + \mathrm{Cov}(\tilde{u}_{\delta}^{\Delta}(t), \langle \tilde{u}(s),\beta^{(\delta)} \rangle_{L^{2}(\Lambda)})\mathrm{Cov}(\tilde{u}_{\delta}^{\Delta}(s), \langle \tilde{u}(t),\beta^{(\delta)} \rangle_{L^{2}(\Lambda)}) \mathrm{d}s\mathrm{d}t.
\end{align*}
Using the same rescaling arguments for the operator sine function in \Cref{result:rescaling_trigonometric_operator_families} \ref{enum:rescaling}, we obtain 
\begin{align*}
&\mathrm{Cov}(\tilde{u}_{\delta}^{\Delta}(t), \langle \tilde{u}(s), \beta^{(\delta)} \rangle_{L^{2}(\Lambda)}) \\&= \int_{0}^{t \land s}\langle S_{\vartheta,\delta}(\delta^{-1}(t-r))\Delta K,S_{\vartheta,\delta}(\delta^{-1}(s-r))\beta^{(\delta)} \rangle_{L^{2}(\Lambda_{\delta})}\mathrm{d}r. 
\qedhere
\end{align*}
\end{proof}
\begin{proof}[Proof of \Cref{result:asymptotics_observed_fisher_information}] \phantomsection\label{proof:obs_fisher}
Given \asskernel{} and \assini{}, we will show that the expectation and covariance of the observed Fisher information satisfies
\begin{align*}
\mathbb{E}[\delta^{2}I_{\delta}] &\rightarrow \frac{T^{2} }{4 \vartheta(0)}\Vert \nabla K \Vert_{L^{2}(\mathbb{R}^{d})}^{2}, \quad
\quad\mathrm{Var}(\delta^{2}I_{\delta}) \rightarrow 0, \quad \delta \rightarrow 0. 
\end{align*}
We begin by considering the case of zero initial conditions. 
\begin{step}[Zero initial conditions]
If we assume zero initial conditions, we need to show
\begin{align}
\label{eq:convergence_obs_fisher_zero_initial}
\mathbb{E}[\delta^{2}\tilde{I}_{\delta}] &\rightarrow \frac{T^{2} }{4 \vartheta(0)}\Vert \nabla K \Vert_{L^{2}(\mathbb{R}^{d})}^{2}, \quad
\quad\mathrm{Var}(\delta^{2}\tilde{I}_{\delta}) \rightarrow 0, \quad \delta \rightarrow 0. 
\end{align}
The representations \eqref{eq:rep_observed_fisher_information_expectation} and \eqref{eq:rep_observed_fisher_information_variance} in \Cref{result:variance_and_covariance_representations} are given by 
\begin{align*}
\delta^{2}\mathbb{E}[\tilde{I}_{\delta}] = \int_{0}^{T} \int_{0}^{t}\Vert S_{\vartheta, \delta}(\delta^{-1}r) \Delta K \Vert^{2}_{L^{2}(\Lambda_{\delta})} \mathrm{d}r \mathrm{d}t, \quad \delta > 0,
\end{align*}
and
\begin{align*}
&\mathrm{Var}(\delta^{2}\tilde{I}_{\delta}) \\&= 2\int_{0}^{T}\int_{0}^{T} \left(\int_{0}^{t \land s} \langle S_{\vartheta,\delta}(\delta^{-1}(t-r))\Delta K,S_{\vartheta,\delta}(\delta^{-1}(s-r))\Delta K \rangle_{L^{2}(\Lambda_{\delta})} \mathrm{d}r\right)^{2}\mathrm{d}s\mathrm{d}t.
\end{align*}
The limits \eqref{eq:emerging_equipartition_1} and \eqref{eq:emerging_orthogonality_1} in \Cref{result:asymptotics_emerging_energetic_expressions} \ref{enum:concrete_asymptotic_equipartition} and \ref{enum:concrete_asymptotic_orthogonality} show the pointwise convergence of the innermost integrands provided that $t \neq s$. Note that the diagonal $\{(t,s)\colon t=s \text{ for } t,s \in [0,T]\}$ is a set of Lebesgue measure zero, and we have almost sure convergence of the integrand. By \eqref{eq:energy_upperbound_1} in \Cref{result:uniform_upperbounds_operator_sine}, we have
\begin{align}
\label{eq:DCT_1}
&\sup_{0 < \delta \leq 1} \Vert S_{\vartheta,\delta}(\delta^{-1}r)\Delta K \Vert^{2}_{L^{2}(\Lambda_{\delta})} \leq \sup_{0 < \delta \leq 1} \Vert (-A_{\vartheta,\delta})^{-1/2} \Delta K \Vert^{2}_{L^{2}(\Lambda_{\delta})} < \infty,
\end{align}
and 
\begin{align}
&\sup_{0 < \delta \leq 1}|\langle S_{\vartheta,\delta}(\delta^{-1}(t-r))\Delta K,S_{\vartheta,\delta}(\delta^{-1}(s-r))\Delta K \rangle_{L^{2}(\Lambda_{\delta})}| \nonumber\\
&\leq \sup_{0 < \delta \leq 1}\Vert S_{\vartheta,\delta}(\delta^{-1}(s-r))\Delta K \Vert_{L^{2}(\Lambda_{\delta})}\Vert S_{\vartheta,\delta}(\delta^{-1}(t-r)) \Delta K\Vert_{L^{2}(\Lambda_{\delta})}\nonumber\\
\label{eq:DCT_2}
& \leq \left(\sup_{0 < \delta \leq 1} \Vert (-A_{\vartheta,\delta})^{-1/2} \Delta K \Vert^{2}_{L^{2}(\Lambda_{\delta})}\right)^{2}<\infty.
\end{align}
As both \eqref{eq:DCT_1} and \eqref{eq:DCT_2} are finite and independent of all time variables, the convergences in \eqref{eq:convergence_obs_fisher_zero_initial} follow using the dominated convergence theorem.
\end{step}
\begin{step}[Non-zero initial conditions]
The expectation of the observed Fisher information given by \eqref{eq:expectation_observed_fisher} in \Cref{result:dependence_of_observed_fisher_information_on_initial_conditions} satisfies
\begin{equation*}
\delta^{2}\mathbb{E}[I_{\delta}]=\delta^{2}\mathbb{E}[\tilde{I}_{\delta}] + \delta^{2}\Vert \mathcal{L}_{\delta}^{C}+ \mathcal{L}_{\delta}^{S} \Vert^{2}_{L^{2}([0,T])} \rightarrow \frac{T^{2} }{4 \vartheta(0)}\Vert \nabla K \Vert_{L^{2}(\mathbb{R}^{d})}^{2}, \quad \delta \rightarrow 0,
\end{equation*}
where the convergence follows immediately from the convergence of $\mathbb{E}[\delta^{2}\tilde{I}_{\delta}]$ in \eqref{eq:convergence_obs_fisher_zero_initial}  and with \eqref{eq:boundedness_initial_object} in \Cref{result:dependence_of_observed_fisher_information_on_initial_conditions}.  By \eqref{eq:observed_fisher_representation_with_initials} we obtain for the variance:
\begin{align*}
\mathrm{Var}(\delta^{2}I_{\delta})&=\mathrm{Var}(\delta^{2} \tilde{I}_{\delta}+  2\delta^{2} \langle \tilde{u}_{\delta}^{\Delta},\mathcal{L}_{\delta}^{C}+\mathcal{L}_{\delta}^{S} \rangle_{L^{2}([0,T])})\\
&= \mathrm{Var}(\delta^{2}\tilde{I}_{\delta}) + 2\mathrm{Cov}(\delta^{2}\tilde{I}_{\delta}, 2\delta^{2} \langle \tilde{u}_{\delta}^{\Delta},\mathcal{L}_{\delta}^{C}+\mathcal{L}_{\delta}^{S}\rangle_{L^{2}([0,T])}))\\&+\mathrm{Var}(2\delta^{2} \langle \tilde{u}_{\delta}^{\Delta},\mathcal{L}_{\delta}^{C} +\mathcal{L}_{\delta}^{S} \rangle_{L^{2}([0,T])})).
\end{align*}
By the first step, we already know that $\mathrm{Var}(\delta^{2}\tilde{I}_{\delta})$ converges to zero. Thus, as the covariance term can be bounded in terms of the variances, it suffices to observe
\begin{align*}
\mathrm{Var}(2\delta^{2}\langle \tilde{u}_\delta^{\Delta},\mathcal{L}_{\delta}^{C}+\mathcal{L}_{\delta}^{S} \rangle_{L^{2}([0,T])})&=\mathbb{E}\left(\left(2\delta^{2}\langle \tilde{u}_\delta^{\Delta},\mathcal{L}_{\delta}^{C}+\mathcal{L}_{\delta}^{S} \rangle_{L^{2}([0,T])}\right)^{2}\right)\\
&\leq 4\mathbb{E}[\delta^{2}\tilde{I}_{\delta}] \delta^{2} \Vert \mathcal{L}_{\delta}^{C} + \mathcal{L}_{\delta}^{S} \Vert_{L^{2}([0,T])}^{2} \rightarrow 0, \quad \delta \rightarrow 0,
\end{align*}
where we have used the Cauchy-Schwarz inequality, \eqref{eq:boundedness_initial_object} in \Cref{result:boundedness_initial_object} and \textit{Step 1}. \qedhere
\end{step}
\end{proof}
\begin{proof}[Proof of \Cref{result:remaining_bias}] 
\phantomsection\label{proof:bias} Throughout this proof, recall the notation for $\beta^{(\delta)}$ and $\beta^{(0)}$ defined in \eqref{eq:beta_delta}:
\begin{equation*}
\begin{aligned}
\beta^{(\delta)}(x) &=\delta^{-1}(A_{\vartheta,\delta}- \vartheta(0)\Delta)K(x), \quad \delta > 0,\\ 
\beta^{(0)}(x) &= \Delta( \langle \nabla\vartheta(0),x \rangle_{\mathbb{R}^{d}} K ) (x)- \langle \nabla\vartheta(0), \nabla K(x)\rangle_{\mathbb{R}^{d}},\quad x \in \mathbb{R}^{d}.
\end{aligned}
\end{equation*}
We wish to show under the Assumptions \asskernel{} and \assini{} that the remaining bias satisfies
\begin{equation*}
\delta^{-1}(I_{\delta})^{-1}R_{\delta} \xrightarrow{\mathbb{P}}  \frac{\langle \nabla K, \nabla\beta^{(0)}\rangle_{L^{2}(\mathbb{R}^{d})}}{\Vert \nabla K \Vert^{2}_{L^{2}(\mathbb{R}^{d})}}, \quad \delta \rightarrow 0.
\end{equation*}
\begin{step}[Zero initial conditions]
By \Cref{result:asymptotics_observed_fisher_information} we have $\tilde{I}_{\delta}/\mathbb{E}[\tilde{I}_{\delta}]\xrightarrow{\mathbb{P}} 1$. Applying the dominated convergence theorem as well as \eqref{eq:emerging_equipartition_2} and \eqref{eq:emerging_orthogonality_2} in \Cref{result:asymptotics_emerging_energetic_expressions} to \eqref{eq:expectation_representation_remaining_bias} and \eqref{eq:variance_representation_remaining_bias} in \Cref{result:variance_and_covariance_representations} respectively, yields the convergences
\begin{equation*}
\mathbb{E}[\delta \tilde{R}_{\delta}] \rightarrow \frac{T^{2}}{4\vartheta(0)}\langle \nabla K, \nabla \beta^{(0)} \rangle_{L^{2}(\mathbb{R}^{d})}, \quad \mathrm{Var}(\delta \tilde{R}_{\delta}) \rightarrow 0, \quad \delta \rightarrow 0.
\end{equation*}
Note that we have used that both \eqref{eq:DCT_1} and \eqref{eq:DCT_2} remain valid if $\Delta K$ is replaced by $\beta^{(\delta)}$ given \Cref{result:fractional_and_inverse_limits} \ref{enum_bounds:5} and the dominated convergence theorem is applicable. 
Thus, with Chebyshev's inequality, we also have $\tilde{R}_{\delta}/\mathbb{E}[\tilde{R}_{\delta}]\xrightarrow{\mathbb{P}}1$ as $\delta \rightarrow 0$. An application of Slutsky's lemma and the continuous mapping theorem shows
\begin{equation*}
\delta^{-1}(\tilde{I}_{\delta})^{-1}\tilde{R}_{\delta} = (\mathbb{E}[\tilde{I}_{\delta}]\tilde{I}_{\delta}^{-1}) (\delta^{2}\mathbb{E}[\tilde{I}_{\delta}])^{-1}\mathbb{E}[\delta \tilde{R}_{\delta}] \tilde{R}_{\delta}(\mathbb{E}[\tilde{R}_{\delta}])^{-1} \xrightarrow{\mathbb{P}} \frac{\langle \nabla K, (\beta^{(0)})'\rangle_{L^{2}(\mathbb{R}^{d})}}{\Vert \nabla K \Vert^{2}_{L^{2}(\mathbb{R}^{d})}}.
\end{equation*}
\end{step}
Next, we consider the situation with non-zero initial conditions.
\begin{step}[Decomposition and connection to localisation]
Observe that
\begin{align*}
\delta {R}_{\delta}&= \delta \int_{0}^{T}{u}_{\delta}^{\Delta}(t)\langle {u}(t), (A_{\vartheta}-\vartheta(0)\Delta)K_{\delta} \rangle_{L^{2}(\Lambda)}\mathrm{d}t\\
&= \int_{0}^{T}{u}_{\delta}^{\Delta}(t) \langle {u}(t),\delta^{-1}((A_{\vartheta}-\vartheta(0)\Delta )K)_{\delta} \rangle_{L^{2}(\Lambda)}\mathrm{d}t\\
&= \int_{0}^{T}{u}_{\delta}^{\Delta}(t) \langle {u}(t), \beta^{(\delta)}_{\delta} \rangle_{L^{2}(\Lambda)}\mathrm{d}t.
\end{align*}
Using \eqref{eq:initial_decomp_1} and \eqref{eq:initial_decomp_2}, we decompose $\delta {R}_{\delta} = \delta \tilde{R}_{\delta} + V_{1,\delta} + V_{2,\delta} + V_{3,\delta}$, where
\begin{align*}
V_{1,\delta} &\coloneqq \int_{0}^{T}(\locu(t)+\locv(t))\langle \tilde{u}(t),\beta_{\delta}^{(\delta)} \rangle_{L^{2}(\Lambda)} \mathrm{d}t\\
V_{2,\delta} & \coloneqq \int_{0}^{T}(\locu(t)+\locv(t)) \langle C_{\vartheta}(t)u_0 + S_{\vartheta}(t)v_0, \beta_{\delta}^{(\delta)} \rangle_{L^{2}(\Lambda)} \mathrm{d}t
\\
V_{3, \delta} &\coloneqq \int_{0}^{T} \tilde{u}_{\delta}^{\Delta}(t)\langle C_{\vartheta}(t)u_0 + S_{\vartheta}(t)v_0, \beta_{\delta}^{(\delta)} \rangle_{L^{2}(\Lambda)}\mathrm{d}t. 
\end{align*}
We immediately obtain the decomposition
\begin{equation*}
\delta^{-1}({I}_{\delta})^{-1}{R}_{\delta}= \delta^{-1}({I}_{\delta})^{-1} \tilde{R}_{\delta} + \delta^{-2} ({I}_{\delta})^{-1} (V_{1,\delta}+ V_{2, \delta} + V_{3, \delta}). 
\end{equation*}
In particular, by \Cref{result:asymptotics_observed_fisher_information}, we have
\begin{equation*}
\delta^{-1} ({I_{\delta}})^{-1}\tilde{R}_{\delta} \xrightarrow{\mathbb{P}}  \frac{\langle \nabla K, \nabla\beta^{(0)}\rangle_{L^{2}(\mathbb{R}^{d})}}{\Vert \nabla K \Vert^{2}_{L^{2}(\mathbb{R}^{d})}}, \quad \delta \rightarrow 0.
\end{equation*}
Thus, as $\delta^{-2}({I}_{\delta})^{-1} \in \mathcal{O}_{\mathbb{P}}(1)$ in view of \Cref{result:asymptotics_observed_fisher_information}, we have to show
\begin{equation*}
V_{1,\delta} + V_{2, \delta} + V_{3,\delta} \xrightarrow{\mathbb{P}} 0, \quad \delta \rightarrow 0.
\end{equation*}
\end{step}
\begin{step}[Convergence of $V_{3,\delta}$]
By the linearity of the integral, it suffices to show the convergence in probability for both individual terms in the sum
\begin{equation*}
V_{3,\delta}= \int_{0}^{T}\tilde{u}_{\delta}^{\Delta}(t)\langle C_{\vartheta}(t)u_0, \beta^{(\delta)}_{\delta} \rangle_{L^{2}(\Lambda)}\mathrm{d}t + \int_{0}^{T}\tilde{u}_{\delta}^{\Delta}(t)\langle S_{\vartheta}(t)v_0, \beta^{(\delta)}_{\delta} \rangle_{L^{2}(\Lambda)}\mathrm{d}t.
\end{equation*}
By applying the Cauchy-Schwarz inequality to the time integral, we have
\begin{equation}
\label{eq:cosine_and_sine_initial_beta_delta_parts}
\begin{aligned}
&\mathbb{E}\left[\left(\int_{0}^{T}\tilde{u}_{\delta}^{\Delta}(t)\langle C_{\vartheta}(t)u_0, \beta^{(\delta)}_{\delta} \rangle_{L^{2}(\Lambda)}\mathrm{d}t\right)^{2}\right] \\
&\leq \mathbb{E}[\delta^{2}\tilde{I}_{\delta}]\int_{0}^{T}\delta^{-2}\langle C_{\vartheta}(t)u_0, \beta^{(\delta)}_{\delta}\rangle^{2}_{L^{2}(\Lambda)}\mathrm{d}t
\end{aligned}
\end{equation}
and
\begin{equation*}
\begin{aligned}
&\mathbb{E}\left[\left(\int_{0}^{T}\tilde{u}_{\delta}^{\Delta}(t)\langle S_{\vartheta}(t)v_0, \beta^{(\delta)}_{\delta} \rangle_{L^{2}(\Lambda)}\mathrm{d}t\right)^{2}\right] \\
&\leq \mathbb{E}[\delta^{2}\tilde{I}_{\delta}]\int_{0}^{T}\delta^{-2}\langle S_{\vartheta}(t)v_0, \beta^{(\delta)}_{\delta}\rangle^{2}_{L^{2}(\Lambda)}\mathrm{d}t.
\end{aligned}
\end{equation*}
By the functional calculus, both the operator cosine and sine are self-adjoint on $L^{2}(\Lambda)$, and the scaling property \Cref{result:rescaling_trigonometric_operator_families} of the operator sine and cosine yields
\begin{align*}
\delta^{-1}\langle C_{\vartheta}(t)u_0, \beta_{\delta}^{(\delta)} \rangle_{L^{2}(\Lambda)} &= \delta^{-1}\langle u_0, C_{\vartheta}(t)\beta_{\delta}^{(\delta)} \rangle_{L^{2}(\Lambda)} \\
&= \delta^{-1}\langle u_0, (C_{\vartheta,\delta}(\delta^{-1}t)\beta^{(\delta)})_{\delta} \rangle_{L^{2}(\Lambda)}\\
&= \delta^{-1}\langle (u_0)_{\delta^{-1}}, C_{\vartheta,\delta}(\delta^{-1}t)\beta^{(\delta)} \rangle_{L^{2}(\Lambda_{\delta})},
\end{align*}
and
\begin{align*}
\delta^{-1}\langle S_{\vartheta}(t)v_0, \beta^{(\delta)}_{\delta} \rangle_{L^{2}(\Lambda)} &= \langle (v_0)_{\delta^{-1}}, S_{\vartheta,\delta}(\delta^{-1}t) \beta^{(\delta)} \rangle_{L^{2}(\Lambda_{\delta})}. 
\end{align*}
Choose $1/2 > \gamma > 3/4 - \alpha/2$ and observe by the fractional rescaling property, for instance in \citet*[Lemma 16]{altmeyerParameterEstimationSemilinear2023}, that
\begin{align*}
&\delta^{-1}\langle C_{\vartheta}(t)u_0, \beta_{\delta}^{(\delta)} \rangle_{L^{2}(\Lambda)} \\&= \delta^{-1} \langle (-A_{\vartheta,\delta})^{1-\gamma} (u_0)_{\delta^{-1}}, C_{\vartheta,\delta}{(\delta^{-1}t)} (-A_{\vartheta,\delta})^{-1+\gamma} \beta^{(\delta)}\rangle_{L^{2}(\Lambda_{\delta})}\\
&= \delta^{2(1-\gamma)-1} \langle ((-A_{\vartheta})^{1-\gamma} u_0)_{\delta^{-1}}, C_{\vartheta,\delta}{(\delta^{-1}t)} (-A_{\vartheta,\delta})^{-1+\gamma} \beta^{(\delta)}\rangle_{L^{2}(\Lambda_{\delta})},
\end{align*}
and
\begin{align*}
&\delta^{-1}\langle S_{\vartheta}(t)v_0, \beta^{(\delta)}_{\delta} \rangle_{L^{2}(\Lambda)} \\&= \langle (-A_{\vartheta,\delta})^{1/2}(v_0)_{\delta^{-1}}, S_{\vartheta,\delta}(\delta^{-1}t) (-A_{\vartheta,\delta})^{-1/2}\beta^{(\delta)} \rangle_{L^{2}(\Lambda_{\delta})}\\
&= \delta \langle ((-A_{\vartheta})^{1/2}v_0)_{\delta^{-1}}, S_{\vartheta,\delta}(\delta^{-1}t) (-A_{\vartheta,\delta})^{-1/2}\beta^{(\delta)}\rangle_{L^{2}(\Lambda_{\delta})}.
\end{align*}
Another application of the Cauchy-Schwarz inequality amounts to
\begin{equation}
\label{eq:operator_cosine_delta_trick}
\begin{aligned}
&|\delta^{-1}\langle C_{\vartheta}(t)u_0, \beta_{\delta}^{(\delta)} \rangle_{L^{2}(\Lambda)}|\\ &\leq \delta^{1-2\gamma} \Vert ((-A_{\vartheta})^{1-\gamma} u_0)_{\delta^{-1}} \Vert_{L^{2}(\Lambda_{\delta})} \Vert C_{\vartheta}(\delta^{-1}t) \Vert_{\mathcal{L}(L^{2}(\Lambda_{\delta}))} \Vert (-A_{\vartheta,\delta})^{-1+\gamma} \beta^{(\delta)} \Vert_{L^{2}(\Lambda_{\delta})},
\end{aligned}
\end{equation}
and
\begin{equation}
\label{eq:operator_sine_delta_trick}
\begin{aligned}
&|\delta^{-1}\langle S_{\vartheta}(t)v_0, \beta^{(\delta)}_{\delta} \rangle_{L^{2}(\Lambda)}|\\
& \leq \delta \Vert (-A_{\vartheta})^{1/2}v_0 \Vert_{L^{2}(\Lambda)} \Vert S_{\vartheta,\delta}(\delta^{-1}t) \Vert_{\mathcal{L}(L^{2}(\Lambda_{\delta}))} \Vert (-A_{\vartheta,\delta})^{-1/2}\beta^{(\delta)} \Vert_{L^{2}(\Lambda_{\delta})}. 
\end{aligned}
\end{equation}
All norms in \eqref{eq:operator_cosine_delta_trick} and \eqref{eq:operator_sine_delta_trick} remain bounded as $\delta \rightarrow 0$ by \Cref{result:fractional_and_inverse_limits} and the rescaling property. As $\gamma<\frac{1}{2}$, we notice that $1-2\gamma>0$. Thus, both terms \eqref{eq:operator_cosine_delta_trick} and \eqref{eq:operator_sine_delta_trick} and subsequently in \eqref{eq:cosine_and_sine_initial_beta_delta_parts} converge to zero.
\end{step}
\begin{step}[Convergence of $V_{2,\delta}$]
Using Cauchy-Schwarz, we obtain for the deterministic part:
\begin{equation}
\label{eq:V2_bound}
\begin{aligned}
&|V_{2,\delta}|^{2} \\ 
&= \left|\int_{0}^{T}\langle C_{\vartheta}(t)u_0 + S_{\vartheta}(t)v_0, \delta^{-2}(\Delta K)_{\delta} \rangle_{L^{2}(\Lambda)} \langle C_{\vartheta}(t)u_0 + S_{\vartheta}(t)v_0, \beta_{\delta}^{(\delta)} \rangle_{L^{2}(\Lambda)} \mathrm{d}t \right|^{2}\\
& \leq {\int_{0}^{T} \langle C_{\vartheta}(t)u_0 + S_{\vartheta}(t)v_0, \delta^{-1}(\Delta K)_{\delta} \rangle^{2}\mathrm{d}t}  {\int_{0}^{T} \langle C_{\vartheta}(t)u_0 + S_{\vartheta}(t)v_0, \delta^{-1}\beta_{\delta}^{(\delta)} \rangle^{2}\mathrm{d}t}. 
\end{aligned}
\end{equation}
By \eqref{eq:boundedness_initial_object} in \Cref{result:boundedness_initial_object} the first factor in  \eqref{eq:V2_bound} remains asymptotically bounded. Thus, the convergence follows exactly as for \eqref{eq:cosine_and_sine_initial_beta_delta_parts} from the upper bounds \eqref{eq:operator_sine_delta_trick} and \eqref{eq:operator_cosine_delta_trick}.
\end{step}

\begin{step}[Convergence of $V_{1,\delta}$]
Using Cauchy-Schwarz we obtain
\begin{align*}
&\mathbb{E}[|V_{1,\delta}|^{2}] = \mathrm{Var}(V_{1,\delta}) \\
&= \int_{0}^{T}\int_{0}^{T}\mathrm{Cov}\Big((\locu(t)+\locv(t))\langle \tilde{u}(t),\beta^{(\delta)}_{\delta} \rangle_{L^{2}(\Lambda)},\\
&\hspace{79pt}(\locu(s)+\locv(s))\langle \tilde{u}(s),\beta^{(\delta)}_{\delta} \rangle_{L^{2}(\Lambda)}\Big)\mathrm{d}s\mathrm{d}t\\
&\lesssim \left(\int_{0}^{T}\int_{0}^{T} |\mathrm{Cov}(\langle \tilde{u}(t),\beta^{(\delta)}_{\delta} \rangle_{L^{2}(\Lambda)}, \langle \tilde{u}(s),\beta^{(\delta)}_{\delta} \rangle_{L^{2}(\Lambda)})|\mathrm{d}s\mathrm{d}t\right)^{1/2} \Vert \locu + \locv \Vert^{2}_T.
\end{align*}
We may upper bound this covariance through
\begin{align}
&|\mathrm{Cov}(\langle \tilde{u}(t), \beta_{\delta}^{(\delta)} \rangle_{L^{2}(\Lambda)}, \langle \tilde{u}(s), \beta_{\delta}^{(\delta)} \rangle_{L^{2}(\Lambda)})| \nonumber\\
&\leq \int_{0}^{t \land s} | \langle S_{\vartheta}(t-r)\beta^{(\delta)}_{\delta}, S_{\vartheta}(s-r)\beta^{(\delta)}_{\delta}\rangle_{L^{2}(\Lambda)}| \mathrm{d}r \nonumber\\
\label{eq:covaraince_integral_1}
&= \delta^{2} \int_{0}^{t \land s}|\langle S_{\vartheta,\delta}(\delta^{-1}(t-r))\beta^{(\delta)}, S_{\vartheta,\delta}(\delta^{-1}(s-r))\beta^{(\delta)}\rangle_{L^{2}(\Lambda_{\delta})}|\mathrm{d}r.
\end{align}
Using \eqref{eq:energy_upperbound_2} in \Cref{result:uniform_upperbounds_operator_sine} and the Cauchy-Schwarz inequality, we obtain the upper bound
\begin{equation}
\label{eq:DCT_3}
\begin{aligned}
\sup_{0 < \delta \leq 1} |\langle S_{\vartheta,\delta}(\delta^{-1}(t-r))\beta^{(\delta)}&, S_{\vartheta,\delta}(\delta^{-1}(s-r))\beta^{(\delta)}\rangle_{L^{2}(\Lambda_{\delta})}| 
\\&\leq \left(\sup_{0 < \delta \leq 1} \Vert (-A_{\vartheta,\delta})^{-1/2}\beta^{(\delta)} \Vert_{L^{2}(\mathbb{R}^{d})}\right)^{2}<\infty.
\end{aligned}
\end{equation}
Using the upper bound \eqref{eq:DCT_3} the dominated convergence theorem and \eqref{eq:emerging_orthogonality_3} in \Cref{result:asymptotics_emerging_energetic_expressions}, the integral \eqref{eq:covaraince_integral_1} converges to zero. Consequently, the bound \Cref{result:boundedness_initial_object} shows $\mathrm{Var}(V_{1,\delta}) \rightarrow 0$ as $\delta \rightarrow 0$ since the other factor remains bounded as $\delta \rightarrow 0$ by \eqref{eq:boundedness_initial_object} in \Cref{result:boundedness_initial_object}.  \qedhere
\end{step}
\end{proof}
\end{appendix}
\bibliographystyle{plainnat}
\bibliography{zotero_22.04.2024}
\end{document}